\documentclass[english,11pt]{article}
\usepackage[german,french,english]{babel}
\usepackage[cp850]{inputenc}
\usepackage{latexsym,graphicx, fancybox}
\usepackage{graphicx} % Required for inserting images

\usepackage{amsmath, amssymb, amsfonts, amsthm} % Packages AMS pour les math├®matiques
\usepackage{color}
\usepackage{tikz}
\usepackage{latexsym}

\usepackage{mathrsfs}
\usepackage{bm}
\usepackage{stmaryrd}
\usepackage[round]{natbib}

\usepackage{etoolbox}
\font\tenmath=msbm10 scaled 1200
\font\sevenmath=msbm7 scaled 1200
\font\fivemath=msbm5 scaled 1200 
	\usepackage{amsthm}

	\newfam\mathfam \textfont\mathfam=\tenmath
	\scriptfont\mathfam=\sevenmath \scriptscriptfont\mathfam=\fivemath
	
	\def\R{{\mathbb R}}
	\def\N{{\mathbb N}}
	\def\E{{\mathbb E}}

	\def\P{{\mathbb P}}
	\def\Z{{\mathbb Z}}
	\def\T{{\mathbb T}}
	\def\Q{{\mathbb Q}}
	
	\def\dbF{\mathbb{F}}
	\def\dbP{\mathbb{P}}

	\def\cF{{\cal F}}

	\def\ges{\geqslant}
\def\bas#1{\begin{assumption}\label{#1}}
\def\eas{\end{assumption}}

\def\ben{\begin{enumerate}}
\def\een{\end{enumerate}}
\newtheorem{Theorem}{Theorem}[section]
\newtheorem{Definition}[Theorem]{Definition}
\newtheorem{Proposition}[Theorem]{Proposition}

\newtheorem{Lemma}[Theorem]{Lemma}
\newtheorem{Corollary}[Theorem]{Corollary}

\newtheorem{Example}[Theorem]{Example}

\newtheorem{assumption}{Assumption}[section]

\usepackage{thm-restate}
\usepackage{float}
\usepackage{hyperref}
\newfam\mathfam \textfont\mathfam=\tenmath
\scriptfont\mathfam=\sevenmath \scriptscriptfont\mathfam=\fivemath

\def \^#1{\if#1i{\accent"5E\i}\else{\accent"5E#1}\fi}

\begin{document}
\selectlanguage{english}

\title{\bf On Path-dependent Volterra Integral Equations: Strong Well-posedness and Stochastic Numerics.
}
\author{
	{\sc Emmanuel Gnabeyeu}\thanks{Laboratoire de Probabilit\'es, Statistique et Mod\'elisation, UMR~8001, \textit{Sorbonne Universit\'e and Universit\'e Paris Cit\'e}, 4 pl. Jussieu, F-75252 Paris Cedex 5, France.  E-mail: {\tt emmanuel.gnabeyeu\_mbiada@sorbonne-universite.fr}.} 
	\and
	{\sc Gilles Pag\`es}\thanks{Laboratoire de Probabilit\'es, Statistique et Mod\'elisation, UMR~8001, Sorbonne Universit\'e, case 158, 4, pl. Jussieu, F-75252 Paris Cedex 5, France. E-mail: {\tt gilles.pages@sorbonne-universite.fr}.}
}

\maketitle 
\renewcommand{\abstractname}{Abstract}
\begin{abstract}
	The aim of this paper is to provide a comprehensive analysis of the path-dependent Stochastic Volterra Integral Equations (SVIEs), in which both the drift and the
	diffusion coefficients are allowed to depend on the whole trajectory of the process up to the
	current time.
	We investigate the existence and uniqueness (aka the strong well-posedness) of solutions to such equations in the $L^p$ setting, \(p>0\), locally in time and their properties specifically their path regularity and flows.
	Then, we introduce a numerical approximation method based on an interpolated $K-$integrated Euler-Maruyama scheme to simulate numerically the process, and we prove the convergence, with an explicit rate, of this scheme towards the strong solution in the $L^p$ norm.
\end{abstract}



% \clearpage
%\tableofcontents

\section{Introduction}
\subsection{Literature review} 
The theory of stochastic Volterra integral equations (SVIEs), seen as a generalization of stochastic differential equations (SDEs), has its origins in the 1980s and has been widely developed since then. 
These equations have garnered significant attention within the mathematical community working on stochastic dynamical systems, particularly for their applications in population dynamics, biology, and physics(see e.g. ~\cite{mohammed1998}), as they provide a natural framework for modeling non-Markovian stochastic phenomena. Their study has also been notably motivated by problems in heat transfer physics~\citep{gripenberg1990}, and they have since been the subject of substantial mathematical investigation.
Early studies begin with the seminal and pioneering work of \cite{Berger1980a, Berger1980b}, where the following Volterra integral equation(SVIEs) driven by Brownian motion with  continuous coefficients was introduced:
\[
X_t = x_0(t) + \int_0^t b(t, s, X_s) \, ds + \int_0^t \theta(t, s, X_s) \, dW_s, \quad t \in [0, T].
\]
Here, \(x_0\) denotes the (deterministic) initial condition, and \((W_t)_{t \in [0,T]}\) is a standard Brownian motion. The so-called drift coefficient \(b\), and diffusion coefficient \(\theta\) are all deterministic, Borel-measurable functions. Under Lipschitz conditions on \(b\) and \(\theta\) with respect to the space variable, uniformly in the time variables \(s\) and \(t\), the authors proved that the equation admits a well-posed solution (strong existence and pathwise uniqueness) by adapting the classical tools from the well-known SDE theory.

\medskip
\noindent These initial results were subsequently generalized in various directions, such as allowing for anticipating integrands using Skorokhod integration and Malliavin calculus (see~\cite{Pardoux1990,Alos1997}). Additionally, \cite{ZhangXi2010} examined SVIEs in an infinite dimensional setting, Banach spaces with locally Lipschitz coefficients and singular kernels. Other studies focused on extensions that incorporated path-dependent coefficients such as  \cite{Kalinin2021} and 
% A slight extension beyond Lipschitz continuous coefficients, along with a framework for the theory of weak solutions, was introduced (see~\cite{abi2021weak,PromelScheffels2023a, PromelScheffels2023b}). In a more recent contribution, \cite{Wang2008} proved the existence and uniqueness of solutions to SVEs with non-Lipschitz coefficients, utilizing a condition analogous to that of \cite{Yamada1971} for stochastic differential equations. Additionally, \cite{ZhangXi2010} examined SVIEs in an infinite dimensional setting, Banach spaces with locally Lipschitz coefficients and singular kernels.
\cite{Protter1985} who generalized the existence and uniqueness results to stochastic Volterra equations (SVIEs) driven by right-continuous semimartingales of the form:
\[
X_t = H_t + \int_0^t f(t, s, X_u, u<s) \, dZ_s, \quad t \in [0, T],
\]
Here \((H_t)_{t \in [0, T]}\) is an adapted and right-continuous process, and \((Z_t)_{t \in [0, T]}\) is a right-continuous semi-martingale. The function \(f\) is assumed to satisfy suitable Lipschitz continuity conditions in the space variable and to be differentiable with respect to \(t\). Furthermore, it was shown that the solution \(X\) inherits the semimartingale property under the assumption that \(f\) is partially differentiable in \(t\).

\medskip
\noindent
In the late 1990s, attempts were made within the financial community to incorporate long-memory effects into continuous-time stochastic volatility models. This paradigm shift has been driven by the necessity to capture persistent dependencies observed in financial markets, particularly through \(H\)-fractional Brownian motion or Volterra equations with regular \(K(t,s) \)(see \cite{CoutinD2001, ComteR1998}). 
In the early 2000s, research shifted to Volterra equations with singular kernels kernel \(K(t,s)\) that blow up as \(s \to t\) (i.e. \(H < 1/2\)). This transition was fueled by the development of a new generation of stochastic volatility models, which forgo the Markovian framework to better align with empirical data. Notably, the study in~\cite{GatheralJR2018} revealed that volatility trajectories exhibit very low H\"older regularity, with \(H \approx 0.1\), underscoring the need for non-Markovian modeling approaches.
As a result, there has been a resurgence of interest in SVIEs within mathematical finance, particularly with the rise of rough volatility models, as highlighted in the work of \cite{ElEuchR2018,EGnabeyeuPR2025,Gnabeyeu2026a, Gnabeyeu2026b}. These models, which use the singular kernel, naturally capture this feature, as their paths have a H\"older continuity exponent \(H\). Singular kernel Volterra equations also arise as limiting dynamics in models of order books via nearly unstable Hawkes processes (see \cite{JaissonR2016, EGnabeyeuR2025}).
In both context, such processes are used to mimic Fractional Brownian motion-driven stochastic differential equations (SDEs). More specifically, within these frameworks, In particular, Volterra equations with fractional kernels \(K\) have emerged as a more tractable alternative to stochastic differential equations driven by genuine \(H\)-fractional Brownian motions, which would otherwise require the use of ``high-order" rough path theory or regularity structures.  

\medskip
\noindent
In contrast, more recent volatility models, including those developed in the context of Brownian SDEs (e.g., \cite{Guyon2023}), still rely on Brownian motion but often consider path-dependent structures, similar to those explored in the seminal work by \cite{RogersWilliamsII} and later extended by \cite{Ruan2020}, with inspiration from \cite{Protter1985} in the context of Volterra equations.
By considering a given initial random function \(x_0\) that evolves deterministically for \( t > 0 \); that is, for \( t > 0 \), the value \( x_0(t) \) is \(\mathcal{F}_0\)-measurable, where \(\mathcal{F}_0\) denotes the history of the process up to time \( t = 0 \), a rather general form of the stochastic version of the path-dependent Volterra equation on \([0,T]\) in  $\mathbb{H} = \mathbb{R}^d$ for any \(T>0\) takes the following form:
\begin{equation}\label{eq:pathVolterra}
	\begin{cases}
		X_t = x_0(t) + \int_0^t K_1(t,s) b(s, X_\cdot^s) \, ds + \int_0^t K_2(t,s) \sigma(s, X_\cdot^s) \, dW_s,\\
		x_0 : (\Omega, \mathcal{F}, \mathbb{P}) \to (\mathbb{H}, \mathcal{B}(\mathbb{H})),  \quad x_0\perp\!\!\!\perp W.
	\end{cases}
\end{equation}
where $K_i : [0,T]^2 \to \mathbb{R}_+, i \in \{1, 2\}$ are the deterministic kernels weighting the drift and diffusion coefficients and modeling the memory or hereditary structure of the system. The terms $X_{\cdot}^s$ or equivalently $X_{\cdot \wedge s}$ denote, the path of the process up to time $s$. The process $(W_t)_{t \geq 0}$ is an $\mathbb{R}^{d^\prime}$-valued Brownian motion independent of $x_0$, both defined on a probability space $(\Omega, \mathcal{A}, P)$ and $\mathcal{F}_t \supset \mathcal{F}_{t, x_0, W}$ a filtration satisfying the usual conditions.
Such equations~\eqref{eq:pathVolterra} can be seen as a generalization of classical stochastic differential equations to non-Markovian and path-dependent settings, where the evolution of the process depends not only on the current state but also on its entire past trajectory. They naturally arise in the modeling of systems corrupted by noise with memory effects and irregular behaviour, including in mathematical finance, physics, and biology.

\subsection{Our contribution}
The aim of this paper is to study, within a general framework, stochastic Volterra
equations~\eqref{eq:pathVolterra} with path-dependent coefficients. We focus in
particular on the case where the kernel may be either singular or non-singular.
Let $\varphi$ be a deterministic function and let $X_0$ be an $\mathbb{H}$-valued
random variable. Throughout this paper, we allow the initial input
$x_0(t) := X_0 \varphi(t)$ to be an $\mathcal{F}_0$-measurable random function.
This choice does not introduce any additional technical difficulty.
A more subtle difficulty arises from existing results on the integrability and
pathwise regularity of solutions to Volterra equations~\eqref{eq:pathVolterra}.
Indeed, following \cite{Ruan2020}, existence and uniqueness of a pathwise continuous
and integrable solution are established only under rather stringent assumptions.

\begin{itemize}
	\item[(i)] Our first objective is therefore to establish existence and uniqueness in Theorem~\ref{Thm:pathExistenceUniquenes} of a strong
	solution to~\eqref{eq:pathVolterra} on appropriate functional space when $X_0 \in L^p(\mathbb{P})$ for any \(p\) large enough.
\end{itemize}
	In contrast to the techniques developed for standard Stochastic Volterra equations in \cite{ZhangXi2010} and further refined in \cite{JouPag22}, our approach in Theorem~\ref{Thm:pathExistenceUniquenes} is based
	on a fixed-point argument in appropriate Bochner spaces under additional appropriate integrability and regularity conditions on the
	kernels.
\begin{itemize}
	\item[(ii)]  We also analyze various qualitative properties of these solutions in Theorem~\ref{prop:pathkernelvolt}, in particular, the existence of continuous modification, their path regularity, supremum-norm estimates and finite $L^p$ moments, all of which are recovered \emph{a posteriori} via Kolmogorov's continuity criterion and, then, the flow properties of the solution.
\end{itemize}
% Pathwise regularity of the solution, as well as supremum-norm estimates, are then recovered \emph{a posteriori} via Kolmogorov's continuity criterion.
In Theorem~\ref{prop:pathkernelvolt2}, we show that our results can be extended to the more natural
setting in which $X_0$ belong to $L^p(\mathbb{P})$ for some $p \in (0,+\infty)$, by proving a representation formula for the solutions to~\eqref{eq:pathVolterra} as
functionals of the Brownian path and the initial value.
\begin{itemize}
	\item[(iii)] Finally, we introduce and analyze an interpolated $K$-integrated Euler scheme, a key tool that gives a clear road-map for the numerical simulation of those equations~\eqref{eq:pathVolterra}. We prove its convergence towards the unique
	strong solution, and derive the rate of this convergence in \(L^p\).
\end{itemize}
 Specifically, we prove in Theorem~\eqref{thm:Eulercvgce2} that,
if $X_0 \in L^p(\mathbb{P})$ for some $p>0$, then the scheme converges in
$L^p(\mathbb{P})$ for the sup-norm, under the same assumptions (up to time regularity of the coefficients of the equation) as those ensuring
strong existence and uniqueness for the path-dependent Volterra process.
 
\subsection{Plan of the paper and Notations}
\noindent {\sc \textbf{Organization of this paper.}} 
The remainder of the paper is organized as follows. Section~\ref{sect-kernelVolterra} presents the setting and the main results, including the existence and uniqueness theorem for path-dependent stochastic Volterra integral equations with Lipschitz continuous drift and diffusion coefficients in the {\em Bochner space}. Section~\ref{sec:num_approx} then focuses on the convergence of a discretization scheme to this process, namely the {\em interpolated $K$-integrated Euler-Maruyama} scheme, as the time step tends to \(0\). Rates of convergence are presented, together with insights into the performance of the scheme in path-dependent settings.
In Section~\ref{sec:num_approx}, we numerically illustrate sample paths of a scaled process and verify the convergence rate when \(K\) is the fractional kernel. Finally, in Section~\ref{sect:ProofMainResult}, we provide the proof of the main results of this paper.

\medskip
\noindent {\sc \textbf{Notations.}} 

\smallskip 
\noindent $\bullet$ Consider for \(d,d^\prime\geq1\), \( \mathbb{H} = \mathbb{R}^d \), \( \bar{\mathbb{H}} = \mathbb{R}^{d^\prime} \) and  \( \tilde{\mathbb{H}}=\mathbb{H} \times \bar{\mathbb{H}} = \mathbb{R}^{d \times d^\prime} \) three Euclidean spaces.
More specifically, in what follows, we shall consider  \( \tilde{\mathbb{H}} := \mathbb{M}_{d,d^\prime}(\R)\) the space of matrices of size \(d\times d^\prime\), equipped with the operator norm \(\|\cdot\|_{\tilde{\mathbb{H}}} =|||\cdot|||\). 

\noindent $\bullet$ ${\cal C}([0,T], \mathbb H) (\text{resp.} \quad {\mathcal C_0}([0,T], \mathbb H))$ denotes the set of continuous functions(resp. null at 0)  from $[0,T]$ to $\mathbb H \quad (\R^d, etc.)$ and ${\cal B}or({\cal C}_{\mathbb H})$ (resp. ${\cal B}or({\cal C}_{\mathbb H, 0})$) denotes its  Borel $\sigma$-field induced by the $\sup$-norm topology. For \( \bar{\mathbb{H}} := \mathbb{R}^{d^\prime} \), we denote by
\(\mathbb{Q}_{W}\) the \(d^\prime\)-dimensional wiener measure on \({\cal C}([0,T], \bar{\mathbb{H}})\).

%\smallskip
%\noindent $\bullet$ ${\rm Leb}_d$ denotes the Lebesgue measure on $(\R^d, {\cal B}or(\R^d))$. 

\smallskip 
\noindent $\bullet$ For $p\in(0,+\infty)$, $L_{\mathbb H}^p(\P)$ or simply $L^p(\P)$ denote the set of  $\mathbb H$-valued random vectors $X$  defined on a probability space $(\Omega, {\cal A}, \P)$ such that $\|\|X\|_{\mathbb H}\|_p:=(\E[\|X\|_{\mathbb H}^p])^{1/p}<+\infty$.

\smallskip 
\noindent $\bullet$ For a random variable/vector/process $X$, we denote by $\mathcal{L}(X)$ or $[X]$ its law or distribution. 

\smallskip 
\noindent $\bullet$ $X\perp \! \! \!\perp Y$  stands for independence of random variables, vectors or processes $X$ and $Y$.  

%\smallskip 
%\noindent $\bullet$ For a function $f: E\to \R$, $\displaystyle \|f\|_{\sup}= \sup_{x\in E}|f(x)|$. 
%

%\noindent $\bullet$ For a measurable function $f: (E, \mathcal{F}, \mu) \to \mathbb{R}$, we define its essential supremum as 
%\[
%\mathrm{ess} \sup f = \inf \left\{ M \in \mathbb{R}: \mu \left( \left\{ x \in E : |f(x)| \geq M \right\} \right) = 0 \right\}.
%\]

\smallskip 
\noindent $\bullet$ Let \(\mathcal{M}:=\mathcal{M}(\mathbb{R}_+,\mathbb{R}_+)\) denote the space of all $(\R_+, {\cal B}or(\R_+))$-measurable functions \(\mu\) on \(\mathbb{R}_+\) such that the restriction \(\mu|_{[0, T]}\), for any \(T > 0\), is a \(\mathbb{R}_+\)-valued finite measure (i.e. the restriction $\mu|_{[0,T]}$ with $T > 0$ is well-defined) and $\mu$ is non-atomic, i.e., $\mu(\{t\}) = 0$ holds for each $t \ge 0$. For \(\mu \in \mathcal{M}\) and a compact set \(E \subset \mathbb{R}_+\), we define the total variation of \(\mu\) on \(E\) by:

\centerline{$|\mu|(E) := \sup \left\{ \sum_{j=1}^N |\mu(E_j)| : \{E_j\}_{j=1}^N \text{ is a finite measurable partition of } E \right\}.$}
\smallskip 
\noindent $\bullet$ Denote:
\begin{align*}
	&\mathbb{T} = [0, T] \subset \mathbb{R}_+, \qquad \mathbb{T}^2 = [0, T] \times [0, T] \subset \mathbb{R}_+^2, \qquad
	&\mathbb{T}^2_- = \big\{(t, s) \in \mathbb{R}_+^2 \mid 0 \leq s \leq t \leq T \big\}.
\end{align*}

Here, ``$_-$" indicates the left neighborhood of \( t \). The state space \( {\cal C}([0, T]; \mathbb{H}) \), is equipped with the uniform or supremum norm $\| \cdot \|_{\sup} :=\| \cdot \|_T$:
We also define the truncated supremum norm:
\begin{equation} \label{uniform}
	\textit{For all}\quad  t \in \mathbb{T}	\quad \|\mathbf{x}\|_t = \sup_{s \in [0, t]} \| \mathbf{x}_s \|_{\mathbb{H}} = \sup\Big\{ \| \mathbf{x}_s \|_{\mathbb{H}} : s \leq t  \Big\}, \qquad \forall \, \mathbf{x} \in {\cal C}([0, T]; \mathbb{H}).
\end{equation}

The results of this paper, still hold  in a straightforward manner for any 2-smooth \footnote{In the sense that
	there exists a constant $C_{\mathbb{H}}\ > 2$ such that for all x, y $\in \mathbb{H},
	\|x + y\|^2_{\mathbb{H}} + \|x - y\|^2_{\mathbb{H}} \leq
	2\|x\|^2_{\mathbb{H}} + C_{\mathbb{H}}\|y\|^2_{\mathbb{H}}$. } separable Banach space \(\mathbb{H}\) endowed with a norm \(\|\cdot\|_{\mathbb{H}}\) (say more generally any separable Hilbert space \(\mathbb{H}\), or \( L^p(E, \mu, \mathbb{H}) \) where \((E,\mathcal{E},\mu)\) is a measure space etc.).) 
\section{Path-dependent Stochastic Volterra Equations with explicit kernels.}\label{sect-kernelVolterra}
%\label{sect-preliminary}

Let $T>0$ be the time horizon and $(\Omega,\cF,\dbF,\dbP)$ be a complete probability space with a family of right-continuous filterations and satisfying the usual conditions \footnote{The \emph{usual hypotheses} means namely, that \(\mathcal{F}\) is \(\mathbb{P}\)-complete, that $\mathcal{F}_0$ contains all $\mathbb{P}$-null sets and that the filtration is right-continuous, i.e., $\mathcal{F}_t=\mathcal{F}_{t+} \equiv \cap_{s>t} \mathcal{F}_s$.}, where  $\Omega$ is the sample space, $\dbP$ a probability measure, $W$ a standard $\bar{\mathbb{H}}$-valued Brownian motion under $\dbP$, and $\dbF:= \dbF^W\equiv\{\cF_t\}_{t\ges0}$ the natural filtration generated by $W$ augmented by all the $\dbP$-null sets in $\cF$.
Throughout this paper, without special declarations, all expectations $\mathbb{E}$
are taken with respect to the probability measure $\mathbb{P}$. 

Let \(\mathbb{X}\) be a space of functions from \(\mathbb{T}\) to \(\mathbb{H}\) to be specified latter.
For \( x = (x_u)_{u \in \mathbb{T}} \in \mathbb{X} \), and for a fixed \( s \in \mathbb{T} \), we define the processes \( x^s_{\cdot} := x_{\cdot \wedge s} = (x_{u \wedge s})_{u \in \mathbb{T}} \) by:
\[
x_{u \wedge s} := 
\begin{cases}
	x_u & \text{if} \quad u \in [0, s] \\
	x_s & \text{if} \quad u \in (s, T]
\end{cases} = x_u \mathbf{1}_{[0,s]}(u)
+ x_s \mathbf{1}_{(s,T]}(u).
\]
It is clear that \( x_{\cdot \wedge s} \in {\cal C}(\mathbb{T}, \mathbb H) \) whenever \(x\in {\cal C}(\mathbb{T}, \mathbb H) \). Explicitly \( x^s \), is the stopped extension of the process \( x \) at time \( s \), i.e the trajectory of \( x \) is followed until \( s \), and then stopped-extended by freezing the value at \( x_s \) over \( (s,T] \).

\medskip

%\subsection{Results on path-dependent Volterra processes with explicit kernels.}

%We consider the below structure of our path-dependent stochastic Volterra equation with distinct kernels $K_1$ and $K_2$.
 Let us fix an \(\mathbb{H}\) -valued, \(\mathbb{F}\)-adapted continuous process \((x_0(t))_{t \geq 0}\).
We will study the following path-dependent Volterra integral equation with distinct kernels $K_1$ and $K_2$ and initial condition or data \((s, x_0, \mathbb{F})\): \(X_t = x_0(t)\) for all \(t \in [0, s]\), and,
\begin{equation}\label{eq:pthvolterra}
	X_t = x_0(t) + \int_{s}^{t} K_1(t,r) b(r, X^r_\cdot) \, dr + \int_{s}^{t} K_2(t,r) \sigma(r, X^r_\cdot) \, dW_r, \quad \text{for all } \mathbb{T}\ni t \geq s, \ \mathbb{P}\text{-a.s.} 
\end{equation}
where the stopped extension \(X^s_\cdot\) of \((X_t)_{t\in\mathbb{T}}\) is given by \((X^s_\cdot)_{u \in \mathbb{T}} = \left( X_u \mathbf{1}_{[0,s]}(u) + X_s \mathbf{1}_{(s,T]}(u) \right)_{u \in \mathbb{T}}\) and $b:\mathbb{T}\times\mathbb{X} \to \mathbb{H}$, $\sigma:\mathbb{T}\times\mathbb{X} \to \tilde{\mathbb{H}}$, $K_i : \mathbb{T}^2_{-} \to \mathbb{R}_+, i \in \{1, 2\}$ for $i \in \{1, 2\}$ are measurable and non-anticipative functions\footnote{In the sense that \(\forall \, x \in \mathbb{X},\, b(s,x) = b(s,x^s_\cdot)\). An example of coefficient could be: $b(s, x^s_{\cdot}) = g(s)\int_0^s \Phi (x(u)) ds$ for a bounded and continuous functions \(g\) and $\Phi$ }, $(W_t)_{t \in \mathbb{T}}$ the standard $\bar{\mathbb{H}}$-valued $(\mathcal{F}_t)_{t \in \mathbb{T}}-$Brownian motion independent of the $\mathbb{H}$-valued random vector $X_0 \in \mathcal{F}_0$.
% and $(\mathcal{F}_t)_{t \in [0,T]}$ be a filtration such that $X_0 \in \mathcal{F}_0$ and $W$ is an $\bar{\mathbb{H}}$-valued ${\cal F}_t$-Brownian motion. 

%\medskip
%Note that, when $d^\prime:=\text{dim}(\bar{\mathbb{H}})\le \text{dim}(\mathbb{H})=:d$,  we can always assume that $d^\prime=d$ by replacing {\em mutatis mutandis} $\sigma$  by $\sigma\circ \iota_{d^\prime,d}$ where $\R^{d^\prime}\!\ni u\mapsto \iota_{d^\prime,d} (u)=$ {\tiny $\begin{pmatrix}u\\ 0_{\R^{d-d^\prime}}\end{pmatrix}$} $\!\in \R^d$ and the $d^\prime$-dimensional   standard  Brownian motion $(W_t)_{t\in[0,T]}$ by  a standard $d$-dimensional Brownian motion (still denoted $W$). It is clear that the resulting equation has the same solutions (if any) as the one in~\eqref{eq:pthvolterra}.
% Such equations can be written in the form
%\begin{align}
%	X_t &= x_0(t) + \int_0^t K_1(t,s)\,
%	b\!\left(s,\left( X_u \mathbf{1}_{[0,s]}(u)
%	+ X_s \mathbf{1}_{(s,T]}(u) \right)_{u \in [0,T]} \right) ds \nonumber \\
%	&\quad + \int_0^t K_2(t,s)\,
%	\sigma\!\left(s,\left( X_u \mathbf{1}_{[0,s]}(u)
%	+ X_s \mathbf{1}_{(s,T]}(u) \right)_{u \in [0,T]} \right) dW_s,
%	\qquad t \in \mathbb{T}.
%\end{align}
\medskip
% The integral \(\int_0^t K_1(t,s) b(s, X_\cdot^s) \, ds\) is defined as a Lebesgue integral, while \(\int_0^t K_2(t,s) \sigma(s, X_\cdot) \, dW_s\) is interpreted as an It\^o integral.
Due to the kernels' dependence on their first argument, the solution is generally neither a semi-martingale nor a Markov process, as the path-dependent structure introduce memory effects into the process dynamics, thereby eliminating the Markovian property and limiting the use of standard stochastic calculus tools.

% destroys the Markovian property.
%A key distinction between regular diffusion processes and Volterra processes, from a technical perspective, lies in the coefficients' dependance on their first argument (or generally the presence of kernels), which introduce memory effects into the process dynamics, thereby eliminating the Markov property and limiting the use of standard stochastic calculus tools.
\begin{Example}[The case of a convolutive kernel]\label{eq:pathSDE}
	Assume that for all $(t,s)\in \mathbb{T}_-^2$ we have \(K_1(t,s)=K_2(t,s)=K(t-s)\) where \(K\) is a convolutive kernel  (see, e.g.,~\cite[Definition 2.1]{EGnabeyeu2025}). If the path-dependent drift and diffusion coefficients $b$ and $\sigma$ in Equation~\eqref{eq:pthvolterra} are of the form
	\begin{equation}\label{eq:b_sigma_general}
		b(t,x) = \mu\,\!\big(t, (\widetilde K \star x)_t\big)
		= \mu\,\!\Big(t, \int_0^t \widetilde K(t-s)\,x(s)\,ds \Big),\;\text{and}\; \sigma(t,x) = \theta\,\!\big(t, (\widetilde K \star x)_t\big),
		\; t \in \mathbb{T}, \ x \in {\cal C}(\mathbb{T}, \mathbb R)
	\end{equation}
	for Borel measurable functions $\mu, \theta:\mathbb{T}\times \mathbb{R} \to \mathbb{R}$ that are Lipschitz in $x$ uniformly in $t$, and uniformly bounded at $x=0$, and where $\widetilde K$ is the $\rho$-pseudo-inverse co-kernel of $K$, that is \(K\star \widetilde K =\mathbf{1}\) (see \cite[Definition 3.1]{Bonesini2023Volterra}), then Equation~\eqref{eq:pthvolterra} corresponds to the framework studied in  \cite{Bonesini2023Volterra}.  In this case, the process $\xi_t := (\widetilde K \star X)_t$ satisfies the diffusion SDE
		\begin{equation}\label{eq:SDE}
		\xi_t = x_0(t) e^{\rho t}( \mathbf{1} \star \widetilde K)_t + \int_{0}^{t} e^{\rho r}\mu(r, e^{-\rho r}\xi_r) \, dr + \int_{0}^{t} e^{\rho r}\theta(r, e^{-\rho r}\xi_r)  \, dW_r, \quad \text{for all } t\in \mathbb{T}, \quad \xi_0 =\mathbf{0}. 
	\end{equation}
	Equation~\eqref{eq:SDE} is especially appealing from a numerical perspective, as it recasts the original non-Markovian dynamics into a tractable diffusion framework amenable to standard discretization schemes.
	
%	This equation~\eqref{eq:SDE} drew some attention for numerical purposes....
\end{Example}
This part focuses on some results concerning the existence of such stochastic integrals equations in a more general setting~\eqref{eq:pthvolterra}  and the uniqueness of their solutions.

\medskip
For the space \(\mathcal{M}:=\mathcal{M}(\mathbb{R}_+,\mathbb{R}_+)\) of all $(\R_+, {\cal B}or(\R_+))$ locally finite positive Borel measures on \(\mathbb{R}_+\), we further assume that every $\mu \in \mathcal{M}$ satisfies the following
\emph{uniform non-atomicity condition}: for every $T>0$,

\begin{equation}
	\label{eq:uniform-non-atomicity}
	\lim_{\eta \to 0}
	\sup_{t \in \mathbb{T}}
	\mu\big([(t-\eta)^+,t]\big)
	= 0 .
\end{equation}
In particular,  $\mu$ is non-atomic, i.e., $\mu(\{t\}) = 0$ holds for each $t \ge 0$. Throughout this paper, we let \(\mu\in\mathcal{M}\) satisfying~\eqref{eq:uniform-non-atomicity}, and shall assume the following for the existence and uniqueness of a solution to Equation~\eqref{eq:pthvolterra}.
\begin{assumption}\label{assump:kernelVolterra}
	We consider the following:
	\begin{enumerate}
		\item[(i)] Assume that the kernels $K_i$, $i=1,2$, satisfy an integrability assumption, namely either condition
		\begin{equation}\label{eq:contKtilde}
			(\widehat {\cal K}^{cont}_{\widehat \theta})\;\;\exists\,\widehat\kappa< +\infty,\;\forall\bar{\delta}\!\in (0,T],\; \widehat \eta(\delta) := \max_{i=1,2} \sup_{t\in \mathbb{T}} \left[\int_{(t-\bar{\delta})^+}^t \hskip-0,25cm K_i\big(t,u\big)^i du\right]^{1/i}\le \widehat \kappa \,\bar{\delta}^{\,\widehat \theta} \; %\text{for some} \; \widehat\theta\in (0,1]
		\end{equation}
		is satisfied for some $\widehat\theta\in (0,1]$ or
		\begin{equation}\label{eq:KisL^2} 
			\big({\cal K }^{int}_{\beta}\big)\hskip1,5cm % \exists\, \beta >1  \mbox{ such that } 
			\sup_{t\in \mathbb{T}}\int_0^t\left(K_1(t,s)^{{\frac{2\beta}{\beta+1}}}
			+K_2(t,s)^{2\beta}\right)ds<+\infty \quad \text{for some} \quad \beta>1.
		\end{equation}
		
		\item[(ii)] Assume that the kernels $K_i$, $i=1,2$, satisfy the continuity assumption i.e. for  some $\theta\in (0,1]$,
		\begin{equation}\label{eq:contK}
			({\cal K}^{cont}_{\theta}) \;\;  
			\exists\, \kappa< +\infty,\;\forall \,\bar{\delta}{\in (0, T)},
			\;\eta(\bar{\delta}):= \max_{i=1,2}\sup_{t\in \mathbb{T}} \left[\int_0^t |K_i(\big(t+\delta)\wedge T,s\big)-K_i(t,s)|^ids \right]^{\frac 1i} \le  \kappa\,\bar{\delta}^{\theta}
		\end{equation}
		% for  some $\theta\in (0,1]$.  
		\item[(iii)] Assume that the functions ${\cal C}(\mathbb{T}, \mathbb H) \ni x\mapsto b(t,x)$ and ${\cal C}(\mathbb{T}, \mathbb H)\ni x\mapsto \sigma(t,x)$ are Lipschitz with Lipschitz  coefficient uniform in $t\in\mathbb{T}$, i.e. \(\exists \, C_{b, T}, \, C_{\sigma, T} \;\) such that \(\forall\, t \in \mathbb{T}, \; \forall\, x,\, y \in {\cal C}(\mathbb{T}, \mathbb H)\):
		\[ \|b(t,x)-b(t,y)\|_{\mathbb{H}}\, , \,  \|\sigma(t,x)-\sigma(t,y)\|_{\tilde{\mathbb{H}}}  \leq (C_{b, T} \vee C_{\sigma, T} ) \Big( \int_{0}^t \| x(s)-y(s)\|_{\mathbb{H}}^p \mu (ds) \Big)^{\frac1p}.
		\]
		
		\item[(iv)] Finally, assume that \( X_0 \in L^p(\mathbb{P}) \) for some \( p \in (0, +\infty) \), 
		the initial input process $t \to x_0(t):= X_0 \varphi(t)$ is continuous and $(\cF_t)$-adapted. Moreover,  for some $\delta > 0$, for any $p>0$ and $T>0, \quad$\\
		$ \int_{0}^T \mathbb{E}\left[ \| x_0(t) \|_{\mathbb{H}}^p \right] \mu(dt) <+\infty, $  $\mathbb{E}\|x_0(t^\prime)-x_0(t)\|_{\mathbb{H}}^p\leqslant C_{T,p}\Big( 1 + \int_{0}^{T}\mathbb{E}\left[ \| x_0(s) \|_{\mathbb{H}}^p \right] \mu (ds) \Big)|t'-t|^{\delta p}$
	\end{enumerate}
\end{assumption}
This setting corresponds, in our framework, to the case of regular and singular kernels and Lipschitz continuous coefficients.

\smallskip
\noindent {\bf Remarks and Comments}.
	\noindent  $\bullet$ Assumptions $({\cal K}^{int}_{\beta})$ and  $(\widehat {\cal K}^{cont}_{\widehat \theta})$ are very close  in the singular case, e.g. when $K_i(t,s)= \varphi_i(t-s)$, $0\le s<t \le T$ with $\varphi_i$ decreasing from $+\infty$ at  $0$.
	\smallskip
	If $({\cal K}^{int}_{\beta})$ is in force, then $(\widehat {\cal K}^{cont}_{\widehat \theta})$ holds with $\widehat \theta =  \frac{\beta-1}{2\beta}$ since H\"older's inequality implies, for $\delta \!\in \mathbb{T}$,  
	\[
	\int_{(t-\delta)^+}^tK_1(t,u) du \le \Big(\int_0^t K_1(t,u)^{\frac{2\beta}{\beta+1}}du \Big)^{\frac{\beta+1}{2\beta}}  (\delta\wedge t)^{\frac{\beta-1}{2\beta}}\le C_{1,\beta,T}\,\delta^{\frac{\beta-1}{2\beta}} 
	\]
	and 
	\[
	\Big(\int_{(t-\delta)^+}^tK_2(t,u)^2 du\Big)^{1/2}  \le \Big(\int_0^t K_2(t,u)^{2\beta}du \Big)^{\frac{1}{2\beta}}  (\delta\wedge t)^{\frac{\beta-1}{2\beta}}\le C_{2,\beta,T}\,\delta^{\frac{\beta-1}{2\beta}} .
	\]
	As a result, in this scenario, we can always assume that $\widehat \theta \ge \frac{\beta - 1}{2\beta}$.\\
	$\bullet$ However, there is no general converse. That said, if for every \( t \in \mathbb{T} \), the functions \( s \mapsto K_i(t,s) \) for \( i = 1,2 \) are non-decreasing (which is true for typical singular kernels), then the condition \( (\widehat {\cal K}^{cont}_{\widehat \theta}) \) implies \( ({\cal K}^{int}_{\beta}) \) for any \( \beta \) such that \( \frac{\beta-1}{2\beta} < \widehat \theta \). In this case, \( \widehat \theta \) represents the supremum (not attained) of \( \frac{\beta-1}{2\beta} \) across all values of \( \beta \) for which \( ({\cal K}^{int}_{\beta}) \) holds.
	This follows directly from the monotonicity of the kernels \( K_i \). Specifically, under the assumption \( (\widehat {\cal K}^{cont}_{\widehat \theta}) \), for \( 0 \leq s < t \leq T \), we have the following bounds:
	
	\centerline{$
	(t-s)^{1/i} K_i(t,s) \leq \left( \int_s^t K_i(t,u)^i \, du \right)^{1/i} \leq \widehat \kappa (t-s)^{\widehat \theta},
	$}
	
	\noindent which implies \(	K_i(t,s) \leq \widehat \kappa (t-s)^{\widehat \theta - 1/i}.\)
%	Thus, we obtain the inequalities: 
%%	\centerline{$K_1(t,s) \leq C (t-s)^{\widehat \theta - 1} \quad \text{and} \quad K_2(t,s) \leq C (t-s)^{\widehat \theta - \frac{1}{2}}.$}
	Consequently, the condition \( ({\cal K}^{int}_{\beta}) \) holds as long as the following inequalities are satisfied $\footnote{The reader must refer to the theory of improper integrals with 0 as a singularity.}$:
	
	\centerline{$
	\frac{2\beta}{\beta+1} (\widehat \theta - 1) > -1 \quad \text{and} \quad 2\beta (\widehat \theta - \frac{1}{2}) > -1, \quad \text{which simplifies to } \quad \frac{\beta - 1}{2\beta} < \widehat \theta.
	$}
	
%	which simplifies to \(\frac{\beta - 1}{2\beta} < \widehat \theta.\)
\begin{Example}[Examples of kernels]\label{Ex:Kernels}
	Following Example~\ref{eq:pathSDE}, we now present several examples of kernels \(K\) of convolutive type.
	\begin{enumerate}
		\item {\em Fractional integration kernel}. Let \(K(t) = K_{\alpha}(t) = \frac{u^{\alpha-1}}{\Gamma(\alpha)} \mbox{\bf 1}_{\R_+}(t)\) with \(\alpha>0.\)
	%	\begin{equation}\label{eq:frackernel}
	%		K(t) = K_{\alpha}(t) = \frac{u^{\alpha-1}}{\Gamma(\alpha)} \mbox{\bf 1}_{\R_+}(t),  \quad \alpha>0.
	%	\end{equation}
		 This  family of  kernels is associated with the fractional integrations of order $\alpha >0$ and encompasses, in particular, the {\em trivial (Markovian) kernel} $K(t) = \mbox{\bf 1}_{\R_+}(t)$. 
		
		\item {\em Exponential-Fractional integration kernel}. Let $K(t) = K_{\alpha,\rho}(t) = e^{-\rho t} \frac{u^{\alpha-1}}{\Gamma(\alpha)} \mbox{\bf 1}_{\R_+}(t)$ with $\alpha, \rho>0 >0$.
		This family extends the fractional integration kernels by introducing an exponential damping factor and naturally includes the {\em exponential kernel} defined by $K(t)= e^{-t} $. 
	\end{enumerate}
	 One checks that these kernels satisfy conditions~\eqref{eq:contKtilde}--\eqref{eq:KisL^2} and~\eqref{eq:contK}, that is, $(\widehat {\cal K}^{cont}_{\widehat \theta})$, $\big({\cal K }^{int}_{\beta}\big)$ and $({\cal K}^{cont}_{\theta})$, for $\alpha >1/2$ (with $\beta \in \bigl(1,\frac{1}{2(\alpha-1)^{-}}\bigr)$ and $\theta = \widehat \theta = \min\bigl( \alpha-\frac12,\; 1\bigr)$, see, e.g., \cite[Example~2.1]{RiTaYa2020} or~\cite{JouPag22} along many others).
\end{Example}

\begin{Lemma}\label{lm:growth}
	Under Assumption ~\ref{assump:kernelVolterra}, the coefficient functions \(b\) and \(\sigma\) have a linear growth in \( x \in \mathbb{X}\) in the sense that there exists a constant \(C_{b,\sigma,T}\) such that for every \(t \in [0, T]\) and \( x \, \in \mathbb{X}\),
	\begin{equation}
		\|b(t,x)\|_{\mathbb{H}}\,  \vee  \,  \|\sigma(t,x)\|_{\tilde{\mathbb{H}}}  \leq  C_{b, \sigma, T}  \Big(1 + \Big( \int_{0}^t \| x(s)\|_{\mathbb{H}}^p\, \mu (ds) \Big)^{\frac1p}\Big).
	\end{equation}
\end{Lemma}

\smallskip
\noindent {\bf Proof:}
%[of Lemma \ref{lm:growth}]
	 Let \( x \, \in \mathbb{X}\), and let \(\mathbf{0} \in \mathbb{X}\) be such that \(\mathbf{0}(t) = 0\) for all \(t \in [0, T]\). Then
	 {\small
	\[
	\|b(t,x)\|_{\mathbb{H}} - \|b(t,\mathbf{0})\|_{\mathbb{H}} 
	\leq \|b(t,x)-b(t,\mathbf{0})\|_{\mathbb{H}} \leq  C_{b, T} \, \left( \int_{0}^t \| x(s)-\mathbf{0}(s)\|_{\mathbb{H}}^p \mu (ds) \right)^{\frac1p}= C_{b, T} \, \left( \int_{0}^t \| x(s)\|_{\mathbb{H}}^p \mu (ds) \right)^{\frac1p}.
	\]
	}
	\noindent Consequently, \(	\|b(t,x)\|_{\mathbb{H}} \leq 
	\left( \sup_{t \in \mathbb{T}} \|b(t,\mathbf{0})\|_{\mathbb{H}}  \vee C_{b, T} \right)
	\left( \left( \int_{0}^t \| x(s)\|_{\mathbb{H}}^p \mu (ds) \right)^{\frac1p} + 1 \right).\)
	Similarly, we have \(		\|\sigma(t,x)\|_{\tilde{\mathbb{H}}} \leq 
	\left( \sup_{t \in \mathbb{T}} \|\sigma(t,\mathbf{0})\|_{\tilde{\mathbb{H}}}  \vee C_{\sigma, T} \right)
	\left( \left( \int_{0}^t \| x(s)\|_{\mathbb{H}}^p \mu (ds) \right)^{\frac1p} + 1 \right).\)
	so that one can take
	\[
	C_{b,\sigma,T} := \sup_{t \in \mathbb{T}} \|b(t,\mathbf{0})\|_{\mathbb{H}} \vee \sup_{t \in \mathbb{T}}\|\sigma(t,\mathbf{0})\|_{\tilde{\mathbb{H}}} \vee C_{b, T} \vee C_{\sigma, T}
	\]
	to conclude. \hfill$\Box$
\vspace{-.2cm}
\subsection{Existence and pathwise uniqueness of the solution.}
As a first preliminary, we will establish   the following lemma, whose proof is postponed to Appendix~\ref{app:B1}.  It provides a control in $L^p(\P)$ of the increments of  general Lebesgue or  stochastic integrals involving the kernels $K_i$.
\begin{Lemma}\label{lem:bounds}
	Let \(T>0\), $(H_t)_{t\in[0,T]}$ (resp. $(\widetilde H_t)_{t\in[0,T]}$) be an $({\cal F}_t)_{t\in[0,T]}$-progressively measurable process
	having values in $\mathbb{H}$ (resp. $\tilde{\mathbb{H}}$) such that $\displaystyle \sup_{t\in [0,T]}\|\|H_t\|_{\mathbb{H}}\|_p+\sup_{t\in [0,T]}\|\|\widetilde{H}_u\|^2_{\tilde{\mathbb{H}}}\|_p<+\infty$ for some $p\ge 2$. Assume that the kernels $K_i$ satisfy $(\widehat {\cal K}^{int}_{T,\widehat \theta})$ for some $\widehat \theta\in (0,1]$ and $({\cal K}^{cont}_{\theta})$ for some $\theta\in (0,1]$
	Let $0\le a<b\le T$. 	Assume that the integral \(\int_a^{t\wedge b} K_1(t,s) H_s \, ds\) is defined as a Lebesgue integral, while \(\int_a^{t\wedge b} K_2(t,s) \widetilde{H}_s \, dW_s\) is interpreted as an It\^o integral.
	Then,  under Assumption \ref{assump:kernelVolterra}
	both processes \([a,T] \ni t\longmapsto \int_{a}^{t\wedge b} K_1(t,s)H_sds\; \mbox{ and }\; [a,T] \ni t \longmapsto \int_{a}^{t\wedge b} K_2(t,s)\widetilde H_sdW_s\)
	are $\theta\wedge\widehat\theta$-H\" older from $[a,T]$ to $L^p(\P)$. To be more precise, there exists a real constant $\kappa, \kappa^\prime= C_{p,T,K_1,K_2}>0$ such that setting
	\( |b_u| =  \|H_u\|_{\mathbb{H}}\) and
	\( |a_u| = \|\widetilde{H}_u\|^2_{\tilde{\mathbb{H}}} \), we have for \(0\leq s\leq t\leq T\)
	\begin{align*}
	 &\, \mathbb{E}\Big[\Big\|\int_{a}^{t\wedge b} K_2(t,u)\widetilde{H}_u dW_u - \int_{a}^{s\wedge b} K_2(s,u)\widetilde{H}_u dW_u\Big\|_{\tilde{\mathbb{H}}}^p \Big] \leq \kappa^\prime_{p}\, \Big(\sup_{u \in [0,T]} \mathbb{E}\left[ |a_u|^{\frac{p}{2}}  \right] \Big)  (t-s)^{p(\theta\wedge \theta^*)} \\
	&\,	\mathbb{E}\Big[\Big\|\int_{a}^{t\wedge b} K_1(t,u) H_u du - \int_{a}^{s\wedge b} K_1(s,u) H_u du \Big\|_{\mathbb{H}}^p \Big] \leq \kappa_{p} \,\Big(\sup_{u \in [0,T]}\,\mathbb{E}\Big[ |b_u|^{p} \Big]\Big)   (t-s)^{p(\theta\wedge \theta^*)}. 
	\end{align*}
	\begin{equation}\label{eq:def_theta_star}
	\text{ Here, $\theta^*$ is given by:} \; \theta^* = \begin{cases}
		\frac{\beta-1}{2\beta}, & \text{if the kernels } K_i \text{ satisfy } ({\cal K}^{int}_{\beta}) \text{ for some } \beta>1, \\
		\widehat{\theta}, & \text{if } (\widehat {\cal K}^{cont}_{\widehat \theta}) \text{ holds}.
	\end{cases}
   \end{equation}
\end{Lemma}
\noindent {\bf Remark:}
The classical Burkholder-Davis-Gundy (BDG) inequality does not directly extend to processes of the form \(( \int_0^t K(t,s) \tilde{H}_s \, dW_s )_{t \in [0,T]}\), which, in general, lack the
local martingale property.\footnote{As an illustrative example, consider \( N_t := \int_0^t K(t) \, dW_s = K(t) W_t \), which is not a local martingale, since \( \mathbb{E}[N_t \mid \mathcal{F}_s] \neq N_s \), violating the local martingale property. More broadly, the local martingale property does not hold for kernels where the dependence on \( s \) and \( t \) cannot be separated.}
To resolve this issue, we introduce the auxiliary process \( M_u := \int_0^u K(t,s) \tilde{H}_s \, dW_s \) for \( u \in [0,t] \), where \( t \) is fixed. It can be verified that \( M_u \) is a local martingale, as it represents a standard stochastic integral with terminal value \( M_t = \int_0^t K(t,s)\tilde{H}_s \, dW_s \). Consequently, we can estimate:
{\small
	\begin{equation}\label{eq:bdg_trick}
		\Big\Vert \| \int_0^t K(t,u)\tilde H_u \, dW_u \|_{\mathbb{H}} \Big\Vert_p \le\Big\Vert \sup_{s \in [0,t]} \| \int_0^s K(t,u)\tilde H_u \, dW_u \|_{\mathbb{H}} \Big\Vert_p \le C_{p}^{BDG} \Big[ \int_0^t K(t,u)^2 \Big\| \|\tilde{H}_u\|_{\tilde{\mathbb{H}}} \Big\|_p^2 \, du \Big]^{\tfrac12}.
	\end{equation}	
}
This trick enables us to bypass the non-local martingale issue by focusing on the local martingale \( M_u \) over the interval \( [0,t] \).
However, this does not provide a pathwise estimate. The norm \(\|X\|_T\) is recovered a posteriori
through Kolmogorov's continuity theorem as demonstrated, for example, in \cite[Lemma 3.4]{ZhangXi2010}, which summarises the Kolmogorov continuity theorem approach.\\
	
The main difficulty for the proof of a wellposedness theorem in this context, is that, we need to estimate the path
norm but the classical BDG inequality cannot be applied to path-dependent SVIEs, as they are
not semi-martingale in general. To overcome this issue, we shall work on {\em Bochner space}.

%\subsubsection{Existence, uniqueness of the solution.}
\medskip
While the norms employed in our approach are more intricate than those used in earlier works, the underlying strategy of the proof is inspired by the foundational ideas of ~\cite{Sznitman1991}, themselves rooted in the seminal contributions of ~\cite{Dobrushin1970}. By constructing an appropriate trajectorial framework and tailoring norms that depend on carefully selected parameters, we are able to establish a classical fixed-point scheme.

Suppose that $p \geq 2$. For $T \geq 0$, set \(\mathbb{X}\), the state space of all \(\mu-\) integrable functions on \(\mathbb{T}\) and consider the $\mathbb{X}$-valued random variables\footnote{i.e. all processes $Y$ on $\mathbb{T}$ that satisfy
	\(
	\lim_{s \to t} \mathbb{E}[|Y_t - Y_s|^p] = 0 \quad \forall\, t \in \mathbb{T}.
	\)} $Y = (Y_t)_{t \in \mathbb{T}}$ having an $L^p$-norm \(\|Y\|_{p,T} \) to be defined later.	
Let \(\mu \in \mathcal{M}(\R_+,\R_+)\) be a finite measure  on \(\mathbb{T}\) satisfying~\eqref{eq:uniform-non-atomicity} such that \( (\mathbb{T}, \mathcal{B}(\mathbb{T}), \mu)\) is a complete measured space,  where we denote by $\mathcal{B}(\mathbb{T})$ the Borel $\sigma$-algebra on $\mathbb{T}$.
Denote by $\mathcal{H}_{p,T} \equiv L^p_{\mathbb{X}}\Big(\Omega \times \mathbb{T}, \mathcal{F} \otimes \mathcal{B}(\mathbb{T}), \mathbb{P} \otimes \mu\Big)$ the \emph{Bochner space} of such (modulo identification of versions) jointly measurable processes $Y : \Omega \times \mathbb{T} \to \mathbb{H}$, for which
\[
 \mathbb{E}\left[ \int_0^T \|Y_t(\omega)\|_{\mathbb{H}}^p \, \mu(dt) \right] < \infty,
\]
so that the \emph{Bochner space} $\mathcal{H}_{p,T}$ is usually equipped with the norm \(\|\cdot\|_{p,T}\) defined by:
\[
\|Y\|_{p,T} := \left(\int_{0}^T \mathbb{E}\big[\| Y_s\|_{\mathbb{H}}^p \big]\mu (ds)\right)^{\frac1p} < +\infty.
\]
It is known (Riesz-Fisher Theorem\footnote{Let \((E,\mathcal{T},\nu)\) be a measured space. For every $1\le p<\infty$,  \(L^p(\nu) := L^p(E,\mathcal{T},\nu)\) is a Banach space (complete normed vector space).} ) that $(\mathcal{H}_{p,T},\|\cdot\|_{p,T})$ is a Banach space for every $1\le p<\infty$ (and indeed, a Hilbert space when $p=2$).  
This follows from the completeness of $L^p$ spaces with values in Banach spaces (see e.g. standard results on Bochner integral and the Pettis measurability theorem).

Our framework is the measurable product space \(\Big(\Omega \times \mathbb{T}, \mathcal{F} \otimes \mathcal{B}(\mathbb{T}), \mathbb{P} \otimes \mu\Big)\). Thus, \(Y_\cdot(\cdot) \equiv ((\omega,t) \to Y_t(\omega))\) is a measurable function from \(\Omega \times \mathbb{T}\) to \(\mathbb{H}\). Any stochastic process $Y=(Y_t)_{t\in\mathbb{T}}$ with values in the Banach space $(\mathbb{H},\|\cdot\|_{\mathbb{H}})$ is thus said to be \emph{jointly measurable} if the mapping \((\omega,t)\longmapsto Y_t(\omega)\)
is measurable w.r.t.\ $\mathcal{F}\otimes\mathcal{B}(\mathbb{T})$.
% Let us equip the complete probability space \((\Omega, \mathcal{F}, \mathbb{P})\) with the complete filtration \(\mathbb{F} = (\mathcal{F}_t)_{t \geq 0}\), supporting a \(\mathbb{F}-\) Brownian motion \(W\) (i.e. a Brownian motion adapted to $\mathbb{F}$) with \(x_0\) a random continuous function, \(\mathcal{F}_0-\) measurable, so that \((\Omega, \mathcal{F}, \mathbb{P}, \mathbb{F} = (\mathcal{F}_t)_{t \geq 0})\) is a filtered probability space and thus $(\Omega, \mathcal{F}, \mathbb{P}, (\mathcal{F}_t)_{0\leq t \leq T},x_0, W)$ is our fixed working framework or set-up.

We recall that, the complete probability space \((\Omega, \mathcal{F}, \mathbb{P})\) is equipped with the complete filtration \(\mathbb{F} = (\mathcal{F}_t)_{t \geq 0}\), supporting a \(\mathbb{F}-\) Brownian motion \(W\) (i.e. a Brownian motion adapted to $\mathbb{F}$) with \(x_0\) being a random continuous function, \(\mathcal{F}_0-\) measurable. Hence, \((\Omega, \mathcal{F}, \mathbb{P}, \mathbb{F} = (\mathcal{F}_t)_{t \geq 0})\) is a filtered probability space and $(\Omega, \mathcal{F}, \mathbb{P}, (\mathcal{F}_t)_{0\leq t \leq T},x_0, W)$ is our fixed working framework or set-up.

we assume that $b:\Omega\times\mathbb{T}\times\mathbb{X} \to \mathbb{H} \in \mathcal{P}_m\otimes \mathcal{B}(\mathbb{X})/ \mathcal{B}(\mathbb{H}) $ and  $\sigma:\Omega\times\mathbb{T}\times\mathbb{X} \to \tilde{\mathbb{H}}\in \mathcal{P}_m\otimes \mathcal{B}(\mathbb{X})/ \mathcal{B}(\tilde{\mathbb{H}})$
where \( \mathcal{P}_m\) denotes the progressively measurable
\(\sigma-\)field on \(\Omega\times\mathbb{T}\) ( the mapping \((\omega,t)\longmapsto Y_t(\omega)\)
is measurable on \([0,t]\) w.r.t. $\mathcal{F}_t\otimes\mathcal{B}([0,t])$)
so that the integral \(\int_0^t K_1(t,s) b(s, X_\cdot^s) \, ds\) is well-defined as a Lebesgue integral and \(\int_0^t K_2(t,s) \sigma(s, X_\cdot) \, dW_s\) is interpretated as an It\^o integral.

 We may also assume that every element of $(\mathcal{H}_{p,T},\|\cdot\|_{p,T})$  is \emph{predictable} i.e. it is measurable with respect to the predictable $\sigma$-algebra on $\Omega\times\mathbb{T}$ ($\sigma$-algebra generated by all left-continuous adapted processes), denoted \(\mathcal{P}\subset\mathcal{P}_m\subset \mathcal{F}\otimes\mathcal{B}(\mathbb{T})\). In fact,
 by \citet[Proposition~3.21]{PeszatZabczyk2007}, every element $X \in \mathcal{H}_{p,T}$ admits a predictable representative, again denoted by $X$. We always work with such representatives.
 % \textcolor{blue}{Le fait que X soit previsible, implique qu'il est adapté par rapport a \cal F, L^p c'est une classe d'equivalence, on dois choisir un representant particulier... Le X previsible donneé par PeszatZabczyk2007}
We now define the functional space still denoted \(\mathcal{H}_{p,T}\)
\[
\mathcal{H}_{p,T} := \left\{ Y \in L^p_{\mathbb{X}}\Big(\Omega \times \mathbb{T}, \mathcal{F} \otimes \mathcal{B}(\mathbb{T}), \mathbb{P} \otimes \mu\Big) \text{ such that } Y \text{ is } (\mathcal{F}_t)_{t \in \mathbb{T}} \text{-adapted} \right\}.
\]
and adopt it as our working space for the remainder of this article.
By \emph{adapted} , we mean that $Y_t$ is $\mathcal{F}_t$-measurable for every $t\in\mathbb{T}$.

% and it is \emph{predictable} if it is measurable with respect to the predictable $\sigma$-algebra $\mathcal{P}$ on $\Omega\times[0,T]$ (the $\sigma$-algebra generated by left-continuous adapted processes).

We first prove existence and uniqueness of a solution to~\eqref{eq:pathVolterra} in $\mathcal{H}_{p,T}$ under appropriate assumptions. To this end, consider the family of norms $\|\cdot\|_{p,c,T}$ below on $\mathcal{H}_{p,T}$ suited to our fixed-point analysis:

\begin{equation} \label{eq:weighted_norm}
	\forall \; c > 0, \quad \|Y\|_{p,c,T} := \sup_{t \leq T} \Big[e^{-c t} \Big(\int_{0}^t \mathbb{E}\big[\| Y_s\|_{\mathbb{H}}^p \big]\mu (ds)\Big)^{\frac1p}\Big].
\end{equation}
It is clear that this norm is chosen so that it satisfies $\| Y \|_{p,c,T} \to 0$ as $c \to \infty$ for any fixed process \(Y\) such that \(\|Y\|_{p,c,T}<\infty\) for all \(c>0\).
It is obvious that $\|\cdot\|_{p,c,T}$ and $\|\cdot\|_{p,T}$ are strongly equivalent since
\begin{equation} \label{eq:norm_equiv}
	\forall Y \in \mathcal{H}_{p,T}, \quad e^{-cT} \|Y\|_p \leq \|Y\|_{p,c,T} \leq \|Y\|_p.
\end{equation}
so that one readily checks that  $\mathcal{H}_{p,T}$ endowed with the norm $\|\cdot\|_{p,c,T}$ is a complete space i.e. $(\mathcal{H}_{p,T}, \|\cdot\|_{p,c,T})$ is a Banach space.
\begin{Definition}
	A strong solution on \(\mathbb{X}\) of the path-dependent Volterra equation ~\eqref{eq:pthvolterra}, with initial condition \((x_0, \mathbb{F})\), on \([0, T]\), is an \(\mathbb{H}\)-valued \(\mathbb{F}\)-adapted continuous process \(X = (X_t)_{t \geq 0}\) such that for some \(p>0\)  \(\Big(\int_{0}^T \mathbb{E}\big[\| X_s\|_{\mathbb{H}}^p \big]\mu (ds)\Big)^{\frac1p} =:\| X \|_{p,T}  < \infty \; \)% \(\mathbb{E}\left[\sup_{t \in [0, T]} \|X_t\|_{\mathbb{H}}^p\right] < \infty\) 
	and ~\eqref{eq:pthvolterra} holds true.
\end{Definition}
The idea of our proof follows from~\cite{Feyel1987}'s approach, which is the well-known approach based
on Banach's fixed point theorem originally developed for the existence and uniqueness of a strong solution of a Brownian SDE
(see among others \cite[Section 7]{Bouleau1988}). The proof thus uses a contraction mapping principle which is close to that of \cite[Theorem 2.3]{MarinelliPrevotRockner2010}.
The core of the argument is to show that % for any fixed \(s\in\mathbb{T}\), 
the map \(\mathcal{T}_c : \mathcal{H}_{p,T} \rightarrow \mathcal{H}_{p,T} , \, \quad \text{defined by}\)
%\[
%\mathcal{T}_c : \mathcal{H}_{p,T} \rightarrow \mathcal{H}_{p,T} , \, \quad \text{defined by}
%\]
\[
\forall X \in \mathcal{H}_{p,T} 
\quad \mathcal{T}_c(X) =
\Big(x_0(t) + \int_{0}^{t } K_1(t,r) b(r, X^r_\cdot) \, dr + \int_{0}^{t} K_2(t,r) \sigma(r, X^r_\cdot) \, dW_r
\Big)_{t \in \mathbb{T}}.
\]
i.e. for $X \in \mathcal{H}_{p,T}$, \(	\mathcal{T}_c(X) := (\mathcal{T}_c(X)_t)_{0 \leq t \leq T}\) with  
\[
(\mathcal{T}_c(X))_t := x_0(t) + \int_{0}^{t} K_1(t,r) b(r, X^r_\cdot) \, dr + \int_{0}^{t} K_2(t,r) \sigma(r, X^r_\cdot) \, dW_r.
\]
is Lipschitz continuous with Lipschitz coefficient strictly less than one when $c$ is sufficiently large.
This contraction property enables the application of Banach's fixed-point theorem, establishing the existence and unique-ness of solutions to equation~\eqref{eq:pthvolterra}.
The application $\mathcal{T}_c$ has the following property.
 
\begin{Proposition}[Well-posedness and Lipschitzianity of $\mathcal{T}_c$]\label{prop:pathExistenceUniquenes}
	The following two claims hold:\\
	\medskip
	\noindent \textit{(i)} The function $\mathcal{T}_c$ is well-defined, and there holds \(\forall p\geq 2 \;\):
	{\small
	\begin{equation}\label{eq:boundnorm}
			\| \mathcal{T}_c(X) \|_{p,c,T} 
			\leq\; \| x_0 \|_{p,c,T} 
			+ C_{b,\sigma,T} \big(1 + C_p^{\text{BDG}}\big) \sup_{t\in [0,T]} \Big( 
			e^{-ct}\big( \int_0^t u^{p\widehat{\theta}}\, \mu(du) \big)^{\!\frac1p} 
			+ \big( \int_0^t u^{p\widehat{\theta}} e^{-pc(t-u)}
			\| X^u_\cdot \|_{p,c,u}\, \mu(du) \big)^{\!\frac1p}
			\Big).
	\end{equation}
	}
	\noindent \textit{(ii)} Under Assumption \ref{assump:kernelVolterra}, $\mathcal{T}_c$ is Lipschitz continuous in the sense that: for any 
	\(X\) and \(Y\) in $\mathcal{H}_{p,T}$,
	\[ \forall p\geq 2, \;
	\| \mathcal{T}_c( X )- \mathcal{T}_c( Y) \|_{p,c,T}^p 
	\leq K^\prime\cdot \sup_{t\in [0,T]} \left( \int_0^t u^{p\widehat{\theta}} e^{-pc(t-u)}\, \mu(du) \right)^\frac1p\cdot\| X - Y \|_{p,c,T},
	\]
	for some constant $K^\prime$ independent of $c$, but depending only on $p$, $T$ and the Lipschitz constants of $b$ and $\sigma$. 
\end{Proposition}
It follows from Proposition \ref{prop:pathExistenceUniquenes} that, for sufficiently large \( c \), the operator \( \mathcal{T}_c \) becomes a contraction on \( (\mathcal{H}_{p,T}, \|\cdot\|_{p,c,T}) \). This directly implies the existence and uniqueness of a strong solution of the Path-dependent Volterra equation, which unifies all the existing well-posedness results.
\begin{Theorem}[Well-posedness results: Strong Existence and uniqueness]\label{Thm:pathExistenceUniquenes} 
	 Let $T > 0$ and $p \geq 2$.
	Under Assumption \ref{assump:kernelVolterra}, the path-dependent Volterra equation defined in ~\eqref{eq:pthvolterra} has a unique strong solution.
	i.e. there exists (up to a $\mathbb{P}$-indistinguishability) a unique $\mathbb{H}$-valued pathwise continuous  $(\mathcal{F}_t)$-adapted process $X = (X_t)_{t \in [0,T]}$ satisfying \eqref{eq:pthvolterra} $\mathbb{P}$-a.s. for all $t \geqslant 0$. %, and the following estimate holds true: 
	Moreover, the following estimate holds true: % it holds that
	\begin{equation}\label{eq:bound_X}  \Big(\int_{0}^T \mathbb{E}\big[\| X_s\|_{\mathbb{H}}^p \big]\mu (ds)\Big)^{\frac1p} =:\| X \|_{p,T} \leq C \Big(1 + \Big( \int_{0}^T \mathbb{E}\left[ \| x_0(t) \|_{\mathbb{H}}^p \right] \mu(dt)\Big)^\frac1p\Big) <\infty
	\end{equation}
%	\[
%	\mathbb{E}\left[\sup_{t \in [0, T]} \left|X_t^{s, \xi, \alpha}\right|^p \right] < \infty.
%	\]
\end{Theorem}
\noindent {\bf Remark:}
	In the case of regular kernels, the solution is a \textit{semimartingale}\footnote{A right-continuous, adapted process \(Z\) is called a \emph{semimartingale} if it admits a decomposition of the form \(Z = M + A\), where \(M\) is a local martingale and \(A\) is a right-continuous, adapted process with paths of finite variation on compacts time intervals.}. The proof is then less involved (see also \cite{Protter1985}), following arguments similar to those used for classical SVIEs. Specifically, for all \(0 \leq s < t < T\), \(\omega \in \Omega\), and a bounded, \(\mathcal{F}_s\)-integrable function \(H = H(t, s, \omega) : \mathbb{T}^2_- \times \Omega \to \tilde{\mathbb{H}}\), satisfying   \(\int_0^t \left( \int_0^s \left| \partial_u H(u, s) \right|^2 du \right)^{1/2} ds < \infty \quad \text{a.s.},\)
	we have that \(Y_t := \int_0^t H(t, s) \, dW_s\) is an \((\mathcal{F}_t)_{t \in [0,T]}\)-semimartingale. By applying the fundamental theorem of calculus and stochastic Fubini's theorem \cite[ Theorem 2.6]{Walsh1986}, we can represent \(Y_t\) as:
	%we can write the representation or decomposition 
	\[
	Y_t = \int_0^t H(s, s) \, dW_s + \int_0^t \Big( \int_0^s \partial_s H(s, u) \, dW_u \Big) ds, \quad \text{where} \quad \partial_s H(s, u) := \Big. \frac{\partial}{\partial t} H(t, u) \Big|_{t = s}.
	\]

	Furthermore, for any stopping time \( T \leq a \), combining generalized Minkowski inequality, H\"older's inequality and the $L^p-BDG$ (Burkholder-Davis-Gundy) inequality yields, with \(C_p := 2^{p-1} (C_p^{BDG})^p\):
	\[
	\mathbb{E} \Big[ \big( \sup_{t \leq T} \| Y_t \|_{\mathbb{H}} \big)^p \Big] 
	\leq C_p \Big( \mathbb{E} \Big[ \Big( \int_0^T \| H(s,s) \|_{\tilde{\mathbb{H}}}^2 ds \Big)^{\frac{p}{2}} \Big]
	+ T^{p-1} \int_0^T \mathbb{E} \Big[ \Big( \int_0^T \| \partial_s H(s,u) \|_{\tilde{\mathbb{H}}}^2 du \Big)^{\frac{p}{2}} \Big] ds \Big)
	\]
	Consequently, we can directly prove the existence and uniqueness of a solution in this setting, using the supremum norm, via the Picard iteration method.
 \subsection{Pathwise Regularity, Moments Control and Maximal Inequality}\label{sect-kernelVolterraflow}
  % \subsubsection{Regularity and Maximal Inequality}
   In this section we provide the proof of the H\"older continuity and the maximal inequality of the solutions $(X_t)_{t\in[0,T]}$ to the path-dependent stochastic Volterra equation ~\eqref{eq:pthvolterra}.
   To this end, we recall that for a real-valued stochastic process $X$ on $\mathbb{T}$, the Kolmogorov continuity theorem states that if there exists some constants $p, a, C > 0$, such that
   \begin{equation}
   	\forall\, s, t \in \mathbb{T}\quad \mathbb{E}\left[ \left| X_t - X_s \right|^p \right]
   	\leq C \cdot |t - s|^{1 + a}, 
   \end{equation}
    then, for every $0 < \theta < a/p$, the process has a $\theta$-H\"older continuous modification . Hence we first need to establish the moment estimates for the increments of $X$.
   
   %%%%%%%%%%%%%%%%%%%%%%
      \begin{Theorem}[Pathwise Regularity and Maximal Inequality]\label{prop:pathkernelvolt}
   	Let $T > 0$ and $p > p_{eu}:=\frac{1}{\delta} \vee\frac{1}{\theta} \vee \frac{1}{ \widehat\theta}$.
   	Under Assumption \ref{assump:kernelVolterra}, the path-dependent Volterra equation~\eqref{eq:pthvolterra} admits, up to a $\P$-indistinguishability, a unique  $({\cal F}_t)$-adapted solution $X=(X_t)_{t\in [0,T]}$, pathwise continuous,  in the sense that, $\P$-$a.s.$,
   	\[
   	\forall\, t\!\in [0,T],  \quad X_t=x_0(t)+\int_0^t K_1(t,s) b(s,X_\cdot^s)ds+\int_0^t K_2(t,s)\sigma(s,X_\cdot^s)dW_s.
   	\]
   	This solution satisfies (up to a $\mathbb P$-indistinguishability or up to a path-continuous version $\tilde{X}$ ): 	
   	\begin{equation}\label{eq:Lpincrements}
   		\forall\, s,\, t\!\in [0,T],\quad \mathbb{E}\left[\left\|X_{t} - X_s\right\|_\mathbb{H}^p\right] \le C_{p,T} {\left( 1 + \int_{0}^{T}\mathbb{E}\left[ \| x_0(s) \|_{\mathbb{H}}^p \right] \mu (ds) \right)}|t-s|^{p(\delta \wedge\theta\wedge \widehat \theta)}.
   	\end{equation}
   	And thus, \( t \mapsto X_t \) admits a H\"older continuous modification (still denoted \( X \) in lieu of \( \tilde{X} \) up to a \( \mathbb{P} \)-indistinguishability).
   	More specifically $X$ admits a version which is H\"older continuous on $[0,T]$ of any order \( a \in \big(0,(\delta \wedge\theta\wedge \widehat\theta) -\frac 1 p\big) \) for any \( p > p_{eu} \), \( T > 0 \). Consequently, \(\forall\, a\!\in \big(0,\delta \wedge\theta\wedge \widehat\theta\big) , \;\):
   	\begin{equation}\label{eq:Holderpaths}
   		\Big \| \sup_{s\neq t\in [0,T]}\frac{\|X_t-X_s\|_{\mathbb H}}{|t-s|^{a}}\Big\|_p^p = \mathbb{E}\Big[\sup_{s\neq t\in [0,T]}\frac{\|X_t-X_s\|^p_{\mathbb H}}{|t-s|^{ap}} \Big]   <C_{a,p,T}  \Big( 1 + \int_{0}^{T}\mathbb{E}\Big[ \| x_0(s) \|_{\mathbb{H}}^p \Big] \mu (ds) \Big)
   	\end{equation}
   	for some positive constant $C_{a,p,T}=C_{a,b,\sigma, K_1,K_2, \theta,p,T} $. In particular
   	\begin{equation}\label{eq:supLpbound}
   		\Big \| \sup_{t\in [0,T]} \|X_t\|_{\mathbb H}\Big\|_p^p=\mathbb{E}\Big[ \sup_{t\in [0,T]} \|X_t\|^p_{\mathbb H}\Big] \le C'_{a,p,T}\Big( 1 + \int_{0}^{T}\mathbb{E}\Big[ \| x_0(s) \|_{\mathbb{H}}^p \Big] \mu (ds) \Big).
   	\end{equation}
   \end{Theorem}
   \noindent {\bf Remark:}
   1. If rather the condition\((\widehat {\cal K}^{cont}_{\widehat \theta})\)
   is satisfied for some $\beta>1$, then one can replace $\widehat \theta$ by $\frac{\beta-1}{2\beta}$ mutatis mutandis in the aboves claims (i.e. in ~\eqref{eq:Lpincrements} and~\eqref{eq:Holderpaths}). From now on, set \(\mathbb{X}:= C([0, T], \mathbb{H})\).  \footnote{At this stage, our results are in particular valid when $\Omega:=C([0,T];\mathbb{H})$ is the canonical space, $W$ the canonical process (namely $W(\omega) = \omega$), $\dbP$ the Wiener measure (namely $W$ is a standard $d^\prime$-dimensional Brownian motion under $\dbP$).}\\
   2. Surprisingly, in the above formulation the differents
   bounds~\eqref{eq:Lpincrements},~\eqref{eq:Holderpaths} and~\eqref{eq:supLpbound} only hold for starting random  vector \(X_0 \in L^p(\P)\) when \(p\) is large enough (say, \(p\geq2\) ) depending on some integrability
   characteristics of the kernels. This unusual restriction can be circumvented thanks to Lemma~\ref{lem:gap} (whose proof is postponed in the subsection \ref{app:representation}),
   applied to a particular functional as can be illustrated in the proof of Theorem~\ref{prop:pathkernelvolt2}.

\begin{Lemma}\label{lm:Uniformgrowth}
	Let the assumption \ref{assump:kernelVolterra} holds true. Then,
	%	As \(\mu\) is a	\(\mathbb{R}_+\)-valued finite measure on \([0, T]\) (i.e. the restriction $\mu|_{[0,T]}$ with $T > 0$ is well-defined), then:
	\begin{enumerate}
		\item[(i)] The functions ${\cal C}([0,T], \mathbb H)=:\mathbb{X}\ni x\mapsto b(t,x)$ and ${\cal C}([0,T], \mathbb H)=:\mathbb{X}\ni x\mapsto \sigma(t,x)$ are Lipschitz with Lipschitz  coefficient uniform in $t\in[0,T]$, i.e. \footnote{In particular $b$ and $\sigma$ have linear growth, i.e. there is a constant  $C>0$ such that \[
			\|b(t, x)\|_{\mathbb{H}} + \|\sigma(t, x) \|_{\tilde{\mathbb{H}}} \leq C (1 + \|x\|_t) \quad \textit{and} \quad \sup_{t \in [0,T]} \left( \|b(t,\bf{0})\|_{\mathbb{H}} + \|\sigma(t,\bf{0})\|_{\tilde{\mathbb{H}}} \right) < +\infty \; \text{if b and \(\sigma\) are continuous wrt t.}\]}
		{\small
			\[
			\exists \, C_{T} = C_{b, \sigma, T} \quad \text{such that} \quad \forall\, t \in [0,T], \; \forall\, x,\, y \in \mathbb{X}, \quad \|b(t,x)-b(t,y)\|_{\mathbb{H}} + \|\sigma(t,x)-\sigma(t,y)\|_{\tilde{\mathbb{H}}} \leq C_T \|x-y\|_t, 
			\]
		}
		\item[(ii)] The coefficient functions \(b\) and \(\sigma\) have a uniform linear growth in \( x \in {\cal C}([0,T], \mathbb H)=:\mathbb{X}\) in the sense that there exists a constant \(C\) such that for every \(t \in [0, T]\) and \( x \, \in \mathbb{X}\),
		\begin{equation}
			\|b(t,x)\|_{\mathbb{H}}\,  \vee  \,  \|\sigma(t,x)\|_{\tilde{\mathbb{H}}}  \leq  C (1 + \|x\|_t).
		\end{equation}
		
		\item[(iii)] Finally, assume that \( X_0 \in L^p(\mathbb{P}) \) for some \( p \in (0, +\infty) \), 
		the process $t \to x_0(t):=X_0\varphi(t)$ continuous and $(\cF_t)$-adapted. Moreover,  for some $\delta > 0$, for any $p>0$ and $T>0, \quad$\\		
		$\mathbb{E}\Big(\sup\limits_{t\in[0,T]}\|x_0(t)\|_{\mathbb{H}}^p\Big)<+\infty, $  $\mathbb{E}\|x_0(t^\prime)-x_0(t)\|_{\mathbb{H}}^p\leqslant C_{T,p}\Big( 1 + \mathbb{E}\Big[ \sup\limits_{t\in[0,T]} \| x_0(t) \|_{\mathbb{H}}^p \Big]\Big)|t'-t|^{\delta p}$
	\end{enumerate}
\end{Lemma}
\smallskip
\noindent {\bf Proof:}
The claims~$(i)$ and~$(ii)$ in the lemma are straightforward, as \(\mu\) is a finite measure on \([0, T]\) (i.e. the restriction $\mu|_{[0,T]}$ with $T > 0$ is well-defined) and in particular, its holds that

\( \int_{0}^T  \| x_0(t) \|_{\mathbb{H}}^p \,\mu(dt) \leq \sup\limits_{t\in[0,T]}  \| x_0(t) \|_{\mathbb{H}}^p \, \mu([0,T]) \leq \|x_0\|_{T}^p \; \mu([0,T]) < +\infty\)
and

\( \int_{0}^T \mathbb{E}\left[ \| x_0(t) \|_{\mathbb{H}}^p \right] \mu(dt) \leq \sup\limits_{t\in[0,T]} \mathbb{E}\left[ \| x_0(t) \|_{\mathbb{H}}^p \right] \mu([0,T]) \leq \mathbb{E}\Big(\sup\limits_{t\in[0,T]}\|x_0(t)\|_{\mathbb{H}}^p\Big) \mu([0,T]) < +\infty\).
\hfill$\Box$

  \begin{Theorem}\label{prop:pathkernelvolt2}   
	Let $T > 0$ and ${\bf p > 0}$. Assume that \(\mu\) is a finite measure on \([0, T]\).
	Under Assumption \ref{assump:kernelVolterra}, the Volterra equation~\eqref{eq:pthvolterra} admits, up to a $\P$-indistinguishability, a unique  $({\cal F}_t)$-adapted solution $X=(X_t)_{t\in [0,T]}$, pathwise continuous,  in the sense that, $\P$-$a.s.$,
	\[
	\forall\, t\!\in [0,T],  \quad X_t=x_0(t)+\int_0^t K_1(t,s) b(s,X_\cdot^s)ds+\int_0^t K_2(t,s)\sigma(s,X_\cdot^s)dW_s.
	\]
	This solution satisfies (up to a $\mathbb P$-indistinguishability or up to a path-continuous version $\tilde{X}$ ): 	
	\begin{equation}\label{eq:Lpincrements2}
		\forall\, s,\, t\!\in [0,T],\quad \mathbb{E}\left[\left\|X_{t} - X_s\right\|_\mathbb{H}^p\right] \le C_{p,T} {\left( 1 + \mathbb{E}\big[ \| x_0  \|_{T}^{p} \big]\right)}|t-s|^{p(\delta \wedge\theta\wedge \widehat \theta)}.
	\end{equation}
	Moreover, \(\forall\, a\!\in \big(0,\delta \wedge\theta\wedge \widehat \theta\big) , \)
	
	\begin{equation}\label{eq:Holderpaths2}
		\Big \| \sup_{s\neq t\in [0,T]}\frac{\|X_t-X_s\|_{\mathbb{H}}}{|t-s|^{a}}\Big\|_p^p = \mathbb{E}\Big[\sup_{s\neq t\in [0,T]}\frac{\|X_t-X_s\|^p_{\mathbb H}}{|t-s|^{ap}} \Big]   <C_{a,p,T}  \Big( 1 + \mathbb{E}\big[ \| x_0  \|_{T}^{p} \big]\Big)
	\end{equation}
	for some positive real constant $C_{a,p,T}=C_{a,b,\sigma, K_1,K_2, \theta,p,T} $. In particular
	\begin{equation}\label{eq:supLpbound2}
		\Big \| \sup_{t\in [0,T]} \|X_t\|_{\mathbb H}\Big\|_p \le C'_{a,p,T}\Big( 1 + \big\| \,\| x_0  \|_{T} \,\big\|_p\Big).
	\end{equation}
		Finally, if the integrability assumption~$({\cal K}^{int}_{\beta})$
		is satisfied for some $\beta>1$, then one can replace  $\widehat \theta$  by $\frac{\beta-1}{2\beta}$ mutatis mutandis in the aboves claims (i.e. in ~\eqref{eq:Lpincrements2} and~\eqref{eq:Holderpaths2}).
\end{Theorem}
\noindent {\bf Remark:} This result appears as a generalization of the classical strong existence-uniqueness result of pathwise continuous solutions established in  \cite[Theorem 1.1]{JouPag22} as an improved version of \cite[Theorem 3.1 and Theorem 3.3]{ZhangXi2010} for classical Volterra SDEs. (The framework is more general with a random function \(\phi\) in front of the starting functional \(x_0\) evolving deterministically).
\begin{Lemma}[Moments Control]\label{lem:supEbound}
	
	Let the assumption \ref{assump:kernelVolterra} holds true and $X$ be a continuous solution to the Path-dependent Volterra equation~\eqref{eq:pthvolterra} with initial condition $x_0\!\in \mathbb{X}$. Then, for any \( p \geq p_{eu} \) and \( T < \infty \), one has:
	\[
	\sup_{t \leq T} \mathbb{E}[\|X_t\|_{\mathbb{H}}^p] \leq C'\Big( 1 + \int_{0}^{T}\mathbb{E}\left[ \| x_0(s) \|_{\mathbb{H}}^p \right] \mu (ds) \Big).
	\]
	Moreover, if  \(\mu\) is a finite measure on \([0, T]\) so that Lemma\ref{lm:Uniformgrowth} applies, we get  for any \( p >0 \):
	\[
	\sup_{t \leq T} \mathbb{E}[\|X_t\|_{\mathbb{H}}^p] \leq  C\Big(1+ \sup_{ t \leq T}\mathbb{E}[\|x_0(t)\|_{\mathbb{H}}^p] \Big),
	\]
	for some constants \( C' \) and \( C \) that only depend on \( K_i|_{[0,T]} \) $i = 1, 2$, \( b,\sigma \), \( p \) and \( T \).
\end{Lemma}

\noindent {\bf Proof of Lemma~\ref{lem:supEbound}: } 
This follows from Jensen's inequality, which, in this context, states that \(\sup_{0 \le t \le T} \mathbb{E}[\|X_t\|_{\mathbb{H}}^p] \leq \mathbb{E}[ \sup_{0 \le t \le T} \|X_t\|_{\mathbb{H}}^p]\). 
Alternatively, the result could also be deduced by reasoning in a similar way as in the proof of the first claim of Theorem~\ref{Thrm:flotVolterra} (See also, the proof of the second Property of Proposition~\ref{propXtilde}). \hfill$\Box$ 

\subsection{Representation of the solution of  the path-dependent Volterra equation} \label{app:representation} 
Recall that the initial \(\mathcal{F}_0\)-measurable input \(x_0(t) := X_0 \varphi(t)\), where \(\varphi\) is a deterministic function and \(X_0\) is an \(\mathbb{H}\)-valued random variable. 
For any random variable \(X_0\) (resp. any realization \(x_0\) of \(X_0\)), we will say that \(X^{X_0}\) (resp. \(X^{x_0}\)) is the solution to the path-dependent Volterra equation~\eqref{eq:pthvolterra} starting from \(X_0 \in \mathbb{H}\) (resp. \(x_0 \in \mathbb{H}\)).

   \begin{Theorem}[Path-dependent Volterra's flow]\label{Thrm:flotVolterra} 
	Assume that assumption \ref{assump:kernelVolterra} on \(b\) and \(\sigma\) is in force and that \(\sup_{t \in [0,T]} \left( \|b(t,\bf{0})\|_{\mathbb{H}} + \|\sigma(t,\bf{0})\|_{\tilde{\mathbb{H}}} \right) < +\infty\). Assume also that the kernels $K_i$, $i = 1, 2$, satisfy $\big({\cal K }^{int}_{\beta}\big)$ and $({\cal K}^{cont}_{\theta})$ for some $\beta > 1$ and $\theta\in (0,1]$ respectively. 
	
	\smallskip
	\noindent $(a)$  Let $X^x$ denotes the solution to the Path-dependent Volterra equation~\eqref{eq:pthvolterra} starting from $x\!\in \mathbb{H}$. For every $p\geq 2$ and every two  $\mathcal{F}_{0}$-measurable random starting function \(x_0(t) := X_0 \varphi(t)\) and \(y_0(t) := Y_0 \varphi(t)\) in \(\mathbb{X}\), we have
	\begin{equation}\label{eq:boundflow1}
		\sup_{t\in [0,T]}\;\Big\| \|X^{X_0}_t-X^{Y_0}_t\|_{\mathbb{H}} \Big\|_p \le C \sup_{t\in [0,T]}\Big\| \|x_0(t) - y_0(t) \|_{\mathbb{H}} \Big\|_p \leq C \Big\| \|x_0 - y_0 \|_{T} \Big\|_p.
	\end{equation}
	where 
	\(C := 2\, e^{K\,T}\) 
	for some positive real constant $K$ only depending on $ K_1, K_2, p, \beta, b, \sigma$ and again \(T\).
	
	\smallskip
	\noindent $(b)$ (Continuity of the flow).  Let $X^x$ denotes the solution to the Path-dependent Volterra equation~\eqref{eq:pthvolterra} starting from $x\!\in \mathbb{H}$ and let $\lambda \!\in (\frac 12, 1)$. There exists {$p^*=p^*_{\beta, \delta,\gamma, \lambda,\text{dim}(\mathbb{H})}$ (made explicit in the proof)}  
	such that for every $p> p^*$,  
	\begin{equation}\label{eq:boundflow2}
		\forall \,x,y\in\mathbb{H},\;\Big\| \sup_{t\in [0,T]} \|X^x_t-X^y_t\|_{\mathbb{H}} \Big\|_p \le C \|(x-y)\varphi\|_{T}^{\lambda},
	\end{equation}
	
	for some positive real constant $C:= C_{p, b,\sigma, K_1, K_2, \beta, \theta, \delta}$.
\end{Theorem}
    \noindent For every $x \in \mathbb{H}$, the diffusion process starting at x, denoted by $X^{x}$ , satisfies the following two obvious facts:
    \noindent {\textit{(a)}.} The process $X^{x}$ is $\mathcal{F}^W_t$-adapted, where the $\P$-completed natural filtration $\mathcal{F}^W_t$ is defined like:  $\mathcal{F}^W_t = \sigma({\cal N}_{\P}, X_0, W_s,\,0\le s\le t), t\!\in [0,T]$.\\
    	{\textit{(b)}.}  If $X_0$ is an $\mathbb{H}$-valued random vector defined on $(\Omega,\cF,\dbP)$ independent of W, then the process $X=(X_t)_{t\in [0,T]}$ starting from $X_0 \in \mathbb{H}$ satisfy: $X_t =X^{X_0}_t.$
    	
    \smallskip \noindent 
    As noted in the remark following Theorem~\ref{prop:pathkernelvolt}, the path regularity bounds for solutions to path-dependent SVIEs require  $X_0\!\in L^p(\P)$ with sufficiently large  $p$. This unusual restriction is removed in Theorem~\ref{prop:pathkernelvolt2} thanks to the following lemma~\ref{lem:gap}.
    \begin{Lemma}[Splitting lemma for path dependent Volterra Equations]\label{lem:gap} 
    	 Assume that the functions $b$ and $\sigma$ are Lipschitz in space uniformly in time in the sense of assumption \ref{assump:kernelVolterra}, that $\displaystyle\sup_{t\in[0,T]}(\|b(t,\bf{0})\|_{\mathbb{H}}+\|\sigma(t,\bf{0})\|_{\tilde{\mathbb{H}}})<+\infty$ and that the kernels $K_i$, $i=1,2$, satisfy $({\cal K}^{int}_{\beta})$ and $({\cal K}^{cont}_{\theta})$ for some $\beta>1$ and $\theta\!\in (0, 1]$ respectively. 
    	Let $\Phi : \mathbb{X}^2 \to \mathbb{R}$ be a Borel functional and let $n \!\in \N$. Assume there exists  $\bar p>0$ and a real constant $C>0$ possibly depending on $n$, such that, for every $\xi \!\in \mathbb{H}$, 
    	\[
    	\| \Phi(X^{\xi},\bar X^{\xi}) \|_{\bar p} \le C \left( 1 + \| \xi \varphi \|_{T} \right),
    	\]
    where $X^{\xi}$ and \(\bar{X}^{\xi}\) respectively denote the solutions of the path-dependent Volterra equation~\eqref{eq:pthvolterra} and its interpolated $K$-integrated discrete time Euler scheme $(\bar{X}^{\xi}_{t_k^n})_{0 \leq k \leq n}$ starting from $\xi$.
    	Then, for every $p\!\in (0,\bar p]$ and every random function $x_0:= X_0 \varphi$ such that the random vector \(X_0\) is independent of $W$, the solution $X=(X_t)_{t\in [0,T]}$ starting from $ X_0$ satisfy
    	\[
    	\| \Phi(X,\bar X) \|_p \le 2^{|1-\frac1p|} C \left( 1 + \big\| \,\| x_0  \|_{T} \,\big\|_p\right).
    	\]
    \end{Lemma}	

    \smallskip
    \noindent Our next task is to establish representations for the path-dependent Volterra process and its interpolated $K-$integrated Euler scheme(~Proposition~\ref{lem:K-intF} ) in order
    to be able to apply the above splitting lemma.
    The next theorem deals with the flow  $x\mapsto (X^x_t)_{t\in [0,T]}$ of the path-dependent stochastic Volterra equation. It proves the existence of a bi-measurable functional $F: \mathbb{H}\times {\cal C}_0(\mathbb{T}, \bar{\mathbb{H}})\to \mathbb{X}$   such that the solution $X=(X_t)_{t\in \mathbb{T}}$ of ~\eqref{eq:pthvolterra} reads $X= F(x_0,W)$ for any (finite) starting random vector  $X_0$.
    
    \begin{Theorem}[ An extension of the Blagove$\check{\rm \bf s}\check{\rm \bf c}$enkii-Freidlin theorem: representation  of  Path-dependent Volterra's flow]\label{thm:FreidlinLike} Let the assumption \ref{assump:kernelVolterra} on \(b\) and \(\sigma\) holds true and that \(\sup_{t \in [0,T]} \left( \|b(t,\bf{0})\|_{\mathbb{H}} + \|\sigma(t,\bf{0})\|_{\tilde{\mathbb{H}}} \right) < +\infty\). Assume also that the kernels $K_i$, $i = 1, 2$, rather satisfy $\big({\cal K }^{int}_{\beta}\big)$ and $({\cal K}^{cont}_{\theta})$ for some $\beta > 1$ and $\theta\in (0,1]$ respectively. Let $X^x$ denotes the solution to the Volterra equation~\eqref{eq:pthvolterra} starting from $x\!\in \mathbb{H}$.  There exists a functional $F: \mathbb{H}\times {\cal C}_0(\mathbb{T}, \bar{\mathbb{H}})\ni(x,w) \mapsto F(x,w) \in \mathbb{X}$ bi-measurable (where the spaces of continuous functions are equipped with the Borel sigma-field induced by the uniform convergence topology) and continuous in $x$ such that, for any stochastic basis $(\Omega, {\cal A}, \P, (\cF_t)_{t\in [0,T]})$, any $\bar{\mathbb{H}}$-dimensional $(\cF_t)_t$-Brownian motion $W$ and any $\cF_0$-measurable $\mathbb{H}$-valued random vector $X_0\!\in L^0(\P)$, the solution to the Volterra equation~\eqref{eq:pthvolterra} is $X=F(X_0,W)$. 
    	\label{thm:BF}
    \end{Theorem}
For clarity and conciseness, the proofs of the above results and Theorem~\ref{thm:FreidlinLike} are deferred to Appendix~\ref{app:B}, where the main technical results are presented.
%\section{ Numerical Approximation of Path-dependent SVIEs}\label{sec:num_approx}

\section{ Stochastic Approximation of Path-dependent Volterra Equations}\label{sec:num_approx}
In this section, we propose a numerical approximation method for path-dependent Stochastic Volterra Integral Equations and its convergence rate.

\noindent {\bf Remark on Notations:}
	To simplify the notations, we introduce the following definitions:
	\begin{equation}\label{eq:notat1}
	\varphi_{a}(t) = \left( \int_0^t K_1(t,s)^a \, ds \right)^{1/a}, \, a\in (0,\frac{2\beta}{\beta+1}] \quad\text{and}\quad \psi_{b}(t) = \left( \int_0^t K_2(t,s)^b \, ds \right)^{1/b}, \, b\in (0,2\beta],
	\end{equation}
	where these quantities are finite due to Assumption~$({\cal K}^{int}_{\beta})$ and H\"older's inequality, as $\beta > 1$. Additionally, we define the supremums:
	\begin{equation}\label{eq:notat2}
	\varphi^*_a(T) = \sup_{t \in [0,T]} \varphi_a(t), \quad \psi^*_b(T) = \sup_{t \in [0,T]} \psi_b(t),
	\end{equation}
	It is noteworthy that \( \varphi^*_1(T) + \psi^*_2(T) < +\infty \) due to Assumption~$({\cal K}^{int}_{\beta})$ and the application of H\"older's inequality, again under the condition that \( \beta > 1 \).

\medskip
We start by the following technical Lemma which will be pivotal for the subsequent developments.
\begin{Lemma}[Bound On Volterra Integral]\label{Lem:BoundVoltIntegral}
	Let $f : [0, T] \to \mathbb{R}_{+}$ be a Borel, locally bounded, positive and non-decreasing function. Assume that the kernels \(K_i\), \(i = 1, 2\), satisfy the integrability assumption~$({\cal K}^{int}_{\beta})$. Then, for every \( i \in \{1,2\}\),
	\begin{equation}\label{eq:VoltIntegral}
		\left( \int_0^t K_i(t,s)^i f(s)^i \, ds \right)^{1/i}
		\le 
		\chi_i(t) \left( 
		\rho\,\frac{f(t)}{a^{1/\rho}} 
		+ (1-\rho)\,a^{1/(1-\rho)} \int_0^t f(s)\,ds
		\right).
	\end{equation}
	for all \( a > 0 \) and \(\rho \in (0,1)\), where \(\chi_1 = \varphi_{\frac{2\beta}{\beta+1}}\) and \(\chi_2 = \psi_{2\beta}\).
\end{Lemma}
As a consequence, we deduce an
extended version of the so-called Gr\"onwall's lemma given in \cite[Lemma 7.3]{pages2018numerical}, at least for Volterra type. 
\begin{Corollary}[``\`A la Generalized Gr\"onwall" Lemma of Volterra type ]\label{Gronwall}
	Assume that the kernels \(K_i\), \(i = 1, 2\), satisfy the integrability assumption~$({\cal K}^{int}_{\beta})$.
	Let $f : [0, T] \to \mathbb{R}_{+}$ be a Borel, locally bounded, and non-decreasing function, and let $\psi : [0, T] \to \mathbb{R}_{+}$ be a non-negative, non-decreasing function satisfying the inequality
	\begin{equation}\label{eq:VoltGronwall}
	\forall t \in [0, T], \, f(t) \leq A \int_0^t K_1(t,s)f(s) \, ds + B \left( \int_0^t K_2(t,s)^2f^2(s) \, ds \right)^{\frac{1}{2}} + \psi(t),
	\end{equation}
	where $A$ and $B$ are positive constants and
	$K_i : \mathbb{T}^2_{-} \to \mathbb{R}_+, i \in \{1, 2\}$ is a either a non singular or a singular kernel. Then, for any $t \in [0, T]$, we have the upper bound \(	f(t) \leq 2 e^{K^{\prime} t} \psi(t),\)
	where 
	$K^{\prime}= K^{\prime}_{A,B,K_1,K_2,T,\beta,\rho} =  2^{\frac{1}{1-\rho}}(1-\rho)\rho^{\frac{\rho}{1-\rho}}\left( A \varphi_{\frac{2\beta}{\beta+1}}^* + B \psi_{2\beta}^* \right)^{\frac{1}{1-\rho}}$,  $\rho \in (0,1)$, $\beta >1$ and $\varphi_{\frac{2\beta}{\beta+1}}^*, \psi_{2\beta}^*$ defined by ~\eqref{eq:notat2}.
\end{Corollary}

\subsection[Temporal discretization by an interpolated K-integrated discrete time Euler scheme.]{Temporal discretization by an interpolated K-integrated discrete time Euler-Maruyama scheme.}\label{subsec:GM}

In the following, $n \in \mathbb{N}^*$ denotes the number of time discretizations, and $h := \frac{T}{n} = t^n_{k}-t^n_{k-1}$ is the time step. For each $k = 0, \dots, n$, we define $t_k^n = kh$. To simplify the notation, we will use the short-hands  $x_{0:k}:=  (x_0, \dots, x_k)$ (or equivalently $x_{t^n_{0}:t^n_{k}}:=  (x_{t^n_{0}}, \dots, x_{t^n_{k}})$) and $x_{t^n_{0}:t^n_{k}}^{0:k}:=  (x_{t^n_{0}}^0, \dots, x_{t^n_{k}}^k)$. 
The interpolated Euler scheme is based on the following interpolator.

\begin{Definition}[Interpolator]\label{definterpolator}
	For each $k = 1, \dots, n$, we define a piecewise affine interpolator $\iota_k$ on $k+1$ points in $\mathbb{H}$ by
	\begin{equation}\label{definterp1}
		x_{0:k} \in \mathbb{H}^{\otimes (k+1)} \longmapsto \iota_k(x_{0:k}) = (\bar{x}_t)_{t \in [0,T]} \in  \mathbb{X},
	\end{equation}
	where for every $t \in [0,T]$, by convention we set, $\iota_0(x_0)_t := x_0$ and $\bar{x}_t$ is defined by
	\begin{align}
		&\forall \ell = 0, \dots, k-1, \, \forall t \in [t_\ell^n, t_{\ell+1}^n), \; \bar{x}_t = \frac{1}{h}(t_{\ell+1}^n-t) x_\ell + \frac{1}{h}(t - t_\ell^n) x_{\ell+1},\;\text{and}\; \forall t \in [t_k^n, T], \; \bar{x}_t = x_k. 
	\end{align}
\end{Definition}
\noindent The ${\cal C}(\mathbb{T}, \mathbb{H})$-valued interpolation operator $\iota_k$ associated to the mesh $(t^n_k = \frac{kT}{n})_{0\le k\le n}$, already introduced in~\cite{Pag2016} is used to transfer our discrete time results to continuous time. It is defined more compactly as:
\vspace{-.6cm}
\begin{equation}\label{eq:interpolator}
\iota_k: \mathbb{H}^{\otimes (k+1)}\ni x_{0:k} \longmapsto \Big(t \mapsto  \sum_{\ell=1}^k \mbox{\bf 1}_{[t^n_{\ell-1},t^n_\ell)}(t) \left(\frac{t^n_{\ell}-t}{h}x_{\ell-1} + \frac{t-t^n_{\ell-1}}{h} x_\ell\right)+\mbox{\bf 1}_{\{t\geq t_k^n\}}x_k\Big)\!.
\end{equation}
\begin{Lemma}[Properties of the interpolator $\iota_k$]\label{interpolatorprop} Let $k \in \mathbb{N}^*$. 
	\begin{enumerate}
		\item[$(a)$] 
		For every $x_{0:k} \in \mathbb{H}^{\otimes (k+1)}$, 
		the function $\iota_k(x_{0:k})$ being continuous and piecewise affine with affinity breaks at times $t^n_k$ where it is equal to $x_k$ so that \(\Vert \iota_k(x_{0:k}) \Vert_{t^n_k} = \sup_{0\leq \ell\leq k}\|x_\ell\|_{\mathbb H}= \max_{\ell=0,\ldots,k} \|x_\ell\|_{\mathbb H}.\)
		\item[$(b)$] 
		For $x_{0:k},y_{0:k}\in\mathbb{H}^{\otimes (k+1)}$, \(	\| \iota_k(x_{0:k}) -\iota_k(y_{0:k}) \|_{t^n_k} = \| \iota_k(x_{0:k}-y_{0:k}) \|_{t^n_k}= \max_{\ell=0,\ldots,k} \|x_\ell-y_\ell\|_{\mathbb{H}}\) so that  $\iota_k$ is Lipschitz with constant $1$ from $\mathbb{H}^{\otimes (k+1)}$ to $\big(\mathbb{X}, \|\cdot\|_{T}\big)$.
	%	\item[$(d)$] For every $x_{0:k}, y_{0:k} \in \mathbb{H}^{\otimes (k+1)}$, we have: \(\Vert \iota_k(x_{0:k}) - \iota_k(y_{0:k}) \Vert_{t^n_k} \leq \max_{0 \leq \ell \leq k} \|x_\ell - y_\ell\|_{\mathbb H}.\)
	\end{enumerate}
\end{Lemma}
\noindent{\bf Proof:} We only need to prove the equality in $(a)$, from which the inequality and the claim $(b)$ can be straightforwardly obtained through Definition \ref{definterpolator}.

\noindent $(a)$  First, it is obvious that $\sup_{0\leq \ell\leq k}\|x_\ell\|_{\mathbb H}\leq \big\Vert \iota_k (x_{0:k}) \big\Vert_{\sup}$ by the definition of $\iota_k$. For every $\ell\in\{0, ..., k-1\}$, for every $t\in[t_\ell, t_{\ell+1}]$, we have \(\big\|\iota_{k}(x_{0:k})_{t}\big\|_{\mathbb H}\leq \|x_\ell\|_{\mathbb H}\vee \|x_{\ell+1}\|_{\mathbb H} \leq \sup_{0\leq \ell\leq k}\|x_\ell\|_{\mathbb H}\)
and for every $t\in[t_k, T]$, we have $\big\|\iota_{k}(x_{0:k})_{t}\big\|_{\mathbb H}= \|x_{k}\|_{\mathbb H}\leq \sup_{0\leq \ell\leq k}\|x_\ell\|_{\mathbb H}$. Then we can conclude $\sup_{0\leq \ell\leq k}\|x_k\|_{\mathbb H}= \big\Vert \iota_k (x_{0:k}) \big\Vert_{\sup} = \max_{\ell=0,\ldots,k} \|x_\ell\|_{\mathbb H}$. 
\hfill$\Box$\\

In \cite{JouPag22}, a \(K\)-discrete Euler scheme was firstly proposed for a non-path-dependent equation before the K-integrated scheme, along with its natural continuous time (or "genuine" ) extension. However, in our case, we will directly propose a \(K\)-integrated discrete-time Euler scheme for the path-dependent volterra Equations, as it offers greater precision compared to the naive version (see. e.g. \cite[Appendix A]{EGnabeyeu2025}).

\begin{Definition}[Interpolated $K$-integrated Euler scheme]\label{def:discretization_scheme}
	Given the Brownian motion $(W_t)_{t \in [0,T]}$ and an initial random vector $X_0$ , the \textit{discrete time interpolated $K$-integrated Euler scheme} $(\bar{X}^h_{t_k^n})_{0 \leq k \leq n}$ for the path-dependent Volterra Integral Equations equation~\eqref{eq:pthvolterra} is given by:
	
     $1.$ $\bar{X}_0^h = X_0$;\qquad $2.$ For all $k=1,\ldots,n$,
		\begin{equation}\label{eq:discretescheme}
			\bar X_{t^n_k}^h = X_0 \varphi(t^n_k) + \sum_{\ell=0}^{k-1} \Big(b_\ell(t^n_{\ell}, \bar X^h_{t_0^n:t^n_{\ell}}) \int_{t^n_{\ell}}^{t^n_{\ell+1}} K_1(t^n_k, s) \, ds + \sigma_\ell(t^n_{\ell}, \bar X^h_{t_0^n:t^n_{\ell}}) \int_{t^n_{\ell}}^{t^n_{\ell+1}} K_2(t^n_k, s) \, dW_s \Big),
		\end{equation}
	The applications $b_\ell$ and $\sigma_\ell$ are defined on $\mathbb{T} \times \mathbb{H}^{\otimes (l+1)} $ and are valued in $\mathbb{H}$ and $\tilde{\mathbb{H}}$, respectively, with
	\begin{align}\label{defbmsigmal}
		b_\ell(t, x_{0:\ell}) := b\left(t, \iota_\ell(x_{0:\ell})\right) \;\text{and}\;\;
		\sigma_\ell(t, x_{0:\ell}) := \sigma\left(t, \iota_\ell(x_{0:\ell})\right).
	\end{align}
	Additionally setting \( \underline s = \underline s_n=  t^n_k := \frac{kT}{n}, \quad [\underline s] := k \quad \hbox{if} \quad s\!\in [t^n_k, t^n_{k+1})\), the process $(\bar X^{h}_{t^n_k})_{k \in 0 \dots n}$ defined by \eqref{eq:discretescheme} satisfies 
	the natural continuous time extension $(\bar{X}^h_t)_{t \in [0,T]}$ (entitled the \textit{genuine (or continuous time) interpolated $K$-integrated Euler scheme}) defined by setting, for all $t \in [t_k^n, t_{k+1}^n)$,
	\begin{equation}\label{eq:def_continuous_2}
		\bar X^{h}_t = \bar X^{h}_0\varphi(t) + \int_0^t K_1(t,s) b \Big(\underline s, \iota_{[\underline s]}\big(\bar X^{h}_{t_0:t_{[\underline s]}} \big) \Big)\, ds + \int_0^t K_2(t,s) \sigma \Big(\underline s, \iota_{[\underline s]}\big(\bar X^{h}_{t_0:t_{[\underline s]}} \big) \Big)\, d W_s.
	\end{equation}
	{which also reads ``in extension'' when $t\!\in [t^n_k,t^n_{k+1})$,}
	%\textcolor{blue}{
		\begin{align}
			\nonumber  \bar X_{t}^h &=\bar X^{h}_0\varphi(t) + b_k(t^n_{k}, \bar X^h_{t_0^n:t^n_{k}})\int_{t^n_{k}}^{t}  K_1(t,s)ds \,  +\sigma_k(t^n_{k}, \bar X^h_{t_0^n:t^n_{k}})\int_{t^n_{k}}^{t}   K_2(t,s)dW_s \\
			\label{eq:Eulergen2bis}   & +  \sum_{\ell=0}^{k-1} \Big(b_\ell(t^n_{\ell}, \bar X^h_{t_0^n:t^n_{\ell}}) \int_{t^n_{\ell}}^{t^n_{\ell+1}} K_1(t, s) \, ds + \sigma_\ell(t^n_{\ell}, \bar X^h_{t_0^n:t^n_{\ell}}) \int_{t^n_{\ell}}^{t^n_{\ell+1}} K_2(t, s) \, dW_s \Big),\;k=0,\ldots,n.
		\end{align}
	\end{Definition}
	\noindent {\bf Remark:} In the above terminology, the term \textit{interpolated} refers to the fact that the drift and diffusion functions in the scheme are based on the piecewise interpolation defined in equation~\eqref{defbmsigmal}. Moreover,
	\begin{enumerate}
		\item When a closed or a semi-closed form is available for the vectors $\displaystyle\Big[\int_{t^n_{\ell-1}}^{t^n_{\ell}}K_1(t^n_k,s)ds\Big]_{1\le \ell\le k}$, $k=1,\ldots,n$ and the covariance matrices $\displaystyle \Big [\int_{t^n_{\ell-1}}^{t^n_{\ell}}K_2(t^n_k,s)K_2(t^n_{k'},s)ds \Big]_{\ell\le k, k'\le n}$, $\ell=1,\ldots,n$, then the above \textit{discrete time interpolated $K$-integrated Euler scheme} becomes  simulable (see practitioner's corner further on).
			
		\item The functions $b_\ell$ and $\sigma_\ell$ defined in \ref{defbmsigmal} process discrete inputs, which often simplifies the computations. 
		
		%For example, if \(b\left(t, (X^t_s)_{s \in [0,T]}\right) := \int_0^t \Phi(s,X_s) \, ds\)
	%	\begin{equation}\label{eq:b_integral}
	%		b\left(t, (X_s)_{s \in [0,T]}\right) := \int_0^t \Phi(s,X_s) \, ds
	%	\end{equation}
	%	with a bounded continuous function $\Phi$, then, by the definition of $b_\ell$,
	%	\begin{equation}\label{eq:NumIntegral}
	%		b_\ell(t_\ell^n, \bar{X}^h_{t_0^n:t_\ell^n}) = \frac{h}{2}\left(\Phi(t_0^n,\bar{X}^h_{t_0^n}) + \Phi(t_\ell^n,\bar{X}^h_{t_\ell^n})\right) + h \sum_{m=1}^{\ell-1} \Phi(t_m^n ,\bar{X}^h_{t_m^n}).
	%	\end{equation}
	%	\noindent Clearly, numerical computation of an integral, as in the definition of \(b\), is more computationally intensive than handling sums, as in the above equation.
		
		\item Due to the lack of Markovianity, $\bar X_{t^n_k}^h$ is generally not a function of $(\bar X^h_{t^n_{k-1}}, W_{t^n_k} - W_{t^n_{k-1}})$. However, it can be computed uniquely from $(\bar X_0^h, \dots, \bar X^h_{t^n_{k-1}})$ and from the Gaussian vector $\left(\int_{t_{\ell-1}}^{t_{\ell}} K_2(t^n_k,s) dW_s \right)_{\ell=1,\dots,k}$ for equation~\eqref{eq:discretescheme}, ensuring that the Euler scheme is well-defined by induction. The complexity for computing $(\bar X_{t^n_k})_{k=0, \dots, n}$ is ${\mathcal O}(n^2)$, and for the discrete time interpolated $K$-integrated scheme, independent Gaussian vectors $ \left(\int_{t_{\ell}}^{t_{\ell+1}} K_2(t^n_k,s) dW_s\right)_{\ell=1, \dots, n}$ must be generated.
	\end{enumerate}
	\medskip
	\noindent {\bf Practitioner's corner:}
	{\em Simulation of the semi-integrated scheme for the Path-dependent stochastic Volterra integral Equations~\eqref{eq:discretescheme}}.
	The exact simulation of the discretization or Euler scheme~\eqref{eq:discretescheme} involves computing the weight matrix \(\left( \int_{t^n_{\ell}}^{t^n_{\ell+1}} K_1(t^n_k, s) \, ds \right)_{1 \leq \ell \leq k \leq n},\)
	and simulating the independent random vectors \(G^{n,\ell} = \left( \int_{t^n_{\ell}}^{t^n_{\ell+1}} K_2(t^n_k, s) \, dW_s \right)_{\ell \leq k \leq n}, \quad \ell = 1, \dots, n.\) Denoting by $\Lambda^n$ the $n\times n$ lower triangular matrix involving the deterministic integrals, namely 
	$$ \Lambda^n=(\Lambda^n_{k\ell})_{k,\ell=1:n}= \left(
	\int_{t_{\ell-1}}^{t_\ell} K_1(t_k, s)ds\mbox{\bf 1}_{\{1\le \ell\le k\le n\}}
	\right)_{k,\ell=1:n}.
	$$
	The deterministic terms can be easily computed using high performance  numerical integration methods, e.g. the quadrature function \(\textit{quad}\) of the library \(\textit{scipy}\). The primary challenge lies in the simulation of the random terms \( \int_{t_{\ell-1}}^{t_{\ell}} K_2(t_k, s) \, dW_s \), which are represented by the \((n+1) \times n\) matrix:
	\begin{equation}\label{eq:GaussianVect}
	G^n = (G^n_{k\ell})_{k=1:n+1, \ell=1:n} = \left( \begin{array}{l}
		\int_{t_{\ell-1}}^{t_\ell} K_2(t_k, s) \, dW_s \, \mathbf{1}_{\{ 1 \leq \ell \leq k \leq n \}}	\\
		\qquad \Delta W_{t_\ell}\mathbf{1}_{\{ k = n+1 \}} %\, k = n+1, \, \ell = 1:n
	\end{array} \right)_{k=1:n+1, \ell=1:n} ,
    \end{equation}
	where the inclusion of the last row is necessary for the consistent joint simulation of the Brownian increments, typically, in the context of volatility model dynamics, both the asset's price process and its volatility can be driven by the same Brownian motion. For instance, such is the case in the quadratic rough volatility model introduced in~\cite{GaJuRo2020}. 
	
		Then the following relation holds: $\overline{X}_0^h = X_0$ and \(\big( \overline X^h_{t_{k}} \big)_{k=1:n}\) is such that for every $k=1,\dots,n$
		\begin{equation}\label{eq:discrete-interpolated-scheme}
		\overline X^h_{t_{k}}= X_0 \varphi(t^n_k)\mbox{\bf 1}+\Lambda^n_{k,\cdot}\odot\big( b_{\ell-1}(t^n_{\ell-1}, \overline{X}^h_{t_0^n:t^n_{\ell-1}})\big)_{\ell=1:n} +G^n_{k,\cdot}\odot\big( \sigma_{\ell-1}(t^n_{\ell-1}, \overline{X}^h_{t_0^n:t^n_{\ell-1}})\big)_{\ell=1:n}
		\end{equation}
		 where $\odot$ denote the Hadamard (pointwise or component-wise) product, 
		and for every $\ell= 1,\ldots,n$, the applications $b_\ell$ and $\sigma_\ell$ are defined on $\mathbb{T} \times \mathbb{H}^{\otimes (\ell+1)} $ by~\eqref{defbmsigmal} with valued in $\mathbb{H}$ and $\tilde{\mathbb{H}}$, respectively.
		
		The reader is referred to Appendix~\ref{app:A} for details on the simulation of the Gaussian stochastic integral terms \(G^n\) appearing in the discretization or the interpolated semi-integrated scheme~\eqref{eq:discrete-interpolated-scheme}.
		{\begin{Proposition}[Blagove$\check{\rm \bf s}\check{\rm \bf c}$enkii-Freidlin's representation formula for  the interpolated $K-$ integrated Euler scheme]\label{lem:K-intF} 
			%	 Assume {that  $(\widehat {\cal K}^{cont}_{\widehat \theta})$} holds for some $\widehat\theta\in (0,1]$.
				Assume either $(\widehat {\cal K}^{cont}_{\widehat \theta})$ or $({\cal K}^{int}_{\beta})$ and $({\cal K}^{cont}_{\theta})$ are in force. Let $n\ge 1$. There exists a bi-measurable functional $\bar F_n : \mathbb{H}\times {\cal C}_0([0,T], \bar{\mathbb{H}})\to  \mathbb{X}$ such that, for any stochastic basis $(\Omega, {\cal A}, \P, (\cF_t)_{t\in [0,T]})$, any $\bar{\mathbb{H}}$-valued $(\cF_t)_t$-Brownian motion $W$ and any $\cF_0$-measurable $\mathbb{H}$-valued random function $X_0\!\in L^0(\P)$,  the interpolated $K$-integrated Euler scheme $\bar X$ with step $\frac Tn$ starting from $X_0$ defined by~\eqref{eq:Eulergen2bis} writes
			\[
			\bar X = \bar F_n (X_0,W).
			\]
		\end{Proposition}
	}
	In the sequel, we present results on the approximation scheme introduced earlier. Specifically, in Section~\ref{subsec:prop_x_tilde} we derive a key preliminary result, Proposition~\ref{propXtilde}, and then proceed to the main result, Theorem~\ref{thm:Eulercvgce2}, along with its associated Corollary~\ref{corol:Eulercvgce2} in Section~\ref{subsec:Cvgce_Euler}. 
	\subsection{Some Properties of the  continuous extension process of the interpolated K-integrated Euler scheme and $L^p$-convergence
		at fixed times.}\label{subsec:prop_x_tilde}
	Here, we provide some properties of the  continuous extension process $(\bar X^h_t)_{t \ge 0}$ defined by \eqref{eq:def_continuous_2}.
	
		\begin{assumption}\label{assump:VoltEulerScheme}
		We consider the following:
		\begin{enumerate}
			\item[(i)] Assume that 
			$X_0\in L^{p}(\mathbb{P})$ and the coefficient functions $b, \sigma$ are continuous in $t$, uniformly Lipschitz continuous in $x$ in the following sense there exists $C_T>0$ such that \(\forall \, t \in \mathbb{T}, \; \forall \, x, y \in \mathbb{X},\):  
			\begin{equation}
				\|b(t,x)-b(t,y)\|_{\mathbb{H}} +  \|\sigma(t,x)-\sigma(t,y)\|_{\tilde{\mathbb{H}}}  \leq  C_{b, \sigma, T} \left( \int_{0}^t \| x(s)-y(s)\|_{\mathbb{H}}^p \mu (ds) \right)^{\frac1p}.
				\label{eq:uniformLipscht2}
			\end{equation}
			\item[(ii)] The coefficient functions $b, \sigma$ are $\gamma$-H\"older in $t$ for some $0< \gamma\leq 1$, uniformly in $x$ in the following sense : there exists $c_T>0$ such that \(\forall \, t, s \in \mathbb{T}, \; \forall \, x \in \mathbb{X}, \) 
			\begin{equation}
				\|  b\left(t, x\right) - b\left(s, x\right) \|_{\mathbb{H}} +
				\|\sigma\left(t, x\right) - \sigma\left(s, x\right)\|_{\tilde{\mathbb{H}} }  
				\leq c_T \left(1 + \Big( \int_{0}^t \| x(s)\|_{\mathbb{H}}^p\, \mu (ds) \Big)^{\frac1p}\right) |t - s|^{\gamma}.
				\label{eq:coefficient_hoelder2}
			\end{equation}
			\item[(iii)] Assume furthermore that the kernels $K_i$, $i=1,2$,  satisfy  the continuity condition $({\cal K}^{cont}_{\theta} )$ (see~\eqref{eq:contK}) for some $\theta\in (0,1]$  and either the integrability condition~$({\cal K}^{int}_{\beta})$  (see~\eqref{eq:KisL^2}) for some $\beta>1$ or {that  $(\widehat {\cal K}^{cont}_{\widehat \theta})$} holds for some $\widehat\theta\in (0,1]$. 
		\end{enumerate}
	\end{assumption}
	Equivalently, Assumptions~\emph{(i)} and~\emph{(ii)} can be compactly reformulated through a
	joint time-space H\"older-Lipschitz continuity condition.
	More precisely, we assume that the coefficients $b$ and $\sigma$ satisfy
	{\small
		\begin{align}
			& \nonumber ({\cal LH}_{\gamma})\quad  \exists\, C_{b,\sigma}< +\infty,\; \forall\, s,t\!\in \mathbb{T}, \; \forall\, x,\,y \!\in \mathbb{X}\quad  \|b(t,y)-b(s,x)\|_{\mathbb{H}} +\|\sigma(t,y)-\sigma(s,x)\|_{\tilde{\mathbb{H}} },\\
			\label{eq:HolLipbsig2} &  \le C_{b,\sigma}\left(\Big(1+\Big( \int_{0}^s \| x(u)\|_{\mathbb{H}}^p\, \mu (du) \Big)^{\frac1p}+\Big( \int_{0}^t \| y(u)\|_{\mathbb{H}}^p\, \mu (du) \Big)^{\frac1p}\Big)|t-s|^{\gamma} + \Big( \int_{0}^t \| x(u)-y(u)\|_{\mathbb{H}}^p \mu (du) \Big)^{\frac1p}\right) 
		\end{align}
	}
	for some $\gamma\in (0,1]$. Note that if there is no dependence on time, then \(\gamma=1\).
%$({\cal LH}_{\gamma})$ in~\eqref{eq:HolLipbsig2}	{assump:VoltEulerScheme_equiv} 
%	\begin{assumption}\label{assump:VoltEulerScheme_equiv} 
%		We consider the following:
		
%		\medskip
%		\noindent {(i)}
%		Assume $b$ and $\sigma$ satisfy the following time-space H\"older-Lipschitz continuity assumption
%		{\small
%			\begin{align}
%				& \nonumber ({\cal LH}_{\gamma})\quad  \exists\, C_{b,\sigma}< +\infty,\; \forall\, s,t\!\in \mathbb{T}, \; \forall\, x,\,y \!\in \mathbb{X}\quad  \|b(t,y)-b(s,x)\|_{\mathbb{H}} +\|\sigma(t,y)-\sigma(s,x)\|_{\tilde{\mathbb{H}} },\\
%				\label{eq:HolLipbsig2} &  \le C_{b,\sigma}\left(\Big(1+\big( \int_{0}^s \| x(u)\|_{\mathbb{H}}^p\, \mu (du) \big)^{\frac1p}+\big( \int_{0}^t \| y(u)\|_{\mathbb{H}}^p\, \mu (du) \big)^{\frac1p}\Big)|t-s|^{\gamma} + \Big( \int_{0}^t \| x(u)-y(u)\|_{\mathbb{H}}^p \mu (du) \Big)^{\frac1p}\right) 
%			\end{align}
%		}
%		for some $\gamma\in (0,1]$. 
		
%		\medskip
%		\noindent {(ii)} Assume furthermore that the kernels $K_i$, $i=1,2$,  satisfy  the continuity condition $({\cal K}^{cont}_{\theta} )$ (see~\eqref{eq:contK}) for some $\theta\in (0,1]$  and either the integrability condition~$({\cal K}^{int}_{\beta})$  (see~\eqref{eq:KisL^2}) for some $\beta>1$ or {that  $(\widehat {\cal K}^{cont}_{\widehat \theta})$} holds for some $\widehat\theta\in (0,1]$. 
%	\end{assumption} 
	
	\begin{Proposition}\label{propXtilde} $\forall\,n \in \mathbb{N}$ and \(h=\frac{T}{n}\), let $(\bar{X}_{t}^h)_{t\in[0,T]}$ be the genuine interpolated $K$-integrated Euler scheme defined by \eqref{eq:def_continuous_2}. Then, under Assumption~\ref{assump:VoltEulerScheme}, or equivalently,  under condition $({\cal LH}_{\gamma})$ in \eqref{eq:HolLipbsig2}, together with
		$({\cal K}^{cont}_{\theta})$ (see~\eqref{eq:contK}) and either
		$({\cal K}^{int}_{\beta})$ (see~\eqref{eq:KisL^2}) or
		$(\widehat{\cal K}^{cont}_{\widehat{\theta}})$ (see~\eqref{eq:contKtilde}),
		the following properties hold.
		\begin{enumerate}
			\item[$(a)$] {\sc Property~{\bf 1}.} {\em $L^p$-integrability and pathwise continuity of $\bar X_t^h$, $p>0$}: The genuine interpolated $K$-integrated Euler scheme satisfies $\bar X_{t}^h\!\in L^p(\P)\; \forall\, t\in\mathbb{T}$ and has a continuous modification.
			
			\item[$(b)$] {\sc Property~{\bf 2}.} {\em Moment control of $\bar X_t^h$ (for $p\ge 2$)}. % $p\ge 2$
			 For every \( p \ge 2 \), we have
			\[
			\sup_{0 \le s \le T} \Big\| \|\bar X_s^h\|_{\mathbb{H}} \Big\|_p
			\le K\,\left(1+\Big\| \left\| x_0 \right\|_{T} \Big\|_p\right),
			\]
			where \( K \ge 2, \) is a constant depending on \( p, b, \sigma, T \).
			\item [$(c)$] {\sc Property~{\bf 3}.} 
			{\em  Control of the increments $\left\| \|\bar X_t^h-\bar X_s^h\|_{\mathbb{H}} \right\|_p$  (for $p\ge 2$)}. % 
			There exists a positive real constant  $C_{b,\sigma,p,T,K_1,K_2}$ not depending on $n$ such that for every $p\ge 2$
			\[
			\forall\, s,\, t\!\in [0,T], \quad  \left\| \|\bar X_t^h-\bar X_s^h\|_{\mathbb{H}} \right\|_p\le C_{b,\sigma,p,T,K_1,K_2}\Big(1 + \big\| \|x_0 \|_{T} \big\|_{p}\Big) |t-s|^{\delta \wedge \theta \wedge \theta^*}.
			\]
			where $\theta^*$ is given by Equation~\eqref{eq:def_theta_star}.
			\item [$(d)$] {\sc Property~{\bf 4}.} {\em % $L^p$-
				Rate of convergence of $\big\|\| X_t-\bar X_t\big\|_{\mathbb{H}}\|_p$  at fixed time $t$ for $p>p_{eu}:= \frac{1}{\theta}\vee \frac{1}{\theta^*}\vee\frac{1}{\delta}$}. % 
			\[
			\forall\, t\!\in [0,T], \quad  \big\|\| X_s-\bar X_s^h\big\|_{\mathbb{H}}\|_p\le \bar C\Big(1 + \big\| \|x_0 \|_{T} \big\|_{p}\Big) \left(h^{\gamma}+h^{\theta\wedge\widehat\theta}\right) \;\text{with}\; \bar C= \bar C _{K_1,K_2, \beta, b,\sigma, p,T } 
			\]
		\end{enumerate}
	\end{Proposition}
	\noindent {\bf Remark:}
		The relations ~\eqref{eq:Lpincrements2} and ~\eqref{eq:Holderpaths2} of Theorem \ref{prop:pathkernelvolt2}  can easily be derived from  {\sc Property~{\bf 3}} and {\sc Property~{\bf 4}} of Proposition \ref{propXtilde}. In fact, for $p\ge 2$, owing to {\sc Property~{\bf 3}}, there exists a constant $C > 0$ (independent of $n$) such that for all $n \ge 1$, the genuine interpolated $K$-integrated Euler scheme $\bar X^h$ with step size $\frac{T}{n}$ satisfies the H\"older regularity:
		\[
		\left\| \|\bar X_t^h-\bar X_s^h\|_{\mathbb{H}} \right\|_p\le C\Big(1 + \big\| \|x_0 \|_{T} \big\|_{p}\Big) |t-s|^{\delta \wedge \theta \wedge \theta^*}, \quad \forall\, s, t \in [0,T].
		\]
		For $p p>p_{eu}= \frac{1}{\theta}\vee \frac{1}{\theta^*}\vee\frac{1}{\delta}> 2 $, owing to {\sc Property~{\bf 4}}, the same genuine interpolated $K$-integrated Euler scheme converges to $X$ in the $p$-norm, i.e., $ \left\| \|X_t - \bar X_t^h\|_{\mathbb{H}} \right\|_p \to 0$ as $n \to \infty$, for all $t \in [0,T]$. Consequently, as $n \to 0$, for $p>p_{eu}$, we obtain:
		\[
		\left\| \|X_t- X_s\|_{\mathbb{H}} \right\|_p \le C\Big(1 + \big\| \|x_0 \|_{T} \big\|_{p}\Big) |t-s|^{\delta \wedge \theta \wedge \theta^*}, \quad \forall\, s, t \in [0,T].
		\]
		Next, let $a \in (0, \theta \wedge \theta^* \wedge \delta)$ and $p > \frac{1}{\theta \wedge \theta^* \wedge \delta-a} \vee p_{eu} $ so that $a < \theta \wedge \theta^* \wedge \delta - \frac{1}{p}$. By tracking constants in the proof of Kolmogorov's criterion (e.g., see~\cite[Chapter 2, p.26]{RevuzYor}), one concludes that~\eqref{eq:Holderpaths2} holds for such $p$.\\
	% \vspace{-.58cm}
	%  Finally, the restrictions on \(p\) in Proposition~\ref{propXtilde}, which are sufficiently large, are bypassed in the Main Theorem~\ref{thm:Eulercvgce2} below, thanks to the splitting Lemma established in Lemma~\ref{lem:gap}.
	 Finally, the sufficiently large restrictions on \(p\) in Proposition~\ref{propXtilde} are bypassed in the main Theorem~\ref{thm:Eulercvgce2} below, thanks to the splitting Lemma established in Lemma~\ref{lem:gap}.
	\subsection{The convergence rate of the genuine interpolated $K$-integrated Euler scheme}\label{subsec:Cvgce_Euler}
	
	The goal of this section is to prove a convergence result for the interpolated $K$-integrated Euler scheme. Let $T>0$ and let $p\!\in (0, +\infty)$ being fixed. 
	
	\begin{Theorem}[Convergence Rate of the genuine interpolated $K$-integrated Euler scheme]\label{thm:Eulercvgce2}  
		Let $T>0$ and let $p\!\in (0, +\infty)$ being fixed. Assume $b$ and $\sigma$ satisfy the time-space H\"older-Lipschitz continuity assumption \(({\cal LH}_{\gamma})\) in equation~\eqref{eq:HolLipbsig2}. Assume the kernels $K_i$, $i=1,2$,   satisfy %the integrability condition~$({\cal K}^{int}_{\beta})$  (see~\eqref{eq:KisL^2}) for some $\beta>1$ and  
		the continuity condition $({\cal K}^{cont}_{\theta} )$ (see~\eqref{eq:contK}) for some $\theta\in (0,1]$. Assume furthermore
		{that  $(\widehat {\cal K}^{cont}_{\widehat \theta})$} holds for some $\widehat\theta\in (0,1]$.
		Let $(X_t)_{t \in [0,T]}$ be the unique strong solution to the path-dependent SVIE~\eqref{eq:pthvolterra}, and let $(\bar{X}_t^h)_{t \in [0,T]}$ be the process defined by \eqref{eq:def_continuous_2} i.e. $\bar X^h$ denotes  the  genuine interpolated $K$-integrated  Euler scheme of~\eqref{eq:pthvolterra} with time step $h = \frac Tn$, then $\bar X^h$ has a pathwise continuous modification.
		
		There exists a real constant $C= C_{K_1,K_2,\beta, p,b,\sigma,T} \!\in (0, +\infty)$  (not depending on $n$) such that, for every $n\ge 1$,  
		\begin{equation}\label{eq:incrementXbar}
			\forall\, s,\, t\!\in [0, T], \qquad 
			%\big\|X_t - X_s\big\|_p + 
			\big\|\|\bar X^h_t -\bar X^h_s\|_{\mathbb H}\big\|_p \le C\Big(1 + \big\| \|x_0 \|_{T} \big\|_{p}\Big) |t-s|^{\delta \wedge \theta \wedge \theta^*}
		\end{equation}
		and
		%	\begin{equation}\label{eq:Xt-barX_t}
			%		\max_{k=0,\ldots,n}\big\|\|X_{t_k} -\bar X_{t_k}\|_{\mathbb H}\big\|_p \le \sup_{t\in [0,T]}\big\|\|X_t -\bar X_t\|_{\mathbb H}\big\|_p\le C\left( 1 + \big\| \|X_0 \phi\|_{T} \big\|_{p} \right) \big(\tfrac Tn\big)^{\gamma\wedge \theta \wedge  \widehat\theta}.
			%	\end{equation}
		\begin{equation}\label{eq:Xt-barX_t}
			\max_{k=0,\ldots,n}\big\|\|X_{t_k} -\bar X_{t_k}^h\|_{\mathbb H}\big\|_p \le \sup_{t\in [0,T]}\big\|\|X_t -\bar X_t^h\|_{\mathbb H}\big\|_p\le C\Big(1 + \big\| \|x_0 \|_{T} \big\|_{p}\Big) \left(\big(\tfrac Tn\big)^{\gamma}+\big(\tfrac Tn\big)^{\delta \wedge \theta \wedge \theta^* }\right).
		\end{equation}
		Moreover, there exists for every $\varepsilon\in(0,1)$  a real constant $C_{\varepsilon}= C_{\varepsilon, K_1,K_2,\beta,\gamma, p,b,\sigma,T} \!\in (0, +\infty)$ such that 
		\begin{equation}\label{eq:sup|Xt-barX_t|}
			\Big \|\max_{k=0,\ldots,n}\|X_{t_k} -\bar X_{t_k}^h\|_{\mathbb H} \Big\|_p \le \Big\|\sup_{t\in [0,T]}\|X_t -\bar X_t^h\|_{\mathbb H} \Big\|_p\le C_{\varepsilon}\Big(1 + \big\| \|x_0 \|_{T} \big\|_{p}\Big)\big(\tfrac Tn\big)^{(\delta \wedge\gamma\wedge\theta\wedge \widehat\theta) (1-\varepsilon)}.
		\end{equation}
	\end{Theorem}
	\noindent {\bf Remark} 1. If $(\widehat {\cal K}^{cont}_{\widehat \theta})$ is not satisfied, but the integrability condition~$({\cal K}^{int}_{\beta})$  (see~\eqref{eq:KisL^2}) holds for some $\beta>1$, the above theorem remains valid by replacing {\em mutatis mutandis} $\widehat \theta$ by $\frac{\beta-1}{2\beta}$ in all three inequalities.
	2. In particular, when considering a drift $b$ and volatility $\sigma$ that are either autonomous or Lipschitz in time, then $\gamma = 1$ in the above relations~\eqref{eq:Xt-barX_t}--\eqref{eq:sup|Xt-barX_t|}
		
	\medskip 
	For clarity and conciseness, the proofs of Proposition~\ref{propXtilde}, Theorem~\ref{thm:Eulercvgce2} and Corollary~\ref{corol:Eulercvgce2} are postponed to Appendices~\ref{app:B2}, where the
	main technical results are presented.
	
	\medskip 
	From Definition~\ref{def:discretization_scheme}, we introduce a\textit{ pathwise continuous piecewise affine extension} of the discrete time Euler scheme $(\bar{X}^h_{t_k})_{0 \leq k \leq n}$, denoted by $\hat{X}^h = (\hat{X}_t^h)_{t \in [0,T]}$, defined by $\hat{X}^h := \iota_n(\bar{X}^h_{t_0 : t_n})$. We then obtain the following convergence result:
	
	\begin{Corollary}\label{corol:Eulercvgce2}
		Let $p \in [1,+\infty)$ be such that $\|\|x_0\|_{\sup}\|_{p} < +\infty$ and assume that the
		assumptions of Theorem~\ref{thm:Eulercvgce2} (i.e. Assumption ~\eqref{assump:VoltEulerScheme})  are satisfied. For $n$ being an integer with $h:=\frac Tn$, suppose that 
		$\big(\bar X_{t_k^n}^h\big)_{k=0,\ldots,n}$ is the discrete time interpolated $K$-integrated Euler scheme in~\ref{def:discretization_scheme}.  Then, its pathwise continuous piecewise affine extension denoted $\hat{X}^h = (\hat{X}_t^h)_{t \in [0,T]}$, defined by $\hat{X}^h := \iota_n(\bar{X}^h_{t_0 : t_n})$ is such that:
		\[
		\Big\|\big\| X - \iota_n\big((\bar X_{t_k}^h)_{k=0,\ldots,n}\big) \big \|_{T} \Big\|_p\to 0\quad \mbox{ as }\quad n\to +\infty.
		\]
	\end{Corollary}
		  	      \medskip
	\noindent {\bf Proof of Corollary \ref{corol:Eulercvgce2}:}
	We write,  
	 {\small
	\begin{align*}
		\Big\|\big\| X - \iota_n\big((\bar X_{t_k}^h)_{k=0,\ldots,n}\big) \big \|_{T} \Big\|_p&\leq \Big\|  \big \| X-\iota_n\big( (X_{t^n_k})_{k=0,\ldots, n}\big)\big\|_{T}\Big\|_p + \left\Vert \;\| \iota_n\big((\bar X_{t^n_k}^h)_{k=0,\ldots,n}\big) -\iota_n\big((X_{t^n_k})_{k=0,\ldots,n}\big) \|_{T}\;\right\Vert_p
	\end{align*}
	}
	\noindent As a consequence of Lemma \ref{interpolatorprop} (b) and (c), we have:
	\[
	\big\| \iota_n\big((\bar X_{t^n_k}^h)_{k=0,\ldots,n}\big) -\iota_n\big((X_{t^n_k})_{k=0,\ldots,n}\big)\big \|_{T} = \max_{k=0,\ldots,n} \|\bar X_{t^n_k}^h-X_{t^n_k}\|_{\mathbb{H}} 
	\]  
	Let $p>1$, calling upon Theorem~\ref{thm:Eulercvgce2}, we hence get that 
	\begin{align*}
		\Big\|\big\| \iota_n\big((\bar X_{t^n_k}^h)_{k=0,\ldots,n}\big) -\iota_n\big((X_{t^n_k})_{k=0,\ldots,n}\big)\big \|_{T} \Big\|_p&=\big \| \max_{k=0,\ldots,n} \|\bar X_{t^n_k}^h-X_{t^n_k}\|_{\mathbb{H}}  \big\|_p \\
		&\le C\Big(1 + \big\| \|x_0 \|_{T} \big\|_{p}\Big)\left(\big(\tfrac Tn\big)^{\gamma}+\big(\tfrac Tn\big)^{\delta \wedge \widehat \theta }\right) \to 0\; \mbox{ as }\; n\to +\infty.
	\end{align*}
	Finally, we know that the uniform continuity modulus $w(\Xi, \delta)=\sup_{0\le s\le t\le (s+\delta)\wedge T}|\Xi(t)-\Xi(s)|$ of a function $\Xi\!\in {\cal C}([0,T], \mathbb{H})$ converges to $0$ as $\delta\to 0$.  
	$ \big \| X-\iota_n\big( (X_{t^n_k})_{k=0,\ldots, n}\big)\big\|_{T} \le  w\big(X, \tfrac Tn\big)\le 2\big \| X\big\|_{T}
	$ and according to equation~\eqref{eq:supLpbound} of Theorem~\ref{prop:pathkernelvolt}, $\big\| \|X\|_{T}\big\|_p <+\infty$, it follows by dominated convergence that  
	$
	\Big\|  \big \| X-\iota_n\big( (X_{t^n_k})_{k=0,\ldots, n}\big)\big\|_{T}\Big\|_p\to 0\quad \mbox{ as }\quad n\to +\infty.
	$ \\
	Combining above two yields  the convergence,
	$
	\quad \Big\|\big\| X-  \iota_n\big((\bar X_{t^n_k}^h)_{k=0,\ldots,n}\big) \big \|_{T} \Big\|_p\to 0\quad \mbox{ as }\quad n\to +\infty.
	$ \hfill$\Box$
	
	\section{Applications and Numerical Simulations with Fractional Kernels}
	In dimension 1, we focus on  the special case of a {\em scaled} convolutive path-dependent stochastic Volterra equation associated to a convolutive kernel $K:\R_+\to \R_+$ satisfying~\eqref{eq:contKtilde} and ~\eqref{eq:contK} 
	\begin{equation}\label{eq:Volterrameanrevert}
		X_t = X_0 +\int_0^t K_\alpha(t-s)b(s,X^s_\cdot)ds + \int_0^t K_\alpha(t-s)\sigma(s,X^{s}_\cdot)dW_s, \quad X_0\perp\!\!\!\perp W.
	\end{equation}
	where $K_\alpha$ is a fractional kernel of Example~\ref{Ex:Kernels}, that is $K_\alpha(t)=K_{1,\alpha,0}(t) = \frac{t^{\alpha-1}}{\Gamma(\alpha)}, t \in [0,T], \alpha>0$ and the drift $b(t,x) $ and diffusion coefficient $\sigma(t,x) $ are Lipschitz continuous in $x\in \mathbb{X}$,  uniformly in $t\!\in \mathbb{T}$. 
	Note that, if \(\alpha\in(0,1)\), then the kernel \(K_\alpha\) is singular.
	%	One can check (see, e.g. \cite[Example~2.1]{RiTaYa2020}) that 
	Following Example~\ref{Ex:Kernels}, the assumptions of the above theorems (assumption~\ref{assump:kernelVolterra} or conditions $(\widehat {\cal K}^{cont}_{\widehat \theta})$, $\big({\cal K }^{int}_{\beta}\big)$ and $({\cal K}^{cont}_{\theta})$)  are satisfied with $\beta \in \bigl(1,\frac{1}{2(\alpha-1)^{-}}\bigr)$ and $\theta = \widehat \theta = \min\bigl( \alpha-\frac12,\; 1\bigr).$ Hence, if $\alpha \in \bigl(\tfrac12, \tfrac32\bigr)$, one has \(\frac{\beta-1}{2\beta} = \tfrac12 - \tfrac{1}{2\beta} \leq \tfrac12 - (\alpha-1)^{-} \)  and $\theta = \widehat \theta = \alpha-\frac12.$
		As a consequence, considering $\big({\cal K }^{int}_{\beta}\big)$ and $({\cal K}^{cont}_{\theta})$, the final rate of convergence is $\mathcal{O}\!\left(\left(\tfrac{T}{n}\right)^{\gamma \wedge (\alpha - \tfrac12)}\right)$ at a fixed time $t$, and $\mathcal{O}\!\left(\left(\tfrac{T}{n}\right)^{(\gamma \wedge (\alpha - \tfrac12))(1-\varepsilon)}\right)$ for every $\varepsilon \in (0,1)$,
		uniformly over the interval $[0,T]$.
			Therefore, the Euler scheme $\bar X^h$ converges to $X$ in $L^p(\mathbb{P})$ with a strong rate of order \(\gamma \wedge (\alpha - \tfrac12),\)
		where $\gamma \in (0,1)$ depends on the time regularity of the drift and volatility coefficients.
		
		\medskip
		\noindent {\bf Remark: } In the settings of Example~\ref{eq:pathSDE}, if the Kernel \(K\) is the fractional kernel \(K_\alpha\) of Equation~\ref{eq:Volterrameanrevert} and the path-dependent drift and diffusion coefficients $b$ and $\sigma$ are of the form\eqref{eq:b_sigma_general}, then
	%	If the path-dependent drift and diffusion coefficients $b$ and $\sigma$ are of the form
	%	\begin{equation}\label{eq:b_sigma_general}
	%		\sigma(t,x) = G\!\big(t, (\widetilde K \star x)_t\big)
	%		= G\!\Big(t, \int_0^t \widetilde K(t-s)\,x(s)\,ds \Big),
	%		\qquad t \in [0,T], \ x \in \mathbb{X},
	%	\end{equation}
	%	where $G:[0,T]\times \mathbb{R} \to \mathbb{R}$ is Borel measurable, Lipschitz in $x$ uniformly in $t$, and uniformly bounded at $x=0$, and where $\widetilde K$ is the $\rho$-pseudo-inverse co-kernel of $K$ (see \cite[Definition 3.1]{Bonesini2023Volterra}), then \eqref{eq:b_sigma_general} corresponds to a particular case of the framework in \cite{Bonesini2023Volterra}.  
	by \cite[Theorem 5.3]{Bonesini2023Volterra}, the genuine interpolated \(K-\)integrated Euler scheme $\bar X^h$ converges to $X$ in $L^p(\mathbb{P})$ with strong rate of order $\gamma \wedge \tfrac12$.
		\subsection{Numerical Illustration} 
		The discrete time interpolated $K$-integrated Euler scheme $(\bar{X}_{t_k^n})_{0 \leq k \leq n}$ in equation~\eqref{eq:discretescheme} for the path-dependent Volterra Integral Equations equation~\eqref{eq:Volterrameanrevert} is given by:
		{\small
		\begin{align*}
			\bar X_{t_{k}^n} &= X_0 + \frac{1}{\Gamma(\alpha)}\sum_{\ell=1}^{k}b_{\ell-1}(t_{\ell-1}^n, \bar{X}_{t_0^n:t_{\ell-1}^n}) \int_{t^n_{\ell-1}}^{t_\ell^n} (t^n_k-s)^{\alpha-1}ds + \frac{1}{\Gamma(\alpha)}\sum_{\ell=1}^{k} \sigma_{\ell-1}(t_{\ell-1}^n, \bar{X}_{t_0^n:t_{\ell-1}^n})\int_{t^n_{\ell-1}}^{t^n_\ell} (t^n_k-s)^{\alpha-1}dW_s \\ &=
			  X_0 +\frac{\big(\frac{T}{n}\big)^{\alpha}}{\Gamma(\alpha+1)}\sum_{\ell=1}^{k} b_{\ell-1}(t_{\ell-1}^n, \bar{X}_{t_0^n:t_{\ell-1}^n}) \big[(k-l+1)^{\alpha}-(k-l)^{\alpha}\big] +\sum_{\ell=1}^{k}\sigma_{\ell-1}(t_{\ell-1}^n, \bar{X}_{t_0^n:t_{\ell-1}^n}) G^n_{k\ell}.
		\end{align*}
		}
		Here, $b_\ell$ and $\sigma_\ell$ are defined in~\eqref{defbmsigmal}, 
		$G^n$ is defined in~\eqref{eq:GaussianVect} and can be efficiently simulated using
		Proposition~\ref{prop:EfficientSimul}.
		One derives from~\eqref{eq:VCV1} that for every \( \ell=1,\ldots,n,\) 
		\begin{equation}\label{eq:VCV1}
			\Sigma^{n,\ell} = \left[ \frac{T}{n}\int_0^{1} K_\alpha(\frac{T}{n}(i-u)) K_\alpha(\frac{T}{n}(j-u)) \, du \right]_{1 \leq i,j \leq n-\ell+1}=  \big(\tfrac Tn\big)^{2\alpha-1}\frac{1}{\Gamma(\alpha)^2}\big(\Sigma_{ij}\big)_{i,j=1:n-\ell+1}.
		\end{equation}
	where $\Sigma$ is defined by 
	\[
	\Sigma_{ij} = \int_0^1 (i-1+v)^{\alpha-1}(j-1+v)^{\alpha-1}dv,\quad i\neq j\ge 1, \quad  \Sigma_{ii}= \frac{1}{2\alpha-1}\big(i^{2\alpha-1}-(i-1)^{2\alpha-1}\big), \, i\ge 1,
	\]
	\begin{Example}
		We consider path-dependent drift and diffusion coefficients of the form
		\begin{equation}\label{eq:sigma_general}
		b(t,x)= \mu_0-\lambda x(t) + \int_0^t x(t - u)\, \nu(du), \quad 	\sigma(t,x)
			=
			G\!\left(
			\int_0^t \Phi\big(x(u)\big)\,\mu(du)
			\right),
			\quad t \in [0,T], \ x \in \mathbb{X},
		\end{equation}
		in the general delay case with $\mu_0, \lambda \in \R$ 
		where \(\mu, \nu \in \mathcal{M}\) and the function $G:\mathbb{R}\to\mathbb{R}$ is assumed to be bounded Lipschitz and
		$\Phi:\mathbb{R}\to\mathbb{R}$ Lipschitz continuous so that, the mapping $x\mapsto \sigma(t,x)$ is globally
		Lipschitz on path space, and Lipschitz continuous in time  with $\gamma=1$
		in the sense of~\eqref{eq:HolLipbsig2}.
	\end{Example}
	Then, Equation~\eqref{eq:Volterrameanrevert} has a unique solution \( (X_t)_{t\geq0} \) adapted to \( \mathcal{F}^{X_0, W}_t \), starting from \( X_0 \in L^p(\mathbb{P}), p>0 \). This follows by applying the existence Theorem~\ref{Thm:pathExistenceUniquenes} 
	to each time interval \( [0,T] \), \( T \in \mathbb{N} \), and gluing the solutions together, utilizing the uniform linear growth of the drift and \( \sigma \) in time.\\

	\noindent Moreover, by the definition of $b_\ell$ and $\sigma_\ell$ in~\eqref{defbmsigmal}, we have
	\(b_0(t_0^n, \bar{X}^h_{t_0^n:t_0^n}) = \mu_0-\lambda \bar{X}^h_{t^n_0}\), \(\sigma_0(t_0^n, \bar{X}^h_{t_0^n:t_0^n}) = G(0)\) and for every \( \ell\geq1\), a trapezoidal rule yields the approximation:
	{\small
		\begin{equation}\label{eq:NumIntegral}
		%	b_\ell(t_\ell^n, \bar{X}^h_{t_0^n:t_\ell^n}) \approx -\lambda \bar{X}^h_{t^n_\ell} + \frac{h}{2}\left(\Phi(\bar{X}^h_{t_0^n}) + \Phi(\bar{X}^h_{t_\ell^n})\right) + h \mathbf{1}_{\ell \geq2}\sum_{m=1}^{\ell-1} \Phi(\bar{X}^h_{t_m^n}), \;
		\sigma_\ell(t_\ell^n, \bar{X}^h_{t_0^n:t_\ell^n}) \approx G\!\left(\frac{h}{2}\left(\Phi(\bar{X}^h_{t_0^n}) + \Phi(\bar{X}^h_{t_\ell^n})\right) + h \mathbf{1}_{\ell \geq2}\sum_{m=1}^{\ell-1} \Phi(\bar{X}^h_{t_m^n}) \right).
		\end{equation}
	}
	\noindent Clearly, numerical computation of an integral, as in the definition of \(\sigma_\ell\) in~\eqref{defbmsigmal}, is more computationally intensive than handling sums, as in the above equation~\eqref{eq:NumIntegral}.
	
		\medskip
	The path-dependent Volterra equation~\eqref{eq:Volterrameanrevert} is driven by a pseudo-fractional Brownian motion
	with Hurst parameter $H = \alpha-\frac12 \in (0,1]$, which corresponds to the standard framework of rough models
	(see~\cite{GatheralJR2018,Pages2024}) when $H \in (0,\tfrac12)$, or of long-memory models
	(see~\cite{ComteR1998, EGnabeyeu2025}) when $H \in (\tfrac12,1)$. In particular, when considering a drift $b$ and volatility $\sigma$ Lipschitz in time
	(i.e.\ $\gamma = 1$), the convergence rate is at least \(H\).
	Note that in the long-memory setting, the rates $\mathcal{O}\!\left(\left(\tfrac{T}{n}\right)^{\alpha - \tfrac12}\right)$ or $\mathcal{O}\!\left(\left(\tfrac{T}{n}\right)^{(\alpha - \tfrac12)(1-\varepsilon)}\right)$,
	for sufficiently small $\varepsilon > 0$, are faster than $	\mathcal{O}\!\left(\left(\tfrac{T}{n}\right)^{1/2}\right).$
%	In the figures~\ref{fig:trajectories}, we display sample trajectories of our path-dependent stochastic process.	Two figures are shown, corresponding to two values of $\alpha$, namely $\alpha = 0.9$ (i.e., $H = 0.4$)	and $\alpha = 0.6$ (i.e., $H = 0.1$), in order to emphasize that the smaller the parameter $\alpha$, the rougher the trajectories of the associated process $X$.
	
	\medskip
	\noindent	We work in the following settings: \(\mu \equiv \lambda_1\) is the Lebesgue measure on $(\R_+, {\cal B}or(\R_+))$
	\[
	\nu\equiv 0,\quad G(x) = \eta_1 \sqrt{c + \tanh(x)} \quad \text{and}\quad \Phi(x) = \eta_2 \sqrt{\kappa_2(x-a)^2 + \kappa_0}, \; \text{i.e.}\, \Phi^2\in \mathrm{Pol}_2(\mathbb{R})\subset \mathrm{Pol}(\mathbb{R})
	\]
	\noindent
	Here, $\kappa_2,\kappa_0,\mu_0,\lambda,\eta_1, \eta_2>0$, $c>1$, and $a\in\mathbb{R}$ are fixed constants.
	In particular, the numerical values used for the simulation of the two settings are as follows:
	$\mu_0 = 2$, $\lambda = 0.2$, $\kappa_2 = 0.384$, $a = 0.095$, 
	$\kappa_0 = 0.0025$, $x_0 = 0$, $\eta_1=\eta_2 =c= 1$.
		\begin{figure}[H]
		\centering
		\begin{minipage}{0.48\linewidth}
			\centering
			\includegraphics[width=1\linewidth]{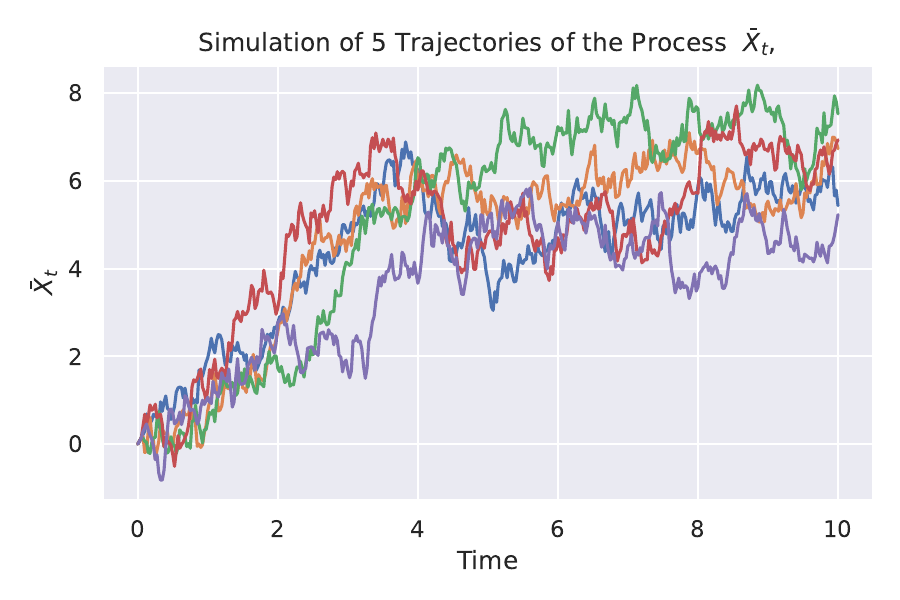
			}
			% \caption{}
		\end{minipage}%
		\hfill
		\begin{minipage}{0.48\linewidth}
			\centering
			\includegraphics[width=1\linewidth]{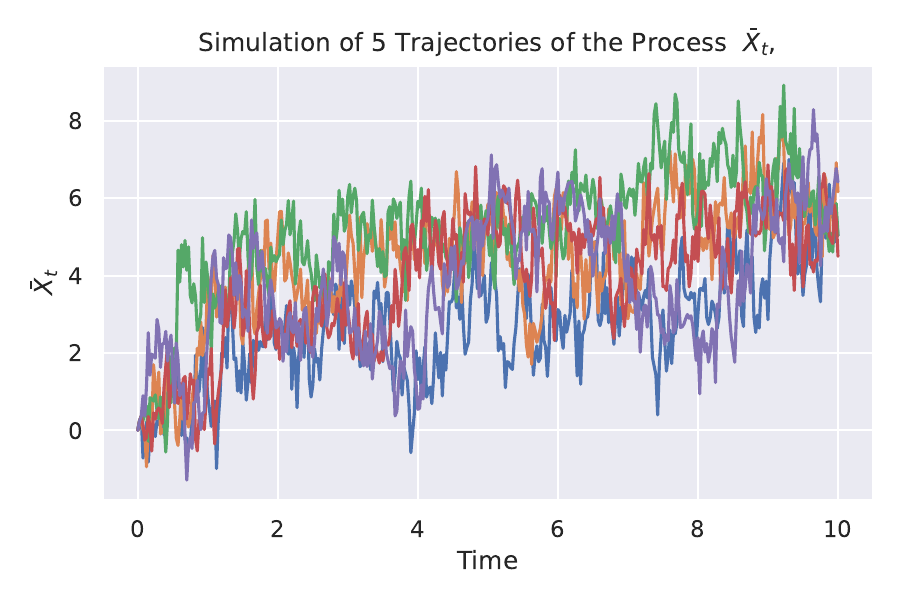}
			% \caption{}
		\end{minipage}
		\caption[Simulationsamples paths.]{\textit{	Simulation of \(5\) trajectories of the process \(X\)  for \(\alpha=0.9\) (left) and  \(\alpha=0.6\)  (right). Number of time steps: \( n = 400 \).}}
		\label{fig:trajectories}
	\end{figure}
	In the figures~\ref{fig:trajectories}, we display sample trajectories of our path-dependent stochastic process.	Two figures are shown, corresponding to two values of $\alpha$, namely $\alpha = 0.9$ (i.e., $H = 0.4$)	and $\alpha = 0.6$ (i.e., $H = 0.1$), in order to emphasize that the smaller the parameter $\alpha$, the rougher the trajectories of the associated process $X$.
	
	\vspace{-.2 cm}
	\subsection{Numerical Verification of the Strong Convergence Rate}
	We now numerically verify the convergence of the continuous-time extension~\eqref{eq:def_continuous_2} of the interpolated $K$-integrated discrete time Euler scheme introduced in Section \ref{subsec:GM}. Specifically, we examine the $L^2$ strong rate of convergence, defined as \(\sup_{t \in [0,T]} \mathbb{E}[|X_t - \overline{X}_t|^2]^{1/2}.\) As predicted by Theorem~\ref{thm:Eulercvgce2}, this rate is at least $\gamma \wedge (\alpha-\tfrac12)$
	(see, e.g., Equation~\eqref{eq:Xt-barX_t} ),
	and in the present setting where the drift and volatility coefficients are Lipschitz in time
	(i.e., $\gamma=1$), it reduces to $\alpha-\tfrac12 =: H$.

	The theoretical error in Equation~\eqref{eq:Xt-barX_t} is computed with respect to the exact solution, $X$, which is not available in closed form. Thus, an alternative method is required to establish convergence.
	To this end, we set the time grid as $t^n_k = kT/n$ for $k = 0, \dots, n$ (notice that \(t^{2n}_{2k}=t^n_k\)), and we compare two numerical schemes: $(\overline{X}^n_{t_k})_{k \in \{0, \dots, n\}}$ and $(\overline{X}^{q \cdot n}_{t_{q \cdot k}})_{k \in \{0, \dots, n\}}$, where the refinement ratio $q$ is a small integer relative to $n$ (we take $q = 2$). The superscripts $n$ and $q \cdot n$ emphasize the grid dimension.
	
	However, to measure strong convergence, both schemes must be constructed on the same probability space, i.e. driven by the same Brownian motion.
	Empirically, we show that there exists a constant $C$ such that for all $k \in \{0, \dots, n\}$ and sufficiently many values of $n$:
	\begin{equation}\label{eq:estimateScheme}
	\left\||\overline{X}^n_{t^n_k} - \overline{X}^{2n}_{t^{2n}_{2k}}|\right\|_2 \leq C \big(\frac{T}{n}\big)^{\alpha-\tfrac12 } =: C_T \big(\frac{1}{n}\big)^{H}.
	\end{equation}
	This inequality numerically guarantees the announced strong convergence rate. Indeed, since $\overline{X}^n \to X$ in $L^2$ as $n \to \infty$, we obtain in particular for every $k \in \{0, \dots, n\}$ $X_{t_k^n}
	=
	\lim_{i \to \infty} \bar X^{2^i n}_{t_k^n}
	\quad \text{in } L^2$ so that for \( i \ge 0\), applying estimate~\eqref{eq:estimateScheme} with the substitution $n \mapsto 2^i n$, yields \(\Big\||\bar X^{2^{i+1} n}_{t_k^n} - \bar X^{2^i n}_{t_k^n}|\Big\|_2
	\le
	\frac{C_T}{(2^i n)^H}.\) Consequently,
	 {\small
	\[
	\Big\||X_{t_k^n} - \overline{X}^n_{t_k^n}|\Big\|_2 \leq \sum_{i=0}^{\infty} \Big\||\overline{X}^{2^i n}_{t_k^n} - \overline{X}^{2^{i+1} n}_{t_k^n}|\Big\|_2 \leq C_T \sum_{i=0}^{\infty} \frac{1}{2^{iH}n^H} = \frac{C_T}{n^H} \sum_{i=0}^{\infty} \frac{1}{2^{iH}} = \frac{C_T}{n^H} \frac{1}{1 - 2^{-H}}.
	\]
	 }
	\noindent
	Note that the  inequality in Equation~\eqref{eq:estimateScheme} is necessary for numerical convergence, as it ensures that:
		\[
		\max_{k \in \{0, \dots, n\}} \mathbb{E}[|X_{t^n_k} - \overline{X}^n_{t^n_k}|^2] \leq \frac{C_T^2}{n^{2H}}, \quad \max_{k \in \{0, \dots, n\}} \mathbb{E}[|X_{t^{2n}_{2k}} - \overline{X}^{2n}_{t^{2n}_{2k}}|^2] \leq \frac{C_T^2}{(2n)^{2H}}.
		\]
		This implies:
		{ \small
		\[
		\frac1n \sum_{k=1}^n \left\||\overline{X}^n_{t^n_k} - \overline{X}^{2n}_{t^{2n}_{2k}}|\right\|_2 \leq \max_{k \in \{0, \dots, n\}} \left(\mathbb{E}[|\overline{X}^n_{t^n_k} - \overline{X}^{2n}_{t^{2n}_{2k}}|^2]\right)^{1/2} \leq \frac{C_T}{n^{H}} + \frac{C_T}{(2n)^{H}} = \left(1 + \frac{1}{2^{H}}\right)\frac{C_T}{n^{H}}.
		\]
	   }
	\noindent In Figure \ref{fig:rateMaxPointError}, we plot the mean error $\frac1n \sum_{k=1}^n \left\||\overline{X}^n_{t^n_k} - \overline{X}^{2n}_{t^{2n}_{2k}}|\right\|_2$ for different values of $n$, with $H \in \{0.1, 0.4\}$, and compare it with a polynomial fit to the log-error (dashed line).
	%All errors are estimated using  \(10000\) Monte Carlo simulations.
	
%	 \noindent In Figure \ref{fig:rateEndPointError}, we plot the end point error $\left\||\overline{X}^n_T - \overline{X}^{2n}_T|\right\|_2$ for different values of $n$, with $H \in \{0.1, 0.4\}$, and compare it with a polynomial fit (dashed line). The solid lines represent the observed error, and the dashed lines show the polynomial fit to the log-error.\\
%	The slope of the dashed line indicates the approximation of the strong convergence rate of the process. The interpolated $K$-integrated discrete time Euler scheme consistently adheres to, and sometimes sharper than, the predicted theoretical rate.
	
		\begin{figure}[H]
		\centering
		\begin{minipage}{0.49\linewidth}
			\centering
			\includegraphics[width=.75\linewidth]{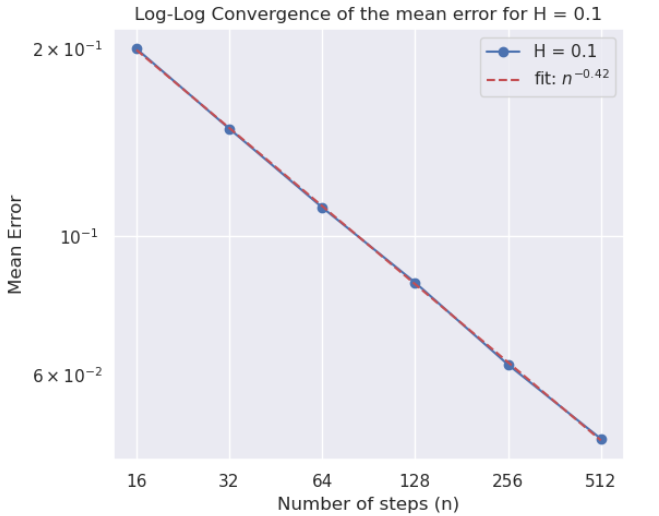
			}
			% \caption{}
		\end{minipage}%
		\hfill
		\begin{minipage}{0.49\linewidth}
			\centering
			\includegraphics[width=.75\linewidth]{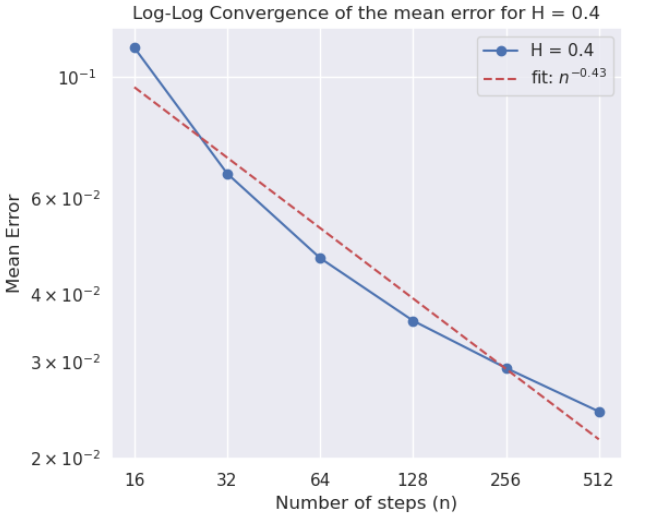}
			% \caption{}
		\end{minipage}
		\caption[Convergence rate.]{\textit{Strong convergence rate of the mean error in log-log plot with 1000 Monte Carlo iterations for $H = 0.1$ (left) and $H = 0.4$ (right) for the process $X$. %The solid lines represent the observed error, and the dashed lines show the polynomial fit to the log-error. 
			}
		}
		\label{fig:rateMaxPointError}
	\end{figure}
	\noindent In Figure \ref{fig:rateEndPointError}, we plot the end point error $\left\||\overline{X}^n_T - \overline{X}^{2n}_T|\right\|_2$ for different values of $n$, with $H \in \{0.1, 0.4\}$, and compare it with a polynomial fit (dashed line). The solid lines represent the observed error, and the dashed lines show the polynomial fit to the log-error.
\vspace{-.4cm}
	\begin{figure}[H]
		\centering
		\begin{minipage}{0.49\linewidth}
			\centering
			\includegraphics[width=.75\linewidth]{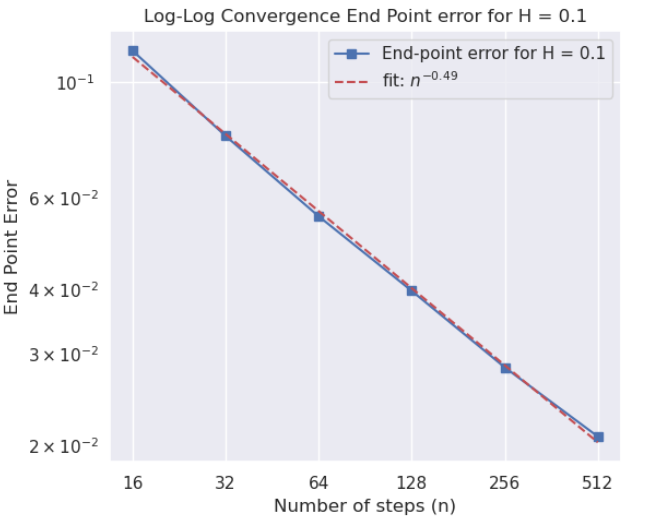
			}
			% \caption{}
		\end{minipage}%
		\hfill
		\begin{minipage}{0.49\linewidth}
			\centering
			\includegraphics[width=.75\linewidth]{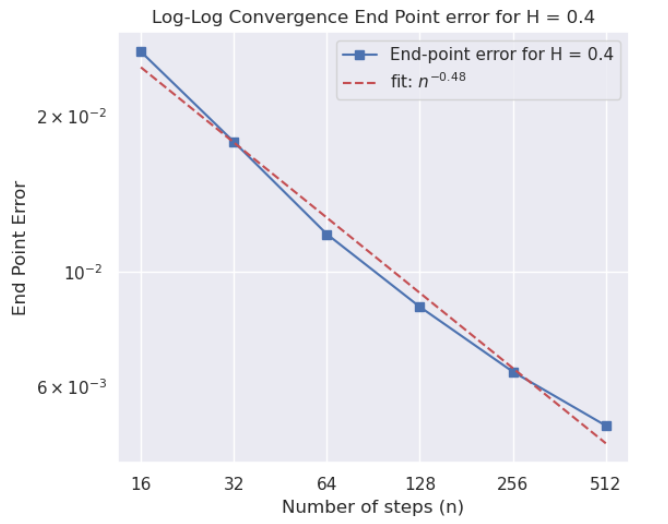}
			% \caption{}
		\end{minipage}
		\caption[Convergence rate.]{\textit{Strong convergence rate of the end point error in log-log plot with 1000 Monte Carlo iterations for $H = 0.1$ (left) and $H = 0.4$ (right) for the process $X$. }
		}\label{fig:rateEndPointError}
	\end{figure}
%	\noindent The interpolated $K$-integrated discrete time Euler scheme consistently adheres to, and sometimes sharper than, the predicted theoretical rate.

 \noindent The slope of the dashed line indicates the approximation of the strong convergence rate of the process. The interpolated $K$-integrated discrete time Euler scheme consistently adheres to, and sometimes sharper than, the predicted theoretical rate.
	\section{Proofs of the Main Results}\label{sect:ProofMainResult}
	% Proposition \ref{prop:pathExistenceUniquenes}, Theorems~\ref{Thm:pathExistenceUniquenes},~\ref{prop:pathkernelvolt} and~\ref{prop:pathkernelvolt2}}\label{sect:ProofMainResult}
	\noindent {$\rhd$ {\em Preliminaries}:}
	The proof is carried out in two steps. First,Proposition~\ref{prop:pathExistenceUniquenes} and Theorem~\ref{Thm:pathExistenceUniquenes} establish the existence and uniqueness of a solution in the Bochner space $(\mathcal{H}_{p,T}, \|\cdot\|_{p,T})$. Since these solutions are not guaranteed to be continuous, we then prove \emph{a posteriori} (in contrast to most existing results), the existence of a continuous modification and their path regularity in Theorems~\ref{prop:pathkernelvolt} and~\ref{prop:pathkernelvolt}, using either Kolmogorov's continuity criterion or the Garsia-Rodemich-Rumsey inequality; see e.g., \citep[Lemma 1.1]{GRR1070}.
	\subsection{Proof of Proposition \ref{prop:pathExistenceUniquenes} and Theorem~\ref{Thm:pathExistenceUniquenes}}
	\noindent {\bf Proof of Proposition \ref{prop:pathExistenceUniquenes}: }
	Let \( X, Y \in \mathcal{H}_{p,T} \) and define \(\mathcal{T}_c\) by \(	\mathcal{T}_c(X) := (\mathcal{T}_c(X)_t)_{0 \leq t \leq T}\) with 
	\[
	(\mathcal{T}_c(X))_t := x_0(t) + \int_{0}^{t} K_1(t,r) b(r, X^r_\cdot) \, dr + \int_{0}^{t} K_2(t,r) \sigma(r, X^r_\cdot) \, dW_r.
	\]
	\noindent {\sc Step~1} ({\em Well-definedness}).
	We need to prove that $\|\mathcal{T}_c(X)\|_{p,c,T} < \infty$  so that $\mathcal{T}_c(X) \in \mathcal{H}_{p,T}$. By the triangle inequality,
	\begin{equation}\label{eq:norm_X}
		\| \mathcal{T}_c(X) \|_{p,c,T} 
		\leq \| x_0 \|_{p,c,T} + \Big\| \int_{0}^{\cdot } K_1(\cdot,r) b(r, X^r_\cdot) \, dr  \Big\|_{p,c,T} + \Big\| \int_{0}^{\cdot} K_2(\cdot,r) \sigma(r, X^r_\cdot) \, dW_r \Big\|_{p,c,T}. 
	\end{equation}
	For the first and deterministic term,  using the triangle inequality in 
	\(\mathbb{H}\) for Bochner integrals and applying H\"older inequality with exponent \((\frac{p}{p-1},p)\) taking advantage to the the linear growth of $b$ in Lemma \ref{lm:growth} yields
	\begin{align*}
		\mathbb{E} \left[ \left\| \int_0^t K_1(t,r) b(r,X^r_\cdot) \,dr \right\|^p_{\mathbb{H}} \right] &\leq 2^{p-1}C_{b,\sigma,T}^p\, \Big( \int_0^t K_1(t,r)\, dr \Big)^{p - 1}
		\int_0^t K_1(t,r) \big( 1 + \int_0^r \mathbb{E} \left[\|X_u^r\|^p_{\mathbb{H}} \right] \mu(du) \big) dr\\
		&\leq 2^{p-1}C_{b,\sigma,T}^p\, \left( \int_0^t K_1(t,r)\, dr \right)^{p } \left( 1 + e^{pc t}  \|X^t_\cdot\|^p_{p,c,t}\right).
	\end{align*}
	Consequently, calling upon the condition 	\((\widehat {\cal K}^{cont}_{\widehat \theta})\) in ~\eqref{eq:contKtilde} , we deduce that:
	{\small
		\begin{equation}\label{eq:boundnorm1}
			\Big\| \int_{0}^{\cdot} K_1(\cdot,r) b(r, X^r_\cdot) \, dr \Big\|_{p,c,T}^p 
			\leq 2^{p-1}C^p_{b,\sigma,T} \sup_{t\in [0,T]}\left(  e^{-ct} \int_0^t u^{p\widehat{\theta}}\, \mu(du) + \int_0^t u^{p\widehat{\theta}} e^{-pc(t-u)} \| X^u_\cdot \|_{p,c,u}\, \mu(du)\right). 
		\end{equation}
	}
	\noindent
	For the martingale terms, using the Burkholder-Davis-Gundy (BDG) inequality and standard arguments (Jensen's inequality
	\hyperref[fn:first]{\footnotemark[\getrefnumber{fn:first}]} with coefficient \(\frac{p}{2}\geq1\)) %(H\"older inequality with exponent \((\frac{p}{p-2},\frac{p}{2})\)) 
	taking advantage of the the linear growth property of $\sigma$ in Lemma \ref{lm:growth}, gives sucessively:
	\begin{align*}
		&\mathbb{E} \left[ \left\| \int_0^t K_2(t,r)\sigma(r,X^r_\cdot)\,dW_r \right\|^p_{\mathbb{H}} \right] \leq C_p^{\text{BDG}}\, \left( \int_0^t K_2(t,r)^2\, dr \right)^{p/2 - 1} \int_0^t K_2(t,r)^2 \mathbb{E}\left[ \|\sigma(r,X^r_\cdot)\|^p_{\tilde{\mathbb{H}}} \right] dr\\
		&\hspace{3cm} \leq 2^{p-1}C_{b,\sigma,T}^p C_p^{\text{BDG}} \left( \int_0^t K_2(t,r)^2 dr \right)^{p/2 - 1}
		\int_0^t K_2(t,r)^2 \Big( 1 + \int_0^r \mathbb{E} \left[\|X_u^r\|^p_{\mathbb{H}} \right] \mu(du) \Big) dr\\
		& \hspace{3cm} \leq 2^{p-1}C_{b,\sigma,T}^p\, C_p^{\text{BDG}} \left( \int_0^t K_2(t,r)^2\, dr \right)^{\frac{p}{2}} \left( 1 + e^{pc t}  \|X^t_\cdot\|^p_{p,c,t}\right).
	\end{align*}
	Consequently, with the condition 	\((\widehat {\cal K}^{cont}_{\widehat \theta})\)	in ~\eqref{eq:contKtilde} , we obtain:
	{\small
		\begin{equation}\label{eq:boundnorm2}
			\Big\| \int_{0}^{\cdot} K_2(\cdot,r) \sigma(r, X^r_\cdot) \, dW_r \Big\|_{p,c,T}^p 
			\leq 2^{p-1}C^p_{b,\sigma,T}C_p^{\text{BDG}} \sup_{t\in [0,T]}\left(  e^{-ct} \int_0^t u^{p\widehat{\theta}}\, \mu(du) + \int_0^t u^{p\widehat{\theta}} e^{-pc(t-u)} \| X^u_\cdot \|_{p,c,u}\, \mu(du)\right). 
		\end{equation}
	}
	\noindent As the measure \( \mu \) is finite on \( [0,T] \), \( \int_0^T u^{p\widehat{\theta}}\, \mu(du) <\infty \). Then, if \(\| X^T_\cdot \|_{p,c,T} <\infty \), using the monotonicity of exponential decay, inequalities~\eqref{eq:boundnorm1} and~\eqref{eq:boundnorm2} together with assumption~\ref{assump:VoltEulerScheme}~(iv) implies that $\|\mathcal{T}_c(X)\|_{p,c, T} < \infty$ and thus $\mathcal{T}_c(X) \in \mathcal{H}_{p,T}$ whenever \( X \in  \mathcal{H}_{p,T} \). The inequality~\eqref{eq:boundnorm} follows by injecting~\eqref{eq:boundnorm1} and~\eqref{eq:boundnorm2} into~\eqref{eq:norm_X}.
	Moreover by Proposition~\ref{prop:pathkernelvolt} below, we will have that $\mathcal{T}_c(X)$ is continuous in $L^p$.
	
	\smallskip
	\noindent {\sc Step~2} ({\em Lipschitzianity}).
	For $0\leq t \leq T$ and $c > 0$, 
	We want to show that \( \mathcal{T}_c \) is Lipschitz with respect to the norm \(\|X\|_{p,c,T}\) in \eqref{eq:weighted_norm}.
	Let us estimate \( \|\mathcal{T}_c(X) - \mathcal{T}_c(Y)\|_{p,c,T} \). For all \( t \in [0,T] \), we have
	\begin{align*}
		\mathbb{E} \left[ \left\| (\mathcal{T}_c(X))_t - (\mathcal{T}_c(Y))_t \right\|_{\mathbb{H}}^p \right]
		&\leq 2^{p-1} \mathbb{E} \left[ \left\| \int_0^t K_1(t,r) (b(r,X^r_\cdot) - b(r,Y^r_\cdot))\,dr \right\|^p_{\mathbb{H}} \right] \\
		&\quad + 2^{p-1} \mathbb{E} \left[ \left\| \int_0^t K_2(t,r)(\sigma(r,X^r_\cdot) - \sigma(r,Y^r_\cdot))\,dW_r \right\|^p_{\mathbb{H}} \right].
	\end{align*}
	For the first and deterministic term,  using the triangle inequality in 
	\(\mathbb{H}\) for Bochner integrals and applying Jensen's inequality twice \footnote{For two measurable real-valued functions \( f \), \( g \), and \( p \geq 1 \), we have
		\begin{align*}
			\left| \int f(s)g(s) \, ds \right|^p 
			&\leq \left| \int |f(s)|^{1 - \frac{1}{p}} \cdot |f(s)|^{\frac{1}{p}} g(s) \, ds \right|^p 
			\leq \left( \int |f(s)| \, ds \right)^{p - 1} \cdot \int |f(s)||g(s)|^p \, ds.
		\end{align*}
	\label{fn:first}} or rather H\"older inequality with exponent \((\frac{p}{p-1},p)\) yields 
	\begin{align*}
		\left\| \int_0^t K_1(t,r) (b(r,X^r_\cdot) - b(r,Y^r_\cdot))\,dr \right\|^p_{\mathbb{H}}
		&\leq \left( \int_0^t K_1(t,r) \|b(r,X^r_\cdot) - b(r,Y^r_\cdot)\|_{\mathbb{H}}\, dr \right)^{p} \\
		&\leq \left( \int_0^t K_1(t,r)\, dr \right)^{p - 1}
		\int_0^t K_1(t,r) \|b(r,X^r_\cdot) - b(r,Y^r_\cdot)\|_{\mathbb{H}}^p \, dr.
	\end{align*}
	Taking expectation and using the Lipschitz assumption on \( b \), we get:
	\begin{align*}
		&\mathbb{E} \left[ \left\| \int_0^t K_1(t,r) (b(r,X^r_\cdot) - b(r,Y^r_\cdot))\,dr \right\|^p_{\mathbb{H}} \right] \\
		&\qquad\qquad\qquad \leq C_b^p\, \left( \int_0^t K_1(t,r)\, dr \right)^{p - 1}
		\int_0^t K_1(t,r) \left( \int_0^r \mathbb{E} \left[\|X_u^r - Y_u^r\|^p_{\mathbb{H}} \right] \mu(du) \right) dr.
	\end{align*}
%	Now, we exchange the order of integration owing to Fubini's Theorem \cite[Thm. 1]{Kailath_Segall},  and further details can be found in \cite[Thm. IV.65]{Protter}.(see also \cite[ Theorem 2.6]{Walsh1986}, \cite[ Theorem 2.6]{Veraar2012}), and estimate:
	Now, we exchange the order of integration owing to Fubini's Theorem \citep[Thm. 1]{Kailath_Segall}, (see also \cite[ Theorem 2.6]{Walsh1986}), and estimate:
	\begin{align*}
		&\int_0^t K_1(t,r) \left( \int_0^r \mathbb{E} \left[\|X_u^r - Y_u^r\|^p_{\mathbb{H}} \right] \mu(du) \right) dr  = \int_0^t \left( \int_{ u}^t K_1(t,r)\, \mathbb{E} \left[\|X_u^r - Y_u^r\|^p_{\mathbb{H}} \right] dr \right) \mu(du)\\
		&\qquad \leq \left( \int_{0}^t K_1(t,r)\, dr \right) \cdot \int_0^t \mathbb{E} \left[\|X_u^t - Y_u^t\|^p_{\mathbb{H}} \right]\, \mu(du)
	\end{align*}
	Noting that for every \( t \in [0,T] \), \(  e^{-pct} \, \int_0^t \mathbb{E} \left[\|X_u^t - Y_u^t\|^p_{\mathbb{H}} \right]\, \mu(du) \leq \|X^t_\cdot - Y^t_\cdot\|^p_{p,c,t} \leq  \|X - Y\|^p_{p,c,T}\), we then have:
	\begin{align}
		&\mathbb{E} \left[ \left\| \int_0^t K_1(t,r) (b(r,X^r_\cdot) - b(r,Y^r_\cdot))\,dr \right\|^p_{\mathbb{H}} \right] \leq C_b^p\, \left( \int_0^t K_1(t,r)\, dr \right)^{p } e^{pc t}  \|X - Y\|^p_{p,c,T}\label{boundI}
	\end{align}
	
	Similarly, for the martingale terms, using the Burkholder-Davis-Gundy (BDG) inequality and standard arguments taking advantage of the Lipschitz property of $\sigma$, yield
	\begin{align*}
		&\mathbb{E} \left[ \left\| \int_0^t K_2(t,r)(\sigma(r,X^r_\cdot) - \sigma(r,Y^r_\cdot))\,dW_r \right\|^p_{\mathbb{H}} \right] \leq C_p^{\text{BDG}}\, \mathbb{E} \left[ \left( \int_0^t K_2(t,r)^2 \|\sigma(r,X^r_\cdot) - \sigma(r,Y^r_\cdot)\|^2_{\tilde{\mathbb{H}}} dr \right)^{p/2} \right].
	\end{align*}
	Applying Jensen's inequality % or rather H\"older inequality with exponent \((\frac{p}{p-2},\frac{p}{2})\)
	\hyperref[fn:first]{\footnotemark[\getrefnumber{fn:first}]} with coefficient \(\frac{p}{2}\geq1\) to factor out the \( L^2 \)-norm of the kernel, then using the Lipschitz assumption on \(\sigma\) and exchanging the order of integration as before gives sucessively:
	\begin{align*}
		&\mathbb{E} \left[ \left\| \int_0^t K_2(t,r)(\sigma(r,X^r_\cdot) - \sigma(r,Y^r_\cdot))\,dW_r \right\|^p_{\mathbb{H}} \right] \notag \\
		&\hspace{3cm} \leq C_p^{\text{BDG}}\, \left( \int_0^t K_2(t,r)^2\, dr \right)^{p/2 - 1} \int_0^t K_2(t,r)^2 \mathbb{E}\left[ \|\sigma(r,X^r_\cdot) - \sigma(r,Y^r_\cdot)\|^p_{\tilde{\mathbb{H}}} \right] dr\\
		&\hspace{3cm} \leq C_{\sigma}^p C_p^{\text{BDG}} \left( \int_0^t K_2(t,r)^2 dr \right)^{p/2 - 1}
		\int_0^t K_2(t,r)^2 \left( \int_0^r \mathbb{E}\left[ \|X_u^r - Y_u^r\|_{\mathbb{H}}^p \right] \mu(du) \right) dr.
	\end{align*}
	As before, we switch the order of integration  and introducing the exponential decay weight \( e^{-pc t} \), we find:   
	\begin{align*}
		&e^{-pc t} \int_0^t K_2(t,r)^2 \left( \int_0^r \mathbb{E}\left[ \|X_u^r - Y_u^r\|_{\mathbb{H}}^p \right] \mu(du) \right) ds  = e^{-pc t} \int_0^t \left( \int_{u}^t K_2(t,r)^2\, \mathbb{E} \left[\|X_u^r - Y_u^r\|^p_{\mathbb{H}} \right] dr \right) \mu(du)\\
		&\qquad \leq \left( \int_{0}^t K_2(t,r)^2\, dr \right) \cdot e^{-pc t} \int_0^t \mathbb{E} \left[\|X_u^t - Y_u^t\|^p_{\mathbb{H}} \right]\, \mu(du) \leq  \left( \int_{0}^t K_2(t,r)^2\, dr \right)\|X - Y\|^p_{p,c,t}. 
	\end{align*}
	Combining everything, we obtain the estimates:
	\begin{align}
		\mathbb{E} \left[ \left\| \int_0^t K_2(t,r)(\sigma(r,X^r_\cdot) - \sigma(r,Y^r_\cdot))\,dW_r \right\|_{\mathbb{H}}^p \right] \leq C_{\sigma}^p\, C_p^{\text{BDG}} \left( \int_0^t K_2(t,r)^2\, dr \right)^{\frac{p}{2}} e^{pc t} \|X - Y\|_{p,c,T}^p.\label{boundII}
	\end{align}
	
	Combining both the deterministic and stochastic estimates, \eqref{boundI}--\eqref{boundII} with the condition 	\((\widehat {\cal K}^{cont}_{\widehat \theta})\)	in ~\eqref{eq:contKtilde} , we obtain:
	\begin{align*}
		\| \mathcal{T}_c( X )- \mathcal{T}_c( Y ) \|_{p,c, T}^p 
		&\leq \sup_{t\in [0,T]} \left( e^{-pc t} \int_0^t u^{p\widehat{\theta}} \cdot \left( C_{b}^p + C_{\sigma}^p\, C_p^{\text{BDG}} \right)\cdot e^{pc u} \cdot \|X - Y\|_{p,c,T}^p\, \mu(du) \right) \\
		&\leq \left( C_{b}^p+ C_{\sigma}^p\, C_p^{\text{BDG}} \right)\times \sup_{t\in [0,T]} \left( e^{-pc t} \int_0^t u^{p\widehat{\theta}} e^{pc u}\, \mu(du) \right)\cdot \|X - Y\|_{p,c,T}^p,
	\end{align*}
	so that 
	\begin{equation}\label{eq:lipschtbound}
		\| \mathcal{T}_c( X )- \mathcal{T}_c( Y ) \|_{p,c, T} 
		\leq K^\prime \cdot \sup_{t\in [0,T]} \left( \int_0^t u^{p\widehat{\theta}} e^{-pc(t-u)}\, \mu(du) \right)^\frac1p \cdot \|X - Y\|_{p,c,T}
	\end{equation}
	for some constant $K^\prime$ depending only on $p$, $T$, and the Lipschitz constants of $b$ and $\sigma$.  \hfill$\Box$
	
	\medskip
	\noindent {\bf Proof of Theorem~\ref{Thm:pathExistenceUniquenes}: }
	\smallskip
	\medskip 
	\noindent {\sc Step~1} ({\em Existence}).
	Proposition \ref{prop:pathExistenceUniquenes} implies that $\mathcal T_c$ is a Lipschitz continuous function. Thus, the image \(\mathcal{F}_c := \mathcal T_c\left( \mathcal{H}_{p,T} \right)\)
	is a closed subset of $\mathcal{H}_{p,T}$.
	Since $\|\cdot\|_{p,T}$ and $\|\cdot\|_{p,c,T}$ are equivalent, it suffices to find $c > 0$ such that $\mathcal{T}_c$ is a contraction on $(\mathcal{H}_{p,T}, \|\cdot\|_{p,c,T})$. That is, we seek $c > 0$ and $\rho < 1$ such that
	\begin{equation}\label{eq:contraction}
		\| \mathcal{T}_c( X) -\mathcal{T}_c( Y)\|_{p,c,T} \leq \rho\, \| X - Y \|_{p,c,T}, 
		\quad X,Y \in \mathcal{H}_{p,T}.
	\end{equation}
	To prove that \( \mathcal{T}_c \) is a contraction for sufficiently large \( c \), recall the estimate or the bound in equation~\eqref{eq:lipschtbound}:
	\[
	\|\mathcal{T}_c(X) - \mathcal{T}_c(Y)\|_{p,c,T}
	\leq \kappa(c)\, \|X - Y\|_{p,c,T} \quad \text{where}\quad	\kappa(c) := K^\prime \cdot \sup_{t\in [0,T]} \left( \int_0^t u^{p\widehat{\theta}} e^{-pc(t-u)}\, \mu(du) \right)^\frac1p.
	\]
	We now show that \( \kappa(c) \to 0 \) as \( c \to \infty \), which will imply that the operator \( \mathcal{T}_c \) will becomes a strict contraction on \( (\mathcal{H}_{p,T}, \|\cdot\|_{p,c,T}) \) for sufficiently large \( c \).
	Fix \( t \in [0,T] \), and define
	
	\centerline{$
		f_c^{(t)}(u) := u^{p\widehat{\theta}} e^{-pc(t - u)} \cdot \mathbf{1}_{[0,t]}(u).
		$}
	\noindent Then \( f_c^{(t)}(u) \to 0 \) pointwise on \( [0,t) \) as \( c \to \infty \), and for all \( c > 0 \), \(0 \leq f_c^{(t)}(u) \leq u^{p\widehat{\theta}}.\)
	As the measure \( \mu \) is finite on \( [0,T] \), \( u \mapsto u^{p\widehat{\theta}} \in L^1([0,T], \mu) \). Then, by the Dominated Convergence Theorem and using monotonicity of exponential decay, for each \( t \in [0,T] \), \(	\lim_{c \to \infty} \int_0^t f_c^{(t)}(u) \mu(du) = 0\) as $\mu$ is satisfies ~\eqref{eq:uniform-non-atomicity}.
	Moreover, since the integrals are bounded by \( \int_0^T u^{p\widehat{\theta}} \mu(du) \), independently of \( t \in [0,T] \), the convergence to zero is uniform in  \( t \in [0,T] \). Therefore, \(\lim_{c \to \infty} \kappa(c) = 0.\)
	Hence, there exists \( c_0 > 0 \) such that for all \( c > c_0 \), we have \( \kappa(c) < 1 \), and so for sufficiently large \( c \), the operator \( \mathcal{T}_c \) becomes a contraction on \( (\mathcal{H}_{p,T}, \|\cdot\|_{p,c,T}) \) such that the contraction condition~\eqref{eq:contraction} holds. 
	The Banach fixed-point theorem then ensures the existence and uniqueness of a fixed point in \( \mathcal{H}_{p,c,T} \) such that \( \mathcal{T}_c(X) = X \), i.e., a solution to the integral equation.
	This implies that $\mathcal{T}_c$ is a contraction mapping on a complete metric space.
	By the Banach fixed point theorem, $\mathcal{T}_c$ admits a unique fixed point \(H \in \mathcal{F}_c \subset \mathcal{H}_{p,T} \)
	and this process $H$ is the unique strong solution of the path-dependent volterra equation.
	
	\smallskip
	\noindent {\sc Step~2} ({\em Uniqueness:})
	Alternatively, let $X$ be the unique fixed point of $\mathcal{T}_c$ in $\mathcal{H}_{p,T}$. Proposition~\ref{prop:pathkernelvolt} implies that $X$ has a continuous version, which is a strong solution of~\eqref{eq:pthvolterra} on $[0,T]$. By virtue of Proposition ~\ref{prop:pathkernelvolt} and Lemma~\ref{lem:supEbound}, any continuous solution of~\eqref{eq:pthvolterra} on $[0,T]$ must be this fixed point, hence uniqueness. Since $T \geq 0$ was arbitrary, it follows that ~\eqref{eq:pthvolterra} admits a unique strong solution on $[0,\infty)$.
	
	\medskip 
	\noindent {\sc Step~3} ({\em Bound~\eqref{eq:bound_X}}).
	In fact, by Lemma \ref{lm:growth} and Jensen's inequality,
	{\small
		\[
		\mathbb{E}\left[\|\mathcal{T}_c(0)_t\|_{\mathbb{H}}^p\right] \leq 3^{p-1} \left( \mathbb{E}\left[ \| x_0(t) \|_{\mathbb{H}}^p \right] 
		+ C_{b,\sigma,T} \Big(\big( \int_0^{t} K_1(t,r)\, dr \big)^{p} +  C_p^{\text{BDG}} \big( \int_0^{t} K_2(t,r)^2\, dr \big)^{\frac{p}{2} }\Big) \right).
		\]
	}
	so that by assumption \ref{assump:kernelVolterra} (i),
	\begin{align*}
		\| \mathcal{T}_c(0) \|^p_{p,c,T} 
		&\leq \| \mathcal{T}_c(0) \|_{p,T}^p := \Big( \int_{0}^T \mathbb{E}\left[ \| \mathcal{T}_c(0)_t \|_{\mathbb{H}}^p \right] \mu(dt) \Big) \\
		&\leq 3^{p-1} \Big( \int_{0}^T \mathbb{E}\left[ \| x_0(t) \|_{\mathbb{H}}^p \right] \mu(dt) 
		+ C_{b,\sigma,T} (1 +  C_p^{\text{BDG}}) \int_{0}^T t^{p\widehat{\theta}} \mu(dt) \Big).
	\end{align*}
	Thus \(\exists \, C >0 \,\) such that
	\begin{equation}\label{eq:norm_0}
		% \text{Thus }\;\exists \, C >0 \,\;\text{ such that}\;	
		\| \mathcal{T}_c(0) \|_{p,c,T} 
		\leq \| \mathcal{T}_c(0) \|_{p,T} \leq C \left(1 + \Big( \int_{0}^T \mathbb{E}\left[ \| x_0(t) \|_{\mathbb{H}}^p \right] \mu(dt)\right)^\frac1p\Big) <\infty. 
	\end{equation}
	still owing to Lemma \ref{lm:growth} (ii).	
	\noindent $\Rightarrow$ The bound~\eqref{eq:bound_X} is obtained, by taking \( Y = 0 \), in inequality ~\eqref{eq:lipschtbound} and owing to ~\eqref{eq:norm_0} and the fact that \(\mathcal{T} (X) = X\) since \(X\) is a fixed point of the application \(\mathcal{T}\). The result being true for every \(c\geq 0.\)
	
	Note that, by taking \( Y = 0 \), inequality ~\eqref{eq:lipschtbound} implies that $\|\mathcal{T}_c(X)\|_{p,c, T} < \infty$ and thus $\mathcal{T}_c(X) \in \mathcal{H}_{p,T}$ whenever \( X \in  \mathcal{H}_{p,T} \). 
	We thus recover the well-definedness of the map 
	\(\mathcal{T}_c\), as established in the first step of the proof of Proposition~\ref{prop:pathExistenceUniquenes}.
	
	\medskip	
	\noindent  {\sc Step~4} : Existence and Uniqueness of ~\eqref{eq:pthvolterra} with $X_0 \in L^0(\Omega, {\cal F}_0, \P)$.
	\medskip
	\noindent \\
	$\rhd$ {\em Existence}. We now suppose that $x_0 \in L^0(\Omega, {\cal F}_0, \P)$. For every $k \in \N^*$, let $A_k = \{\|x_0 \|_{T} < k\}$ and $X^{(k)}$ denote the $({\cal F}_t)$-adapted continuous unique solution to~\eqref{eq:pthvolterra} with initial random function $x_0^{(k)} = x_0 \mathbf{1}_{A_k} \in L^\infty(\Omega, {\cal F}_0, \P)$. That is, \( X^{(k)} \) is the sequence of truncated starting function \( x_0^{(k)} \) of the resulting pathwise continuous solution to~\eqref{eq:pthvolterra}.
	We set $X = \sum_{k \in \N^*} X^{(k)} \mathbf{1}_{A_k \setminus A_{k-1}}$ where $A_0 = \emptyset$.
	Let $t \in [0,T]$. Since,
	\[\int_0^t K_2^2(t, s) \|\sigma(s, X^s_{\cdot })\|_{\tilde{\mathbb{H}}}^2 \, ds = \sum_{k \in \N^*} \mathbf{1}_{A_k \setminus A_{k-1}} \int_0^t K_2^2(t, s) \|\sigma(s,X^{(k)}_{\cdot \wedge s})\|_{\tilde{\mathbb{H}}}^2 \, ds < +\infty, \;\P \text{-a.s.}, \]
	 owing to \ref{assump:kernelVolterra} (i) and Lemma \ref{lm:growth}, the stochastic integral $\int_0^t K_2(t, s) \sigma(s,X^s_{\cdot}) \, dW_s$ makes sense. Moreover, since stochastic integrals with respect to an $({\cal F}_s)$-Brownian motion commute with $\cF_0$-measurable random variables, for each $k \in \N^*$, we have
	\begin{align*}
		&\,\mathbf{1}_{A_k \setminus A_{k-1}} \int_0^t K_2(t, s) \sigma(s, X^s_{\cdot}) \, dW_s = \int_0^t K_2(t, s) \mathbf{1}_{A_k \setminus A_{k-1}} \sigma(s, X^s_{\cdot}) \, dW_s \\
		&\hspace{3cm}= \int_0^t K_2(t, s) \mathbf{1}_{A_k \setminus A_{k-1}} \sigma(s, X^{(k)}_{\cdot \wedge s}) \, dW_s = \mathbf{1}_{A_k \setminus A_{k-1}} \int_0^t K_2(t, s) \sigma(s, X^{(k)}_{\cdot \wedge s}) \, dW_s, \;\P \text{-a.s.}.\\
	  &\,\text{Hence, we get,}\qquad\qquad\;
	\int_0^t K_2(t, s) \sigma(s, X^s_{\cdot}) \, dW_s = \sum_{k \in \N^*} \mathbf{1}_{A_k \setminus A_{k-1}} \int_0^t K_2(t, s) \sigma(s, X^{(k)}_{\cdot \wedge s}) \, dW_s, \;\P \text{-a.s.}
    \end{align*}
	\noindent Similarly, we have
	\[\int_0^t K_1(t, s) b( s, X^s_{\cdot}) \, ds = \sum_{k \in \N^*} \mathbf{1}_{A_k \setminus A_{k-1}} \int_0^t K_1(t, s) b(s, X^{(k)}_{\cdot \wedge s}) \, ds, \;\P \text{-a.s.}, \]
	\noindent and thus we deduce that
	\begin{align*}
		X_t &= \sum_{k \in \N^*} \mathbf{1}_{A_k \setminus A_{k-1}} \left( x_0^{(k)}(t) + \int_0^t K_1(t, s) b(s, X^{(k)}_{\cdot \wedge s}) \, ds + \int_0^t K_2(t, s) \sigma(s, X^{(k)}_{\cdot \wedge s}) \, dW_s \right) \\
		&= x_0(t) + \int_0^t b(s, X^s_{\cdot}) \, ds + \int_0^t \sigma(s, X^s_{\cdot}) \, dW_s, \;\P \text{-a.s.},
	\end{align*}
	with the sum over $k \in \N^*$ providing a continuous modification of the right-hand side.
	
	\smallskip
	$\rhd$  {\em Uniqueness}.  Let $X$ and $\widetilde X$ denote two $({\cal F}_t)$-adapted continuous  solutions to~\eqref{eq:pthvolterra} starting from the same initial condition $x_0\!\in L^0(\Omega,{\cal F}_0,\P)$ and let $X^0$ denote the solution starting from $0$. For $k\in \N^*$, $\mbox{\bf 1}_{A_k} X +  \mbox{\bf 1}_{A_k^c} X^0$ and $\mbox{\bf 1}_{A_k} \widetilde X +  \mbox{\bf 1}_{A_k^c} X^0$ are both solutions to~\eqref{eq:pthvolterra} starting from $x_0\mbox{\bf 1}_{A_k}\!\in L^{\infty}(\Omega,{\cal F}_0, \P)$. By uniqueness, $\P(\{X=\widetilde X\}\cap A_k)=\P(A_k)$. Letting $k\to\infty$, we conclude that $
	\P(X=\widetilde X)=1$. %
	Hence uniqueness holds starting from any random variable lying in $L^0(\Omega, {\cal F}_0, \P)$.  \hfill$\Box$
	
	\subsection{Proof of Theorems~\ref{prop:pathkernelvolt} and~\ref{prop:pathkernelvolt2}}
	\noindent {\bf Proof of Theorem\ref{prop:pathkernelvolt}: } {\em Existence of a continuous version to~\eqref{eq:pthvolterra} and $L^p$ pathwise regularity.}
	
	\smallskip
	\noindent {\sc Step~1} ({\em $L^p$-pathwise regularity  $p\!\in (p_{eu}, +\infty)$}).
	For $0\leqslant t<t^\prime\leqslant T$
	$$
	\begin{aligned}
		X_{t^\prime} - X_t &=  (x_0(t^\prime)-x_0(t)) + \int_0^t \left[ b(t^\prime, s, X_\cdot) - b(t, s, X_\cdot) \right] \, ds + \int_0^t \left[ \sigma(t^\prime, s, X_\cdot) - \sigma(t, s, X_\cdot) \right]\, dW_s.\\
		&+\int_t^{t^{\prime}}b(t^{\prime},s,X_\cdot)\mathrm{d}s+\int_t^{t^{\prime}}\sigma(t^{\prime},s,X_\cdot)\mathrm{d}W_s
	\end{aligned}$$
	Thus, since it holds that $\mathbb{E}\|x_0(t^\prime)-x_0(t)\|_{\mathbb{H}}^p\leqslant C_{T,p}\left( 1 + \int_{0}^{T}\mathbb{E}\left[ \| x_0(s) \|_{\mathbb{H}}^p \right] \mu (ds) \right)|t'-t|^{\delta p}$, we have using Lemma \ref{lem:bounds} with $H_u:=b(u,X_{\wedge u})$ and $\widetilde{H}_u:=\sigma(u,X_{\wedge u})$ where we set
	\( |b_u| =  \|H_u\|_{\mathbb{H}}\)  and
	\( |a_u| = \|\widetilde{H}_u\|^2_{\tilde{\mathbb{H}}} \):
	\begin{align*}
		\mathbb{E}\left[\left\|X_{t'} - X_t\right\|_\mathbb{H}^p\right] 
		\leq C_{p,T}\left[\left( 1 + \int_{0}^{T}\mathbb{E}\left[ \| x_0(s) \|_{\mathbb{H}}^p \right] \mu (ds) \right)\vee\left(\sup_{u \in [0,T]}\mathbb{E}\left[\left( |b_u|^{p} + |a_u|^{\frac{p}{2}} \right) \right]\right)\right]|t'-t|^{(\delta p)\wedge( \theta p)\wedge( \widehat\theta p)}\label{eq:estimate}
	\end{align*}
	Then, for any $p\geqslant p_{eu}:=\frac{1}{\delta} \vee\frac{1}{\theta} \vee \frac{1}{ \widehat\theta}, T> 0$, the solution X satisfies (up to a $\mathbb P$-indistinguishability or a path-continuous version $\tilde{X}$ ): \( \forall\, t,\, t^\prime\!\in [0,T],\)	
	\begin{equation*}
		\mathbb{E}\left[\left\|X_{t^\prime} - X_t\right\|_\mathbb{H}^p\right] \le C_{p,T} \left[\left( 1 + \int_{0}^{T}\mathbb{E}\left[ \| x_0(s) \|_{\mathbb{H}}^p \right] \mu (ds) \right)\vee\left(\sup_{u \in [0,T]}\mathbb{E}\left[\left( |b_u|^{p} + |a_u|^{\frac{p}{2}} \right) \right]\right)\right]|t^\prime-t|^{p(\delta \wedge\theta\wedge \widehat\theta)}.
	\end{equation*}
	Using Lemma \ref{lm:growth} we note that:
	\begin{align*}
		&\, \sup_{u \in [0,T]}\mathbb{E}\left[\left( |b_u|^{p} + |a_u|^{\frac{p}{2}} \right) \right] = \sup_{u \in [0,T]} \mathbb{E}\left[\left( \|b(u,X_{\cdot\wedge u})\|^p_\mathbb{H} + \|\sigma(u,X_{\cdot\wedge u})\|^p_\mathbb{\tilde H} \right) \right]  \\
		&\hspace{2cm}\leq 2 C^{\prime}_{b, \sigma, T}  \left(1 + \sup_{u \in [0,T]}\left( \int_{0}^u \mathbb{E}\left[\| X_{s\wedge u} \|_{\mathbb{H}}^p \right]\, \mu (ds) \right)\right) \leq 2 C^{\prime}_{b, \sigma, T}  \left(1 + \int_{0}^T \mathbb{E}\left[\| X_{s\wedge T} \|_{\mathbb{H}}^p \right]\, \mu (ds)\right)
	\end{align*}
	By the bound~\eqref{eq:bound_X} of Theorem\ref{Thm:pathExistenceUniquenes}, it holds that:
	\[\int_{0}^T \mathbb{E}\big[\| X_s\|_{\mathbb{H}}^p \big]\mu (ds) \leq C \left(1 + \int_{0}^T \mathbb{E}\left[ \| x_0(t) \|_{\mathbb{H}}^p \right] \mu(dt)\right)\]
	From this, it follows that the quantities  $a$ and $b$ are bounded, as shown by the inequality:
	\[\sup_{u \in [0,T]}\mathbb{E}\left[\left( |b_u|^{p} + |a_u|^{\frac{p}{2}} \right) \right] \leq  2 C^{\prime}_{b, \sigma, T} (1 + C)\left(1 + \int_{0}^T \mathbb{E}\left[ \| x_0(t) \|_{\mathbb{H}}^p \right] \mu(dt)\right) \]
	This leads us to the inequality~\eqref{eq:Lpincrements}.
	The desired conclusions about the \(L^p-\)  boundness of the \(a\)-H\"older seminorm on \([0, T]\) follow from [\cite{ZhangXi2010}, Theorem 2.10] or owing to Kolmogorov's continuity (Kolmogorov's C-tightness) criterion (see \cite[Theorem 2.1, p. 26, 3rd edition]{RevuzYor}( which can be derived directly by Garsia's inequality (cf. \cite{GRR1070})) \footnote{Let $ X$ be a stochastic process with values in the Polish metric space $(S, \rho)$. If there exist positive constants $\alpha,\beta, c > 0 $ such that \(\mathbb E \left[ \rho( X_s, X_t)\right]^\alpha \le  c | s - t| ^{\beta+d}, \quad s,t \in \mathbb R^d\),
		then $X$ admits a continuous modification and there exists a modification whose paths are H\"{o}lder continuous of order $\gamma$, for every $\gamma \in (0, \frac{\beta}{\alpha})$.} or \cite[Lemma 44.4, Section IV.44, p.100]{RogersWilliamsII}).
	
	That is \( t \mapsto X_t \) admits a H\"older continuous modification (still denoted \( X \) in lieu of \( \tilde{X} \) up to a \( \mathbb{P} \)-indistinguishability), and  for any \( p > 0 \), \( T > 0 \), and some \( a \in \big(0,(\delta \wedge\theta\wedge \widehat\theta) -\frac 1 p\big) \), we have:
	\begin{equation}
		\mathbb{E}\left(\sup\limits_{t\neq t'\in[0,T]}\frac{\|X_{t^\prime}-X_t\|_\mathbb{H}^p}{|t'-t|^{ap}}\right) \leqslant C_{T,p,a}\left( 1 + \int_{0}^{T}\mathbb{E}\left[ \| x_0(s) \|_{\mathbb{H}}^p \right] \mu (ds) \right).
	\end{equation}
	%  Finally, if $a$ and $b$ are locally bounded, choose stopping times $\tau_n \to \infty$ such that $a$ and $b$ are bounded on $[0,\tau_n]$. Then \(
	%  X^n := X\,\bm{1}_{[0,\tau_n]} \)
	%  has a version which is H\"older continuous of any order $a < \delta \wedge\theta\wedge \widehat\theta$, and $X_t = X^n_t$ almost surely on $\{ t \leq \tau_n \}$ for each $t$.
	\smallskip
	\noindent {\bf Remark: }	
	{\em A variant of the Kolmogorov's C-tightness criterion in the spirit of Garsia-Rodemich-Rumsey inequality.}
	An alternative approach to proving the above estimate would involve the use of the Garsia-Rodemich-Rumsey inequality; see Lemma 1.1 in \cite{GRR1070} with \( \psi(u) = |u|^p \) and \( p(u) = |u|^{q + \frac{1}{p}} \) for \( q > \frac{1}{p} \). It states that for a continuous function \( f \) on \( \mathbb{R}_+ \), there exists a constant \( C_{p,q} > 0 \) such that for any \( x_2 > x_1 \geq 0 \),
	\[
	|f(x_2) - f(x_1)| \leq C_{p,q} \cdot |x_2 - x_1|^{q - \frac1p}
	\left( \int_{x_1}^{x_2} \int_{x_1}^{x_2} \frac{|f(t) - f(r)|}{|t - r|^{1 + pq }} \, dr \, dt \right)^{\frac1p}.
	\]
	
	Before stating the result, we recall the following definition. For \( a \in (0, 1] \), the \(a\)-H\"older coefficient of a real-valued function \( f \) or its \(a\)-H\"older seminorm on \([0, T]\) is defined by
	\[
	\|f\|_{C^{a}_{T}} := \sup_{0 \leq x_1 < x_2 \leq T} \frac{|f(x_1) - f(x_2)|}{|x_1 - x_2|^{a}}.
	\]
	
	Let \(a \in (0,1]\). We aim to apply the above GRR inequality to the process \( X \) (see for example \cite[Theorem F.1 (GRR lemma) ]{JouPag22}) for large values of \( p \), namely \(p > p_a := \frac{1}{\delta \wedge \vartheta \wedge \widehat \theta - a}\), and \( q = \tfrac{1}{p} + a \).
	We now have:
	\[
	\| X \|_{C^a_{T}}^p
	= \sup_{0 \leq t < t' \leq T} \frac{\|X_{t'} - X_t\|_{\mathbb{H}}^p}{|t' - t|^{pa}}
	\leq C \cdot \int_0^{T} du \int_0^{T} \frac{|X_u - X_r|^p}{|u - r|^{pa + 2}} \, dr.
	\]
	Taking expectations on both sides of the preceding inequality, and then applying Fubini's theorem along with equation~\eqref{eq:Lpincrements}, we obtain:
	\begin{align*}
		&\,\mathbb{E} \left[  \sup_{0 \leq t < t' \leq T} \frac{\|X_{t'} - X_t\|^p_{\mathbb{H}}}{|t' - t|^{pa}} \right] \leq C \cdot C_{p,T} \cdot \left( 1 + \int_{0}^{T}\mathbb{E}\left[ \| x_0(s) \|_{\mathbb{H}}^p \right] \mu (ds) \right) \cdot \int_0^{T} du \int_0^{T} |u - r|^{p(\delta \wedge \vartheta \wedge \widehat \theta) - pa - 2} \, dr \nonumber \\
		&\hspace{.3cm}\leq \tilde C_{p,T} \cdot (1 + T)^{p(\delta \wedge \vartheta \wedge \widehat \theta - a)}\cdot \left( 1 + \int_{0}^{T}\mathbb{E}\left[ \| x_0(s) \|_{\mathbb{H}}^p \right] \mu (ds) \right)\leq C_{p,T, \delta,\widehat \theta, \vartheta,a} \cdot \left( 1 + \int_{0}^{T}\mathbb{E}\left[ \| x_0(s) \|_{\mathbb{H}}^p \right] \mu (ds) \right).
	\end{align*}
	We then conclude that the bound holds uniformly for all \( T \geq 0 \).
	
	\smallskip
	\noindent {\sc Step~2}  {\em (Maximal Inequalities) :}  {\em Proof of the bound ~\eqref{eq:supLpbound} for $p\!\in (p_{eu}, +\infty)$}.
	Now, we can establish the maximal inequality. Let \( p \) be given. For each \( a \in \big(0, \delta \wedge \vartheta \wedge \widehat \theta \big) \),
	Noting that:
	\[
	\| X_t \|_{\mathbb{H}}^p = \| X_t - X_s + X_s \|_{\mathbb{H}}^p \leq 2^{p-1} \left( \| X_t - X_s \|_{\mathbb{H}}^p + \| X_s \|_{\mathbb{H}}^p \right),
	\]
	Equation~\eqref{eq:supLpbound} can be derived from Equation~\eqref{eq:Holderpaths} by setting \( s = 0 \), as follows:
	\[
	\sup_{0\leq t' \leq T} \| X_{t'} \|_{\mathbb{H}}^p \leq 2^{p-1}\left( \sup_{0\leq t' \leq T} \|X_{t'} - x_0(0)\|_{\mathbb{H}}^p + \|x_0(0)\|_{\mathbb{H}}^p \right) \leq 2^{p-1}\left( \| X \|_{C^a_{T}}^p T^{pa} + \int_{0}^{T} \| x_0(s) \|_{\mathbb{H}}^p\, \mu (ds) \right)
	\]
	Taking expectations on both sides of the inequality and then using the claim~\eqref{eq:Holderpaths} yields:
	\begin{align*}
		\mathbb{E} \left[ \sup_{0 \leq t' \leq T} \|X_{t'}\|_{\mathbb{H}}^p \right] &\leq 2^{p-1} \left( T^{pa} \mathbb{E} \left[ \sup_{0 \leq t < t' \leq T} \frac{\|X_{t'} - X_t\|_{\mathbb{H}}^p}{|t' - t|^{pa}} \right] +  \int_{0}^{T}\mathbb{E}\left[ \| x_0(s) \|_{\mathbb{H}}^p \right] \mu (ds) \right) \\
		& \le 2^{p-1}\left( T^{pa} C_{p,T, \delta,\widehat \theta, \vartheta,a} \cdot \left( 1 + \int_{0}^{T}\mathbb{E}\left[ \| x_0(s) \|_{\mathbb{H}}^p \right] \mu (ds) \right) + \int_{0}^{T}\mathbb{E}\left[ \| x_0(s) \|_{\mathbb{H}}^p \right] \mu (ds)  \right)\\
		&\leq C'_{p,T, \delta,\widehat \theta, \vartheta,a} \left( 1 + \int_{0}^{T}\mathbb{E}\left[ \| x_0(s) \|_{\mathbb{H}}^p \right] \mu (ds) \right)
	\end{align*}
	where the constant $C'_{p,T, \beta,\gamma_T, \vartheta,a} $ is given by: $C'_{p,T, \beta,\gamma_T, \vartheta,a}= 2^{p-1}\left(1+ T^{pa} C_{p,T, \beta,\gamma_T, \vartheta,a}\right)$. This give the announced result and we are done. \hfill$\Box$
	
	\medskip
	\noindent {\bf Proof of Theorem~\ref{prop:pathkernelvolt2}:} We note that: Let $T > 0$ and $p > p_{eu}:=\frac{1}{\delta} \vee\frac{1}{\theta} \vee \frac{1}{ \widehat\theta}$, the results stated in the Corollary are straigthforward consequence of Proposition~\ref{prop:pathkernelvolt} with Lemma \ref{lm:Uniformgrowth} and all the bounds are valid.
	
	To extend the $L^p$-pathwise regularity~\ref{eq:Holderpaths2}  to every \( p \in (0, +\infty) \) such that \(  \big\| \,\| x_0  \|_{T} \,\big\|_p < +\infty \),
	one can take advantage of the form of the control to apply the splitting Lemma~\ref{lem:gap} with \( p < \bar{p} \), where \(\bar{p} = \frac{1}{\delta \wedge \vartheta \wedge \widehat \theta - a} \vee p_{\text{eu}}\),
	and considering the functional \( \Phi: \mathcal C ([0, T],\R) \to \mathbb{H} \) defined by:
	
	$$\Phi(x,y) = \sup_{s \neq t \in [0, T]} \frac{\|x(t)-x(s)\|_{\mathbb{H}}}{|t - s|^a}.$$
	
	The marginal bound~\eqref{eq:Lpincrements2} can be extended likewise by considering the functional $\Phi: \mathbb{X} \to \mathbb{H}$ defined by:
	\(\Phi(x,y) = \|x(t)-x(s)\|_{\mathbb{H}}.\) \\
	
	The sup bound~\eqref{eq:supLpbound2} above, valid for all \(p > \bar{p}\), can be extended simirlarly to every \(p \in (0,+\infty) \) by applying the splitting argument (Lemma~\ref{lem:gap}), using the same \(\bar{p}\) as above and taking the functional \(\Phi(x,y) = \sup_{t \in [0,T]} \|x(t)\|_{\mathbb{H}}.\)\\ 
	This complete the proof and we are done. \hfill$\Box$ 
	
	\subsection{Proofs of Proposition~\ref{propXtilde} and Theorem~\ref{thm:Eulercvgce2} }
	\noindent {$\rhd$ {\em Preliminaries}. As a first preliminary, we will establish   the following lemma  which provides a control in $L^p(\P)$ of some quantities involving the integrator. It  will be used several times  throughout the appendices.    	
		\begin{Lemma} \label{lem:intHolcont} 
			For all \( r \in [0,T] \) and for any process \( X^r_\cdot \in \mathbb{X} \),
			there exists a constant \(C=C_{p,T}\) such that:
			\begin{align}
				\Big( \int_0^{\underline r} \mathbb{E} \left[\|\iota_{[\underline r]}\big(\bar X^{h}_{t_0:t_{[\underline r]}} \big)_u\|^p_{\mathbb{H}} \right] \mu(du)\Big)^\frac{1}{p} 
				&\leq C\,\mu([0,r])^\frac{1}{p} \sup_{0 \le k \le [\underline{r}]} \Big\| \| \bar{X}_{t_k} \|_{\mathbb H} \Big\|_p
				\label{eq:boundIntegrator} \\
				\Bigg( \int_{0}^{r} 
				\mathbb{E}\Big[\|\bar X_u - \iota_{[\underline r]}\big(\bar X_{t_0:t_{[\underline r]}} \big)_u\|_{\mathbb{H}}^p\Big] \,\mu(du)
				\Bigg)^{\frac{1}{p}} 
				&\leq 2\,\mu([0,r])^{\frac{1}{p}}
				\sup_{0\le u\le r}\Big\|\|\bar X_u - \bar{X}_{\underline u}\|_{\mathbb{H}}^p\Big\|_p
				\label{eq:boundHolIntegrator}
			\end{align}
		\end{Lemma}
		\noindent {Proof }
		Let  \( r \in [0,T] \) and consider \( X^r_\cdot \in \mathbb{X} \).
		\noindent {\sc Step~1}
		For all \(\ell=0,...,n\) , \(\forall \, u \in [t_{\ell},t_{\ell+1}]\subset [0,r]\),\\ 
		% \( \|\iota_{[\underline r]}\big(\bar X^{h}_{t_0:t_{[\underline r]}} \big)_u\|_{\mathbb{H}} \leq \| \bar{X}_{t_{\ell}} \|_{\mathbb H} \vee \| \bar{X}_{t_{\ell+1}} \|_{\mathbb H} \leq \| \bar{X}_{t_{\ell}} \|_{\mathbb H}+ \| \bar{X}_{t_{\ell+1}} \|_{\mathbb H}  \) 
		\[ \|\iota_{[\underline r]}\big(\bar X^{h}_{t_0:t_{[\underline r]}} \big)_u\|_{\mathbb{H}} \leq \| \bar{X}_{t_{\ell}} \|_{\mathbb H} \vee \| \bar{X}_{t_{\ell+1}} \|_{\mathbb H} \leq \| \bar{X}_{t_{\ell}} \|_{\mathbb H}+ \| \bar{X}_{t_{\ell+1}} \|_{\mathbb H} \]
		so that, \(\forall\, r>0\),
		\begin{equation*}
			\int_0^{\underline r} \mathbb{E} \left[\|\iota_{[\underline r]}\big(\bar X^{h}_{t_0:t_{[\underline r]}} \big)_u\|^p_{\mathbb{H}} \right] \mu(du) \leq \sum_{\ell=0}^{\underline r-1}\int_{t_\ell}^{t_{\ell+1}}\mathbb{E} \left[\|\iota_{[\underline r]}\big(\bar X^{h}_{t_0:t_{[\underline r]}} \big)_u\|^p_{\mathbb{H}} \right] \mu(du) \leq 2^p \mu([0,\underline r]) \sup_{0 \le k \le [\underline{r}]} \mathbb{E} \left[\| \bar{X}_{t_k} \|_{\mathbb H}^p\right]
		\end{equation*}
		Alternatively, from Equation~\eqref{eq:interpolator}, \(\forall \, u \in [0,\underline{r}]\),
		\begin{align*}
			&\, \mathbb{E} \left[\|\iota_{[\underline r]}\big(\bar X^{h}_{t_0:t_{[\underline r]}} \big)_u\|^p_{\mathbb{H}} \right] \leq (2[\underline r])^{p-1}\sum_{\ell=1}^{\underline r} \mbox{\bf 1}_{[t_{\ell-1},t_\ell)}(u) \left(\frac{(t_{\ell}-u)^p}{h^p}\mathbb{E} \left[\| \bar{X}_{t_{\ell-1}} \|_{\mathbb H}^p\right] + \frac{(u-t_{\ell-1})^p}{h^p} \mathbb{E} \left[\| \bar{X}_{t_\ell} \|_{\mathbb H}^p\right]\right)\\
			&\text{by integrating,}\;
			\Big(	\int_0^{\underline r} \mathbb{E} \left[\|\iota_{[\underline r]}\big(\bar X^{h}_{t_0:t_{[\underline r]}} \big)_u\|^p_{\mathbb{H}} \right] \mu(du)\Big)^\frac1p \leq \Big(2^p[\underline r]^{p-1} \sup_{0 \le k \le [\underline{r}]} \mathbb{E} \left[\| \bar{X}_{t_k} \|_{\mathbb H}^p\right] \sum_{\ell=1}^{\underline r}\int_{t_{\ell-1}}^{t_{\ell}} \mu(du)\Big)^\frac1p\\
			&\hspace{9.5cm}= 2 [\underline r]^{1-\frac1p} \mu([0,\underline r])^\frac1p  \sup_{0 \le k \le [\underline{r}]} \Big\| \| \bar{X}_{t_k} \|_{\mathbb H} \Big\|_p
		\end{align*}
		Consequently, in both cases, there exists a constant \(C=C_{p,T}\) such that:
		\[\Big(	\int_0^{\underline r} \mathbb{E} \left[\|\iota_{[\underline r]}\big(\bar X^{h}_{t_0:t_{[\underline r]}} \big)_u\|^p_{\mathbb{H}} \right] \mu(du)\Big)^\frac1p \leq C\,\mu([0,T])^\frac1p  \sup_{0 \le k \le [\underline{r}]} \Big\| \| \bar{X}_{t_k} \|_{\mathbb H} \Big\|_p\]
		\medskip 
		\noindent {\sc Step~2}
		For all \(u \in [t_\ell,t_{\ell+1}] \subset [0,r]\),
		\begin{align*}
			&\,	\mathbb{E}\Big[\|\bar X_u - \iota_{[\underline r]}\big(\bar X_{t_0:t_{[\underline r]}} \big)_u\|_{\mathbb{H}}^p\Big]
			\leq 2^{p-1}\Big(
			\mathbb{E}\|\bar X_u - \bar X_{t_\ell}\|_{\mathbb{H}}^p
			+
			\mathbb{E}\| \iota_{[\underline r]}\big(\bar X_{t_0:t_{[\underline r]}} \big)_u-\bar X_{t_\ell}\|_{\mathbb{H}}^p
			\Big) \\
			&\hspace{3.5cm}\leq 2^{p-1}\Big(
			\mathbb{E}\|\bar X_u - \bar X_{t_\ell}\|_{\mathbb{H}}^p
			+
			\mathbb{E}\| \bar X_{t_{\ell+1}}-\bar X_{t_\ell}\|_{\mathbb{H}}^p
			\Big) \leq 2^p 
			\sup_{u\in[t_\ell,t_{\ell+1}]} 
			\mathbb{E}\big[\|\bar X_u - \bar{X}_{\underline u}\|_{\mathbb{H}}^p\big]
		\end{align*}
		So that, by integrating, we have:
		\begin{align*}
			% \text{So that, integrating,}\;
			\Bigg(
			\int_{0}^{r} 
			\mathbb{E}\Big[\|\bar X_u &\,- \iota_{[\underline r]}\big(\bar X_{t_0:t_{[\underline r]}} \big)_u\|_{\mathbb{H}}^p\Big] \,\mu(du)
			\Bigg)^{\tfrac{1}{p}}  = \Bigg(
			\sum_{l=0}^{r-1} \int_{t_\ell}^{t_{\ell+1}} 
			\mathbb{E}\Big[\|\bar X_u - \iota_{[\underline r]}\big(\bar X_{t_0:t_{[\underline r]}} \big)_u\|_{\mathbb{H}}^p\Big] \,\mu(du)
			\Bigg)^{\tfrac{1}{p}} \\
			&\leq 
			2 \left(
			\sup_{0 \le u \le r}\mathbb{E}\big[\|\bar X_u - \bar{X}_{\underline u}\|_{\mathbb{H}}^p\big]
			\cdot \int_0^r \mu(dr)
			\right)^{\tfrac{1}{p}} = 2\,\mu([0,r])^{\tfrac{1}{p}}
			\sup_{0\le u\le r}\Big\|\|\bar X_u - \bar{X}_{\underline u}\|_{\mathbb{H}}^p\Big\|_p.
		\end{align*}
		\paragraph{Proof of Proposition~\ref{propXtilde} (Some Properties of the  continuous extension process of the genuine interpolated $K$-integrated Euler scheme)}
		%\medskip 
		%	\noindent {\bf Proof of Proposition~\ref{propXtilde} (Some Properties of the  continuous extension process of the interpolated $K$-integrated Euler scheme).}\\
		\noindent {\sc Property~{\bf 1}.} \label{pp:1} {\em $L^p$-integrability and pathwise continuity of $\bar X_t$, $p>0$}. 
		We will drop the superscript $h$ in $X^h$ and n in $t_k^n$ for simplicity.
		{First we prove that $\bar X_{t^n_k}\!\in L^p(\P)$ by induction on $k$.  If $\bar X_{t^n_{\ell}}\!\in L^p(\P)$ for $\ell=0, \ldots,k-1$, it follows easily from equation~\eqref{eq:discretescheme} that $\bar X_{t^n_k}\!\in L^p(\P)$ since both $b$ and $\sigma$ have linear growth on their space variable uniformly w.r.t. their time variable as a consequence of~\eqref{eq:coefficient_hoelder2}(or~\eqref{eq:HolLipbsig2})  and  $\bar X_{t^n_{\ell-1}}$ and the Gaussian Wiener integral $\displaystyle \int_{t^n_{\ell-1}}^{t^n_\ell}K_2(t^n_k,s)dW_s$ are independent. The induction relies on the Minkowski inequality when $p\ge 1$ and on its subadditive counterpart for $\|\cdot\|_p^p$ when $p\!\in (0,1)$. }
		
		For the continuous time   $K$-integrated Euler scheme~\eqref{eq:def_continuous_2}, one derives likewise  from~\eqref{eq:Eulergen2bis} and Lemma \ref{lem:bounds} that, as $\bar X_{t_\ell}\!\in L^p(\P)$ for every $\ell=0,\ldots,k$, one has $\sup_{t\in[t_k,t_{k+1}]}\|\bar X_t\|_p<+\infty$. As a consequence, $\sup_{t\in[0,T]}\|\bar X_t\|_p<+\infty$. 
		
		\smallskip
		\noindent 
		As for the pathwise continuity, first note that   the functions $\left(\int_{t^n_{\ell}}^{t^n_{\ell+1}}  K_1(t,s)ds\right)_{t\in[t^n_{\ell+1},T]}$, $\ell=0, \ldots,n-1$,  are continuous owing to~$({\cal K}^{cont}_{\theta})$.  So are       $\left(\int_{t^n_\ell}^t K_1(t,s)ds\right)_{t\in[t^n_\ell,t^n_{\ell+1}]}$, $\ell=0,\ldots,n-1$, owing to either~$(\widehat {\cal K}^{int}_{T,\widehat \theta})$ or~$({\cal K}^{int}_{\beta})$  and $({\cal K}^{cont}_{\theta})$. Applying the above Lemma~\ref{lem:bounds} with  $\widetilde H\equiv 1$ between $a=t^n_k$ and $b=t^n_{k+1}$ with $k\in\{0,\cdots,n-1\}$ and $p>\frac{1}{\theta\wedge\theta^*}$ yields that the processes $\left(\int_{t^n_k}^{t\wedge t^n_{k+1}} K_2(t,s)dW_s\right)_{t\in[t^n_k,T]}$ have a continuous modification (null at $t^n_k$) owing to  Kolmogorov's criterion (see~\cite[Theorem 2.1, p.26, $3^{rd}$ edition]{RevuzYor}). Then one concludes that the Interpolated $K$-integrated Euler scheme has a continuous modification.
		
		\medskip 
		\noindent {\sc Property~{\bf ~2}.} \label{pp:2} {\em Moment control of $\bar X_t$, $p\ge 2$}.
		Set \(\bar f_p(t) = \sup_{0 \le s \le t} \Big\| \|\bar X_s\|_{\mathbb{H}} \Big\|_p.\)
		We already know that $ \bar f_p(T) < \infty$, we now aim to derive an estimate that just depend on $\Big\| \|x_0 \|_{T} \Big\|_p$ and the number \(n\) of time-steps. 
		For all \( s \in [0,t] \), we have using the generalized Minkowski  for the deterministic term
		
		\begin{align*}
			&\ \Big\| \left\| \bar X_s
			\right\|_{\mathbb{H}} \Big\|_p
			\leq \Big\| \left\| x_0(s) \right\|_{\mathbb{H}} \Big\|_p +  \int_0^s K_1(s,r)  \Big\| \left\|b \Big(\underline r, \iota_{[\underline r]}\big(\bar X^{h}_{t_0:t_{[\underline r]}} \big) \Big)\right\|_{\mathbb{H}} \Big\|_p\,dr  \\
			&\qquad\qquad\qquad\qquad\qquad + \Big\| \big\| \int_0^s K_2(s,r)\sigma \Big(\underline r, \iota_{[\underline r]}\big(\bar X^{h}_{t_0:t_{[\underline r]}} \big) \Big)\,dW_r \big\|_{\mathbb{H}} \Big\|_p \leq \Big\| \left\| x_0(s) \right\|_{\mathbb{H}} \Big\|_p\\ &\qquad+  \int_0^s K_1(s,r)  \Big\| \left\|b \Big(\underline r, \iota_{[\underline r]}\big(\bar X^{h}_{t_0:t_{[\underline r]}} \big) \Big)\right\|_{\mathbb{H}} \Big\|_p\,dr + C_{p}^{BDG} \left[ \int_0^s K_2(s,r)^2 \Big\| \big\|\sigma \Big(\underline r, \iota_{[\underline r]}\big(\bar X^{h}_{t_0:t_{[\underline r]}} \big) \Big)\big\|_{\tilde{\mathbb{H}}} \Big\|_p^2 \, dr \right]^{\tfrac12}.
		\end{align*}
		where in the last inequality, we used the \(L^p\)-Burkholder--Davis--Gundy (BDG) inequality.

		\noindent Now, under Assumption~$({\cal LH}_{\gamma})$, $b$ and $\sigma $ have  linear growth on $x$ uniformly in $t\!\in \mathbb{T}$. It follows by applying Fubini Tonelli theorem that,
		\begin{align*}
			\Big\| \big\|b \Big(\underline r, \iota_{[\underline r]}\big(\bar X^{h}_{t_0:t_{[\underline r]}} \big) \Big)\big\|_{\mathbb{H}} \Big\|_p & \leq C_{b,\sigma,T} \left(1 + \Big( \int_0^{\underline r} \mathbb{E} \left[\|\iota_{[\underline r]}\big(\bar X^{h}_{t_0:t_{[\underline r]}} \big)_u\|^p_{\mathbb{H}} \right] \mu(du) \Big)^\frac1p\right)\\
			&\leq C_{b,\sigma,T}\,\left(1\vee C \mu([0,T])^\frac1p\right)\,\left(1+\sup_{0 \le k \le [\underline{r}]} \Big\| \| \bar{X}_{t_k} \|_{\mathbb H} \Big\|_p\right).
		\end{align*}
		where we used~\eqref{eq:boundIntegrator} of Lemma~\ref{lem:intHolcont} in the last inequality.
		As $\underline r \le r$ and by the very definition of $[\underline{r}]$,   we have: 
		\[ \sup_{0 \le k \le [\underline{r}]}\Big\| \| \bar{X}_{t_k} \|_{\mathbb H} \Big\|_p \le  \sup_{u \in [0,r]} \Big\| \| \bar{X}_{u} \|_{\mathbb H} \Big\|_p:=\bar f_p(r) \]
		Proceeding analogously for the diffusion coefficient and combining all the above estimates, we obtain the following by setting \(C_1=C_{b,\sigma,T}\,\,\left(1\vee C \mu([0,T])^\frac1p\right)\) and \(C_2=C_{p}^{BDG}\,C_{b,\sigma,T}\,\,\left(1\vee C \mu([0,T])^\frac1p\right)\):
		{\small
			\begin{align*}
				\Big\| \left\| \bar X_s\right\|_{\mathbb{H}} \Big\|_p
				&\leq \Big\| \left\| x_0(s) \right\|_{\mathbb{H}} \Big\|_p +  C_1 \left(\varphi_1(s) + \int_0^s K_1(s,r)  \bar f_p(r)\,dr\right) + C_2 \left(\psi_2(s) + \Big(\int_0^s K_2(s,r)^2 \bar f_p^2(r)\, dr\Big)^\frac12\right)\\
				&\leq \Big\| \left\| x_0(s) \right\|_{\mathbb{H}} \Big\|_p + C_1 \varphi_1(s)+ C_2 \psi_2(s) + \sum_{i \in \{1,2\}} C_i \; \chi_i(s) \left( 
				\frac{\rho}{A^{\frac1\rho}}\,\bar f_p(s) 
				+ (1-\rho)\,A^{\frac{1}{1-\rho}} \int_0^s \bar f_p (r)\,dr
				\right).
			\end{align*}
		}
		\noindent where we applied relation~\eqref{eq:VoltIntegral} from Lemma~\ref{Lem:BoundVoltIntegral} for all \( A > 0 \) for \(i=1\) and \(i=2\) as \(\bar f_p\) is non-decreasing with \(\rho \in (0,1)\), and \(\chi_1 = \varphi_{\frac{2\beta}{\beta+1}}\) and \(\chi_2 = \psi_{2\beta}\) as defined in equation ~\eqref{eq:notat2}.
		Now, passing to the supremun on \([0,t]\) while taking advantage of the fact that \(\bar f_p\) is non-decreasing, we derive that
		\begin{align*}
			&\, \sup_{0 \le s \le t} \Big\| \|\bar X_s\|_{\mathbb{H}} \Big\|_p:= \bar f_p(t) \leq  C^\prime + \sup_{0 \le s \le t} \Big\| \left\| x_0(s) \right\|_{\mathbb{H}} \Big\|_p + C\,\left( 
			\frac{\rho}{A^{\frac1\rho}}\,\bar f_p(t) 
			+ (1-\rho)\,A^{\frac{1}{1-\rho}} \int_0^t \bar f_p(r)\,dr
			\right).\\
			&\text{with}\; C^\prime = C_1 \varphi_1^*(s)+ C_2 \psi_2^*(s) <\infty \;\text{and}\;
			C = C_{K_1,K_2,p,\beta,b,\sigma,T} = C_1\,\varphi^{*}_{\frac{2\beta}{\beta+1}} 
			+ C_2\,\psi^{*}_{2\beta}<\infty \;\text{owing to }\; ({\cal K}^{int}_{\beta}).
		\end{align*}
		We now choose \(A\) large enough, namely \(A = (2\rho\, C)^{\rho}\), so that 
		\(1 - \rho\, C\,A^{-\frac1\rho} = \frac{1}{2} < 1\).
		Consequently, for every \(t \in [0,T]\),
		
		\centerline{$\bar f_p(t) \le 2\,\left(C^\prime+\Big\| \sup_{0 \le s \le t} \left\| x_0(s)\right\|_{\mathbb{H}} \Big\|_p\right)
			+ 2C\,(1-\rho)\,A^{\frac{1}{1-\rho}} \int_0^t \bar f_p(r)\,dr.$}
		\noindent One concludes by Gr\"onwall's lemma (See for example \cite[Lemma 7.2]{pages2018numerical}) that:
		\[
		\bar f_p(t) \le 2\,\left(C^\prime+\Big\| \left\| x_0 \right\|_{t} \Big\|_p\right) e^{K\,t},\;
		\text{with} \;
		K = K_{K_1,K_2,p,\beta,b,\sigma,T}
		= 2C(1 - \rho)\,A^{\frac1{1-\rho}}.
		\]
		To conclude, \(	\sup_{0 \le s \le T} \big\| \|\bar X_s\|_{\mathbb{H}} \big\|_p
		= \bar f_p(T) \le C^{(2)}\,\left(1+\big\| \| x_0 \|_{T} \big\|_p\right),\;
		C^{(2)}:= 2\,\max \big( 1,C^\prime)\, e^{K\,T} \ge 2.\) Note that by the
		same reasoning applied to the Volterra process, we also derive the result of Lemma~\ref{lem:supEbound} namely \(\sup_{0 \le s \le T} \Big\| \| X_s\|_{\mathbb{H}} \Big\|_p \le C^{(2)}\,\left(1+\big\| \| x_0 \|_{T} \big\|_p\right).\)
		
		\smallskip 
		\noindent {\sc Property~{\bf ~3}.}\label{pp:3}  {\em  Control of the increments $\left\| \|\bar X_t-\bar X_s\|_{\mathbb{H}} \right\|_p$ of the interpolated $K$-integrated Euler scheme (for $p\ge 2$)}. 
		{Our aim in this step  is to prove   that there exists a positive real constant  $C_{b,\sigma,p,T,K_1,K_2}$ not depending on $n$ such that 
			\[
			\forall\, s,\, t\!\in [0,T], \quad  \left\| \|\bar X_t-\bar X_s\|_{\mathbb{H}} \right\|_p\le C_{b,\sigma,p,T,K_1,K_2}(1+\left\| \|x_0 \|_{T} \right\|_p) |t-s|^{\delta \wedge \theta \wedge \theta^*}.
			\]
			We recall from the definition~\eqref{def:discretization_scheme}, the equation ~\eqref{eq:def_continuous_2}
			\begin{align*}
				&\ \bar X_t = \bar x_0(t) + \int_0^t K_1(t,s) b \Big(\underline s, \iota_{[\underline s]}\big(\bar X_{t_0:t_{[\underline s]}} \big) \Big)\, ds + \int_0^t K_2(t,s) \sigma \Big(\underline s, \iota_{[\underline s]}\big(\bar X_{t_0:t_{[\underline s]}} \big) \Big)\, d W_s, \quad t\in \mathbb{T}\\
				&\forall\;	0\leqslant s<t\leqslant T, \quad\; \bar X_t - \bar X_s =   (\bar x_0(t)-\bar x_0(s)) + \int_{0}^{t} K_1(t,u) H_u \, du - \int_{0}^{s} K_1(s,u) H_u \, du\\
				&\hspace{6cm}+ \int_{0}^{t} K_2(t,u) \widetilde H_u \, dW_u - \int_{0}^{s} K_2(s,u) \widetilde H_u \, dW_u
			\end{align*}
			where $H_s = b \Big(\underline s, \iota_{[\underline s]}\big(\bar X_{t_0:t_{[\underline s]}} \big)$ and $\widetilde H_s =  \sigma \Big(\underline s, \iota_{[\underline s]}\big(\bar X_{t_0:t_{[\underline s]}} \big)$. We note that, still under Assumption~$({\cal LH}_{\gamma})$, $b$ and $\sigma $ have  linear growth on $x$ uniformly in $t\!\in \mathbb{T}$. It follows by applying Fubini Tonelli theorem that,
			\begin{align*}
				&\,	\Big\| \big\|b \Big(\underline r, \iota_{[\underline r]}\big(\bar X^{h}_{t_0:t_{[\underline r]}} \big) \Big)\big\|_{\mathbb{H}} \Big\|_p \vee \Big\| \big\|b \Big(\underline r, \iota_{[\underline r]}\big(\bar X^{h}_{t_0:t_{[\underline r]}} \big) \Big)\big\|_{\mathbb{H}} \Big\|_p \leq C_{b,\sigma,T} \left(1 + \Big( \int_0^{\underline r} \mathbb{E} \left[\|\iota_{[\underline r]}\big(\bar X^{h}_{t_0:t_{[\underline r]}} \big)_u\|^p_{\mathbb{H}} \right] \mu(du) \Big)^\frac1p\right)\\
				&\leq C_{b,\sigma,T}\,\left(1\vee C \mu([0,T])^\frac1p\right)\,\big(1+\sup_{0 \le k \le [\underline{r}]} \Big\| \| \bar{X}_{t_k} \|_{\mathbb H} \Big\|_p\big) \leq C_{b,\sigma,T}\,\left(1\vee C \mu([0,T])^\frac1p\right)\,\big(1+\sup_{u \in [0,r]} \Big\| \| \bar{X}_{u} \|_{\mathbb H} \Big\|_p\big).
			\end{align*}
			where we used Equation~\eqref{eq:boundIntegrator} of Lemma~\ref{lem:intHolcont} in the penultimate inequality.	      	Consequently, calling upon Lemma~\ref{lem:bounds} and Property~{\bf 2} , it follows that when either $({\cal K}^{int}_{\beta})$  is in force or $\widehat{\cal K}^{cont}_{\widehat \theta}$ holds, we have:
			\[
			\begin{aligned}
				\left\| \|\bar X_t-\bar X_s\|_{\mathbb{H}} \right\|_p &\le \left\| \|\bar x_0(t)-\bar x_0(s) \|_{\mathbb{H}} \right\|_p \\
				&+  C_{p,T,K_1,K_2} \Big(\sup_{s\in [0,T]}\big\| \|b \Big(\underline s, \iota_{[\underline s]}\big(\bar X_{t_0:t_{[\underline s]}} \big) \|_{\tilde{\mathbb H}} \big\|_p   +\sup_{s\in [0,T]}\big\| \|\sigma \Big(\underline s, \iota_{[\underline s]}\big(\bar X_{t_0:t_{[\underline s]}} \big) \|_{\tilde{\mathbb H}} \big\|_p  \Big)(t-s)^{\theta \wedge\theta^*}\\
				&\overset{(\ref{assump:kernelVolterra} (iv))}{\operatorname*{\leq}} K_{T,p}\big(1+  \big\|\|\bar x_0\|_T\big\|_p\big)|t-s|^{\delta} 
				+ C_{p,T,K_1,K_2} \times 2 C_{b,\sigma, T}C^\prime\big(1+  \big\|\|\bar x_0\|_T\big\|_p\big)(t-s)^{\theta \wedge \theta^*}\\
				&\le C^{(3)}\,\big(1+  \big\|\|\bar x_0\|_T\big\|_p\big)(t-s)^{\delta \wedge \theta \wedge \theta^*},
			\end{aligned}
			\]
			where $C^{(3)} = K_{T,p} + 2C_{p,T,K_1,K_2}C_{b,\sigma, T}C^\prime$ with $ C_{p,T,K_1,K_2}$ from~\ref{lem:bounds} and \(C^\prime:=(1\vee C \mu([0,T])^\frac1p) (1+ C^{(2)})\).
		}
		
		\medskip
		\noindent{\sc Property~{\bf ~4}.} \label{pp:4}{\em % $L^p$-
			Rate of convergence of $\big\|\| X_t-\bar X_t\big\|_{\mathbb{H}}\|_p$  at fixed time $t$ for $p>p_{eu}:= \frac{1}{\theta}\vee \frac{1}{\theta^*}\vee\frac{1}{\delta}$}. For such values of $p$ we know from Proposition ~\ref{prop:pathkernelvolt2} that  Equation~\eqref{eq:pthvolterra} has  a unique  adapted  pathwise continuous solution $(X_t)_{t\in [0,T]}$.
		We recall from the definition~\eqref{def:discretization_scheme}, the equation ~\eqref{eq:def_continuous_2}
		\begin{align*}
			&\ \bar X_t = \bar x_0(t) + \int_0^t K_1(t,s) b \Big(\underline s, \iota_{[\underline s]}\big(\bar X_{t_0:t_{[\underline s]}} \big) \Big)\, ds + \int_0^t K_2(t,s) \sigma \Big(\underline s, \iota_{[\underline s]}\big(\bar X_{t_0:t_{[\underline s]}} \big) \Big)\, d W_s, \quad t\in \mathbb{T}.\\
			&\text{Then,}\; \forall t\!\geq0,\; \text{we have,}\; \P\mbox{-}a.s,\;  X_t - \bar{X}_t = x_0(t) - \bar{x}_0(t)
			+ \int_0^t K_1(t,s) \Big(b(s, X_{\cdot \wedge s}) - b \big(\underline{s}, \iota_{[\underline{s}]}\big(\bar{X}_{t_0:t_{[\underline{s}]}} \big)\big)\Big) ds \\
			&\hspace{7cm} + \int_0^t K_2(t,s) \Big(\sigma(s, X_{\cdot \wedge s}) - \sigma \big(\underline{s}, \iota_{[\underline{s}]}\big(\bar{X}_{t_0:t_{[\underline{s}]}} \big)\big)\Big) dW_s\\
			&= \int_0^t K_1(t,s) \Big(b(s, X_{\cdot \wedge s}) - b \big(\underline{s}, \iota_{[\underline{s}]}\big(\bar{X}_{t_0:t_{[\underline{s}]}} \big)\big)\Big) ds + \int_0^t K_2(t,s) \Big(\sigma(s, X_{\cdot \wedge s}) - \sigma \big(\underline{s}, \iota_{[\underline{s}]}\big(\bar{X}_{t_0:t_{[\underline{s}]}} \big)\big)\Big) dW_s.
		\end{align*}
		We denote by $ A_1(t)$ and $ A_2(t)$ the two terms of the sum  in the right-hand side  of the above equation and we set \(g_p(t) = \displaystyle \sup_{s\in [0,t]} \big\|\| X_s-\bar X_s\big\|_{\mathbb{H}}\|_p, \quad t\!\in \mathbb{T}.\)
		This non-decreasing function is finite owing to Proposition~\ref{prop:pathkernelvolt2} and Property~{\bf 1} since $x_0\!\in L^p(\P)$. We straightforwardly deduce that 
		\begin{equation}
			g_p(t)\le \sup_{s\le t}\big(\big\|\| A_1(s)\big\|_{\mathbb{H}}\|_p + \big\|\| A_2(s)\big\|_{\mathbb{H}}\|_p\big).\label{eq:majogpt1}
		\end{equation}
		Let \( [b]_{H,L} \) denote the mixed H\"older-Lipschitz coefficient for the function \( b \) as given by \(({\cal LH}_{\gamma})\). We have the following inequality for all \( s \in [0,t] \) and for any process \( X^s_\cdot \in \mathbb{X} \):
		\begin{align*}
			&\, \left\|\big\| b(s, X_{\cdot \wedge s}) - b \big(\underline{s}, \iota_{[\underline{s}]}\big(\bar{X}_{t_0:t_{[\underline{s}]}} \big)\big)\big\|_{\mathbb{H}}\right\|_p \le [b]_{H,L}\left( \left\|\Big( \int_{0}^s \| X_{u \wedge s}-\iota_{[\underline{s}]}\big(\bar{X}_{t_0 : t_{[\underline{s}]}}\big)_u\|_{\mathbb{H}}^p \mu (du) \Big)^{\frac1p}\right\|_p \right. \\
			&\hspace{2.3cm} + \left. \Big(1+\left\|\big( \int_{0}^{\underline{s}} \| \iota_{[\underline{s}]}\big(\bar{X}_{t_0 : t_{[\underline{s}]}}\big)_u\|_{\mathbb{H}}^p\, \mu (du) \big)^{\frac1p}\right\|_p+ \left\|\big( \int_{0}^s \| X_{u \wedge s}\|_{\mathbb{H}}^p\, \mu (du) \big)^{\frac1p}\right\|_p\Big)|s-\underline{s}|^{\gamma} \right)\;\text{where}\\
			% \end{align*}
		%\begin{align*}
		&\ \left\|\Big( \int_{0}^s \| X_{u \wedge s}-\iota_{[\underline{s}]}\big(\bar{X}_{t_0 : t_{[\underline{s}]}}\big)_u\|_{\mathbb{H}}^p \mu (du) \Big)^{\frac1p}\right\|_p = \Big( \int_{0}^s \mathbb{E} \left[\| X_{u \wedge s}-\iota_{[\underline{s}]}\big(\bar{X}_{t_0 : t_{[\underline{s}]}}\big)_u\|_{\mathbb{H}}^p\right] \mu (du) \Big)^{\frac1p}\\
		&\qquad \leq \Big( \int_{0}^s \mathbb{E} \left[\| X_{u \wedge s}-\bar X_{u \wedge s}\|_{\mathbb{H}}^p\right] \mu (du) \Big)^{\frac1p} + \Big( \int_{0}^s \mathbb{E} \left[\| \bar X_{u \wedge s}-\iota_{[\underline{s}]}\big(\bar{X}_{t_0 : t_{[\underline{s}]}}\big)_u\|_{\mathbb{H}}^p\right] \mu (du) \Big)^{\frac1p}\\
		&\quad\overset{~\eqref{eq:boundHolIntegrator}}{\operatorname*{\leq}} \mu([0,s])^\frac1p  \sup_{0 \le u \le s} \Big\| \|X_{u}- \bar{X}_{u} \|_{\mathbb H} \Big\|_p+ 2\,\mu([0,s])^{\frac{1}{p}}
		\sup_{0\le u\le s}\Big\|\|\bar X_u - \bar{X}_{\underline u}\|_{\mathbb{H}}^p\Big\|_p \; \text{and for the last two terms,}\\
		% 	\end{align*}
	%	\begin{align*}
		&\ \left\|\big( \int_{0}^{\underline{s}} \| \iota_{[\underline{s}]}\big(\bar{X}_{t_0 : t_{[\underline{s}]}}\big)_u\|_{\mathbb{H}}^p\, \mu (du) \big)^{\frac1p}\right\|_p + \left\|\big( \int_{0}^s \| X_{u \wedge s}\|_{\mathbb{H}}^p\, \mu (du) \big)^{\frac1p}\right\|_p = \Big( \int_{0}^{\underline{s}} \mathbb{E} \left[\| \iota_{[\underline{s}]}\big(\bar{X}_{t_0 : t_{[\underline{s}]}}\big)_u\|_{\mathbb{H}}^p\right] \mu (du) \Big)^{\frac1p}\\
		&\qquad\quad+ \Big( \int_{0}^{s} \mathbb{E} \left[\|X_{u\wedge s}\big)_u\|_{\mathbb{H}}^p\right] \mu (du) \Big)^{\frac1p} \overset{~\eqref{eq:boundIntegrator}}{\operatorname*{\leq}} C\,\mu([0,T])^\frac1p  \sup_{0 \le k \le [\underline{s}]} \Big\| \| \bar{X}_{t_k} \|_{\mathbb H} \Big\|_p + \mu([0,s])^\frac1p  \sup_{0 \le u \le s} \Big\| \| X_{u} \|_{\mathbb H} \Big\|_p \\
		&\leq (1\vee C)\,\mu([0,s])^\frac1p  \sup_{0 \le u \le s} \big(\big\| \| \bar{X}_{u} \|_{\mathbb H} \big\|_p + \big\| \| X_{u} \|_{\mathbb H} \big\|_p \big) \leq 2(1\vee C)\,\mu([0,t])^\frac1p  \sup_{ u \in [0,t]}\big(\big\| \| \bar{X}_{u} \|_{\mathbb H} \big\|_p \vee \big\| \| X_{u} \|_{\mathbb H} \big\|_p \big)
	\end{align*}
	Now, calling upon Lemma~\ref{lem:supEbound} and Property {\bf 2}, then, setting \([b]^{\prime}=[b]_{H,L}\left(1+2\,(1\vee C)\,\mu([0,T])^\frac1p C^{(2)}\right)\), we have the following bound
	\begin{equation}\label{eq:lipHolb}
		\big\|\| b(s, X_{\cdot \wedge s}) - b \big(\underline{s}, \iota_{[\underline{s}]}\big(\bar{X}_{t_0:t_{[\underline{s}]}} \big)\big)\|_{\mathbb{H}}\big\|_p \le [b]^{\prime} \Big(\big(1+ \big\| \| x_0 \|_{t} \big\|_p\big)h^\gamma + g_p(s) +
		\sup_{0\le u\le s}\big\|\|\bar X_u - \bar{X}_{\underline u}\|_{\mathbb{H}}^p\big\|_p \Big)
	\end{equation}
	Similarly, using the mixed H\"older-Lipschitz coefficient \( [\sigma]_{H,L} \) for the function \( \sigma \) as given by \(({\cal LH}_{\gamma})\) and proceeding likewise, we have the following where we define \([\sigma]^{\prime}=[\sigma]_{H,L}\left(1+2\,(1\vee C)\,\mu([0,T])^\frac1p C^{(2)}\right)\).
	\begin{equation}\label{eq:lipHolsigma}
		\big\|\| \sigma(s, X_{\cdot \wedge s}) - \sigma \big(\underline{s}, \iota_{[\underline{s}]}\big(\bar{X}_{t_0:t_{[\underline{s}]}} \big)\big)\|_{\tilde{\mathbb{H}}}\big\|_p \le [\sigma]^{\prime} \Big(\big(1+ \big\| \| x_0 \|_{t} \big\|_p\big)h^\gamma + g_p(s) +
		\sup_{0\le u\le s}\big\|\|\bar X_u - \bar{X}_{\underline u}\|_{\mathbb{H}}^p\big\|_p \Big)
	\end{equation}
	For the first term $A_1(t)$ in equation~\eqref{eq:majogpt1}, the   generalized Minkowski inequality implies the following 
	$
	\big\|\| b(s, X_{\cdot \wedge s}) - b \big(\underline{s}, \iota_{[\underline{s}]}\big(\bar{X}_{t_0:t_{[\underline{s}]}} \big)\big)\|_{\mathbb{H}}\big\|_p \le [b]^{\prime} \Big(\big(1+ \big\| \| x_0 \|_{T} \big\|_p\big)h^\gamma + g_p(s) +
	\sup_{0\le u\le s}\big\|\|\bar X_u - \bar{X}_{\underline u}\|_{\mathbb{H}}^p\big\|_p \Big)$
	\begin{align*}
		&\,\big\|\| A_1(t)\big\|_{\mathbb{H}}\|_p \le  \int_0^t    K_1(t,s)  \left\|\big\| b(s, X_{\cdot \wedge s}) - b \big(\underline{s}, \iota_{[\underline{s}]}\big(\bar{X}_{t_0:t_{[\underline{s}]}} \big)\big)\big\|_{\mathbb{H}}\right\|_p ds\\
		&\hspace{2cm} \overset{~\eqref{eq:lipHolb}}{\operatorname*{\leq}}  [b]^{\prime}\int_0^t  K_1(t,s)\Big(\big(1+ \big\| \| x_0 \|_{t} \big\|_p\big)h^\gamma + g_p(s) +
		\sup_{0\le u\le s}\big\|\|\bar X_u - \bar{X}_{\underline u}\|_{\mathbb{H}}^p\big\|_p \Big)ds\\
		&\hspace{1.5cm} \leq \varphi_1(t) [b]^{\prime}\big(1+ \big\| \| x_0 \|_{t} \big\|_p\big) h^{\gamma} + [b]^{\prime} \left(\int_0^t K_1(t,s) g_p(s)   ds + \int_0^t K_1(t,s) \sup_{0\le u\le s}\big\|\|\bar X_u - \bar{X}_{\underline u}\|_{\mathbb{H}}\big\|_p   ds  \right)  
	\end{align*}
	\smallskip
	\noindent For  the second term $ A_2(t)$, combining   the {\em BDG} inequality and both  regular and generalized Minkowski inequalities implies
	\begin{align*}
		&\,\big\| \|A_2(t)\|_{\mathbb{H}}\big\|_p^2 
		\le  (C^{BDG}_p)^2 \int_0^t K_2^2(t,s)\left\| \| \sigma(s, X_{\cdot \wedge s}) - \sigma \big(\underline{s}, \iota_{[\underline{s}]}\big(\bar{X}_{t_0:t_{[\underline{s}]}} \big)\big)\|_{\tilde{\mathbb{H}}} \right\|_p^2 \, ds \\
		&\hspace{2cm}\overset{~\eqref{eq:lipHolsigma}}{\operatorname*{\leq}} (C^{BDG}_p)^2 ([\sigma]^{\prime})^2\int_0^t K_2(t,s)^2 \Big(\big(1+ \big\| \| x_0 \|_{t} \big\|_p\big)h^\gamma + g_p(s) +
		\sup_{0\le u\le s}\big\|\|\bar X_u - \bar{X}_{\underline u}\|_{\mathbb{H}}^p\big\|_p \Big)^2 ds \\
		& \hspace{1.5cm}\leq 3\psi_2^2(t) (C^{BDG}_p)^2 ([\sigma]^{\prime})^2\big(1+ \big\| \| x_0 \|_{t} \big\|_p\big)^2 h^{2\gamma} + 3(C^{BDG}_p)^2 ([\sigma]^{\prime})^2 \left(\int_0^t K_2(t,s)^2 g_p^2(s) ds + L_2(t)  \right)  
	\end{align*}
	where $L_2(t) := \int_0^t K_2(t,s)^2 \sup_{0\le u\le s}\big\|\|\bar X_u - \bar{X}_{\underline u}\|_{\mathbb{H}}\big\|_p^2   ds  $.
	We derive using elementary computations that 
	\begin{align*}
		\big\| \|A_2(t)\|_{\mathbb{H}}\big\|_p &\le  C^{BDG}_p [\sigma]^{\prime} \Bigg[ \psi_2(t) \left( 1 + \big\| \|x_0\|_t \big\|_p \right) h^{\gamma} + \left( \int_0^t K_2(t,s)^2 g_p^2(s) \, ds \right)^{\frac{1}{2}} + L_2(t)^{\frac{1}{2}}  \Bigg].
	\end{align*}
	Combining the above inequalities yields:
	\begin{align*} \big\| &\,\|A_1(t)\|_{\mathbb{H}}\big\|_p +  \big\| \|A_2(t)\|_{\mathbb{H}}\big\|_p \leq  [b]^{\prime} \left(\int_0^t K_1(t,s) g_p(s) \, ds\right) + C^{BDG}_p[\sigma]^{\prime} \left( \int_0^t K_2(t,s)^2 g_p^2(s) \, ds \right)^{\frac{1}{2}}+ R(t)\\
		&\hspace{4cm}\leq \left([b]^{\prime}\varphi_{\frac{2\beta}{\beta+1}} + C^{BDG}_p[\sigma]^{\prime}\psi_{2\beta} \right)\left( 
		\frac{\rho}{A^{\frac1\rho}}\,g_p(t) 
		+ (1-\rho)\,A^{\frac{1}{1-\rho}} \int_0^t g_p(r)\,dr
		\right) + R(t).
	\end{align*} 
	where we applied relation~\eqref{eq:VoltIntegral} from Lemma~\ref{Lem:BoundVoltIntegral} for all \( A > 0 \) for \(i=1,2\) as \(g_p\) is non-decreasing with \(\rho \in (0,1)\), and setting \(C^{(5)}=C^{(5)}_{K_1,K_2, \beta, b,\sigma, p } = \displaystyle \max\big(\varphi_1^*(T) [b]^{\prime}+ C^{BDG}_p [\sigma]^{\prime} \psi_2^*(T), [b]^{\prime}, C^{BDG}_p  [\sigma]^{\prime}\big)\) we have:
	\begin{align*}
		R(t) &= C^{(5)} \left( \Big(1 + \big\| \|x_0 \|_{t} \big\|_{p}\Big)h^{\gamma} + \int_0^t K_1(t,s) \sup_{0\le u\le s}\big\|\|\bar X_u - \bar{X}_{\underline u}\|_{\mathbb{H}}\big\|_p   ds +L_2(t)^{\frac{1}{2}}  \right) \\
		&\leq  C^{(5)} \left( \Big(1 + \big\| \|x_0 \|_{t} \big\|_{p}\Big)h^{\gamma} + \sup_{s\in [0,t]}\big\|\|\bar X_s- \bar{X}_{\underline s}\|_{\mathbb{H}}\big\|_p \, \varphi_1(t) + \sup_{s\in[0,t] }\big\|\|\bar X_s - \bar{X}_{\underline s}\|_{\mathbb{H}}\big\|_p\,\psi_2(t) \right)
	\end{align*} 
	Now, passing to the supremun on \([0,t]\) as in~\eqref{eq:majogpt1} i.e. \(g_p(t) \leq
	\sup_{s\in[0, t]}(\big\| \|A_1(s)\|_{\mathbb{H}}\big\|_p +  \big\| \|A_2(s)\|_{\mathbb{H}}\big\|_p )\) while taking advantage of the fact that \(g_p\) is non-decreasing, we derive that
	\begin{align*}
		&\, \displaystyle \sup_{s\in [0,t]}\big\|\| X_s-\bar X_s\big\|_{\mathbb{H}}\|_p=:g_p(t) \leq C\,\left( 
		\frac{\rho}{A^{\frac1\rho}}\,g_p(t) 
		+ (1-\rho)\,A^{\frac{1}{1-\rho}} \int_0^t g_p(s)\,ds
		\right) + \sup_{s\in[0, t]}R(s).\\
		&\text{where}\;
		C = C_{K_1,K_2,p,\beta,b,\sigma,T} =  [b]^{\prime}\,\varphi^{*}_{\frac{2\beta}{\beta+1}} 
		+ C^{BDG}_p [\sigma]^{\prime}\,\psi^{*}_{2\beta}<+\infty \;\text{owing to }\; ({\cal K}^{int}_{\beta}).
	\end{align*}
	We now choose \(A\) large enough, namely \(A = (2\rho\, C)^{\rho}\), so that 
	\(1 - \rho\, C\,A^{-\frac1\rho} = \frac{1}{2} < 1\).
	Consequently, for every \(t \in [0,T]\),
	\[g_p(t) \le 2\sup_{s\in[0, t]}R(s)
	+ 2C\,(1-\rho)\,A^{\frac{1}{1-\rho}} \int_0^t g_p(s)\,ds\le 2\sup_{s\in[0, t]}R(s) e^{C^{(6)}\,t},
	\text{and} \;
	C^{(6)}_{K_1,K_2,p,\beta,b,\sigma,T}
	= 2C(1 - \rho)\,A^{\frac1{1-\rho}}.\]
	where we used Gr\"onwall's lemma (See for example \cite[Lemma 7.2]{pages2018numerical}) in the last inequality and in particular
	\[ \displaystyle \sup_{t\in [0,T]}\big\|\| X_t-\bar X_t\big\|_{\mathbb{H}}\|_p=:g_p(T) \leq 2\sup_{t\in[0, T]}R(t) e^{C^{(6)}\,T}. \]
	\noindent To conclude, coming back to $R(s)$, we have that, for every $t\in [0,T]$, 
	\[
	\sup_{s\le t} R(s) \le C^{(5)} \Big(1 + \big\| \|x_0 \|_{t} \big\|_{p}\Big)h^{\gamma} + C^{(5)}\big(\varphi^*_1(T)+\psi^*_{2}(T)\big) \sup_{s\in [0,t]}\big\|\|\bar X_s- \bar{X}_{\underline s}\|_{\mathbb{H}}\big\|_p.
	\]
	By combining these bounds with Property~{\bf 3},   one concludes  that, under $(\widehat {\cal K}^{cont}_{\widehat \theta})$
	\[ \displaystyle \sup_{t\in [0,T]}\big\|\| X_s-\bar X_s\big\|_{\mathbb{H}}\|_p=:g_p(T)
	\le C^{(7)}\Big(1 + \big\| \|x_0 \|_{T} \big\|_{p}\Big) \left(h^{\gamma}+h^{\theta\wedge\widehat\theta}\right)
	\]
	with $C^{(7)}= C^{(7)} _{K_1,K_2, \beta, b,\sigma, p,T } = 2\,e^{C^{(6)} T}C^{(5)}\max\left(1,(\varphi^*_1(T)+\psi^*_{2}(T)\big)\right)$.
	Note that if only $({\cal K}^{int}_{\beta})$ is in force, then the final bound holds {\em mutatis mutandis} with $\frac{\beta-1}{2\beta}$ instead of $\widehat \theta$. 
	As a conclusion to this step, we get:
	\[
	\forall\, t\!\in [0,T], \quad  \big\|\| X_s-\bar X_s\big\|_{\mathbb{H}}\|_p \le C^{(7)} \Big(1 + \big\| \|x_0 \|_{T} \big\|_{p}\Big) (h^{\gamma}+h^{\delta \wedge \theta^* }).
	\]
	\paragraph{Proof of Theorem~\ref{thm:Eulercvgce2} (convergence of the genuine interpolated $K$-integrated Euler scheme).}
	%  \noindent {\bf Proof of Theorem~\ref{thm:Eulercvgce2} (convergence of the $K$-integrated Euler scheme).}
	\medskip
	In this paragraph, we provide the proof of Theorem~\ref{thm:Eulercvgce2}, based on the properties of the  continuous extension process $(\bar X_t)_{t \ge 0}$  established in~\ref{propXtilde} and the splitting Lemma
	established in~\ref{lem:gap}.
	
	\medskip
	\noindent {\sc Step~1} ({\em Moment control, $p\!\in (0,2)$}). Let us denote 
	$X^{\xi}= (X^{\xi}_t)_{t\in [0,T]}$ and 
	$\bar X^{\xi}= (\bar X^{\xi}_t)_{t\in [0,T]}$ the 
	solutions to the Volterra equation~\eqref{eq:pthvolterra} and its 
	Euler scheme~\eqref{eq:def_continuous_2}  with starting function $x_0=\xi$ respectively.
	Property~{\bf 2} in Proposition~\ref{propXtilde} applied with $p=2$
	implies that there exists a real constant $C=C_{K_1, K_2, \beta, b,\sigma, T}>0$ (not depending on the time step $\frac Tn$ i.e. on $n$) such that,  if $x_0=\xi$, \( \sup_{t\in [0,T]}\big\|\|\bar X^{\xi}_t\|_{\mathbb{H}}\big\|_2  
	\le C (1 +\|\xi\|_T).\)
	Then it follows from Lemma~\ref{lem:gap} applied with $\bar p=2$ and $\Phi(x,y) =\sup_{t\in[0,T]}\|y(t)\|_{\mathbb{H}}$ that
	\[\forall\, p\!\in (0,2], \quad\sup_{t\in [0,T]}\big\|\|\bar X_t\|_{\mathbb{H}}\big\|_p  
	\le 2^{(1/p-1)^+}C (1+\|\|\xi\|_T\|_p).\]
%	\centerline{$\forall\, p\!\in (0,2], \quad\sup_{t\in [0,T]}\big\|\|\bar X_t\|_{\mathbb{H}}\big\|_p  
%		\le 2^{(1/p-1)^+}C (1+\|\|\xi\|_T\|_p).$}
	\medskip 
	\noindent {\sc Step~2} ({\em  $L^p$-control of the increments of the $K$-integrated Euler scheme for $p\!\in (0,2)$}). 
	If {$p\!\in (0,2)$}, we proceed likewise replacing Property~{\bf 2} in Proposition~\ref{propXtilde} by Property {\bf 3}, with the functional $\Phi(x,y) = \|y(t)-y(s)\|_{\mathbb{H}}$, $s,\, t\!\in [0,T]$ in order  to prove that
	\[
	\forall\, s,\, t\!\in [0,T], \quad  \left\| \|\bar X_t-\bar X_s\|_{\mathbb{H}} \right\|_p\le 2^{(1/p-1)^+}C_{2,T}\Big(1 + \big\| \|x_0 \|_{T} \big\|_{p}\Big) |t-s|^{\delta \wedge \theta \wedge \theta^*}.
	\]
	\noindent {\sc Step~3} ({\em $L^p$-rate of convergence for the marginals at fixed time $t$ {for $p>0$}}).  
	We fix $\bar p>p_{eu}$ and combine Property~{\bf 4} in Proposition~\ref{propXtilde} and Lemma~\ref{lem:gap} applied with  $\Phi(x,y) =  \|x(t)-y(t)\|_{\mathbb{H}}$ to $p\in (0,\bar p]$.

	\smallskip
	\noindent{\sc Step~4} ({\em $L^p$-rate of convergence for the sup-norm, $p$ large enough}). To switch from the $L^p$-convergence of marginals, namely $\big\|\| X_t-\bar X_t\big\|_{\mathbb{H}}\|_p$
	to this one (\(\big\|\sup_{ t \in \T}\| X_t-\bar X_t\big\|_{\mathbb{H}}\|_p\)), we rely on~\cite[Corollary 4.4]{RiTaYa2020} (see Theorem~\ref{thm:GRRtypeLemma} in Appendix~\ref{app:B}) which is a form close to Kolmogorov's $C$-tightness criterion of the Garsia-Rodemich-Rumsey lemma (see~\cite{GRR1070}).
	
%	\begin{Theorem}[GRR lemma]\label{thm:GRRtypeLemma} Let $(Y^n)_{n \ge 1}$ be a sequence of continuous processes where the processes $Y^n= (Y^n_t)_{t\in[0,T]}$ are defined on a probability space $(\Omega,{\cal A}, \P)$. Let $p\ge 1$. Assume  there exists  $a>1$, a sequence $(\delta_n)_{n\ge 1}$ of positive real numbers  converging to $0$ and  a real constant $C>0$ such that 
%		\begin{equation}\label{eq:GRR}
%			\forall\, n\ge 1,\; \forall\, s, t\!\in [0,T], \quad \E\, |Y^n_t-Y^n_s|^p \le C|t-s|^{a} \delta^p_n.
%		\end{equation}
%		Then 
%		there exists a real constant $C_{p,T}>0$ such that
%		\[
%		\forall\, n\ge 1,\quad \E\, \sup_{t\in [0,T]}|Y^n_t-Y^n_0|^{p} \le C_{p,T}\,\delta^p_n.
%		\]  
%	\end{Theorem}
	Let $\varepsilon\!\in (0,1)$. We aim at applying the above inequality to the sequence of processes $Y^n = X-\bar X$ to prove that the order of convergence of the Euler scheme derived at Property~{\bf 4} in Proposition~\ref{propXtilde} is preserved for the supremum over time up to multiplication by the factor $1-\varepsilon$.
	
	\smallskip
	Set $\theta^{*} =\delta\wedge\theta\wedge\widehat\theta $. We first deal with large values of $p$, namely {$p>p_{\varepsilon} =\frac{2}{(\delta\wedge\theta\wedge\theta^{*})\varepsilon}\vee p_{eu}$. Then let $\lambda := \frac{1}{p(\delta\wedge\theta\wedge\theta^{*})}+\frac{\varepsilon}{2}\!\in (0,\varepsilon)$}. For every $s$, $t\!\in [0,T]$, one has
	\begin{align*}
		\big\|\|Y^n_t-Y^n_s\big\|_{\mathbb{H}}\|_p &\le \big\|\|Y^n_t-Y^n_s\big\|_{\mathbb{H}}\|_p^{\lambda} \big\|\|Y^n_t-Y^n_s\big\|_{\mathbb{H}}\|_p^{1-\lambda}\\
		&\le  \big(\big\|\|X_t-X_s\big\|_{\mathbb{H}}\|_p + \big\|\| \bar X_t-\bar X_s \big\|_{\mathbb{H}}\|_p \big)^{\lambda} 
		\big(\big\|\|X_t-\bar X_t\big\|_{\mathbb{H}}\|_p + \big\|\|X_s-\bar X_s \|_{\mathbb{H}}\|_p \big)^{1-\lambda}.
	\end{align*}
	By Property~{\bf 4} in Proposition~\ref{propXtilde}, there exists a positive real constant \(C_1\) such that 
	\[
	\sup_{t\in [0,T]} \big\|\|X_t-\bar X_t\big\|_{\mathbb{H}}\|_p \le C_1  \Big(1 + \big\| \|x_0 \|_{T} \big\|_{p}\Big) \big( \tfrac Tn \big)^{\delta\wedge\theta\wedge\theta^{*}\wedge \gamma}.
	\]
	With Property~{\bf 3}, still from Proposition~\ref{propXtilde}, we deduce the existence of a positive real constant \(C_2\) such that 
	\[
	\forall\, s,\, t\!\in [0,T], \quad \big\|\| X_t-X_s\big\|_{\mathbb{H}}\|_p + \big\|\| \bar X_t-\bar X_s \|_{\mathbb{H}}\|_p \le C_2 \Big(1 + \big\| \|x_0 \|_{T} \big\|_{p}\Big)|t-s|^{\delta\wedge\theta\wedge\theta^{*}}.
	\]
	Hence,
	using that   $a:=\lambda (\delta\wedge\theta\wedge\theta^{*}) p=1+ \frac{\varepsilon}{2}\,(\delta\wedge\theta\wedge\theta^{*})\,p >1$,  we derive
	\[
	\E\, \|Y^n_t-Y^n_s\|_{\mathbb{H}}^p \le C^p_3  \Big(1 + \big\| \|x_0 \|_{T} \big\|_{p}\Big)^p|t-s|^{a} \big( \tfrac Tn \big)^{p(\delta\wedge\theta\wedge\theta^{*}\wedge \gamma)(1-\lambda)} 
	\]
	with  $C_3= (2C_1)^{1-\lambda}C_2^{\lambda}$. Hence, as $Y^n_0=0$ and $a>1$, it follows from the above {\em GRR} lemma~\ref{thm:GRRtypeLemma} % \cite[Corollary 4.4]{RiTaYa2020} 
	that there exists 
	a positive real constant $\kappa_{\varepsilon,p} =\kappa_{\varepsilon,p,\beta, \gamma, \theta, b,\sigma,T}$ such that, for every $n\ge 1$, 
	\[
	\Big\| \sup_{t\in [0,T]} \|X_t-\bar X_t\|_{\mathbb{H}}\Big\|_{p} = \Big\| \sup_{t\in [0,T]} \|Y^n_t-Y^n_0\|_{\mathbb{H}} \Big\|_{p}  \le \kappa_{\varepsilon,p}  \Big(1 + \big\| \|x_0 \|_{T} \big\|_{p}\Big) \big( \tfrac Tn \big)^{(\delta\wedge\theta\wedge\theta^{*}\wedge \gamma)(1-\lambda)} .
	\]
	As $\lambda \le \varepsilon$, it is clear that, up to an updating of the constant $\kappa_{\varepsilon,p}$ to take into account the values of $n\!\in\{1,\ldots,\lfloor T\rfloor\}$, one has for every $n\ge 1$, 
	\[
	\Big\| \sup_{t\in [0,T]} \|X_t-\bar X_t\|_{\mathbb{H}}\Big\|_{p}    \le \kappa_{\varepsilon,p}  \Big(1 + \big\| \|x_0 \|_{T} \big\|_{p}\Big) \big( \tfrac Tn \big)^{(\delta\wedge\theta\wedge\theta^{*}\wedge \gamma)(1-\varepsilon)} .
	\]
	Now we take advantage of the splitting Lemma
	established in~\ref{lem:gap}  to extend this result to the low integrability setting of $x_0$, i.e. $\|x_0\|_{T}\!\in L^p(\P)$, $p\in (0,p_\varepsilon]$. 
	
	\medskip
	\noindent {{\sc Step~5} ({\em $L^p$-rate of convergence of the supremum over time, $p\!\in(0, p_{\varepsilon})$})}. It follows from Lemma~\ref{lem:gap} {applied with $\bar p= p_{\varepsilon}+1$ and 
		$\Phi(x,y)= \sup_{t\in [0,T]}\|x(t)-y(t)\|_{\mathbb{H}}$  that 
		\begin{align*}
			\Big\| \sup_{t\in [0,T]} \|X_t-\bar X_t\|_{\mathbb{H}}\Big\|_p 
			&\le 2^{(1/p-1)^+} (\kappa_{\varepsilon,p_{\varepsilon}+1}) \Big(1 + \big\| \|x_0 \|_{T} \big\|_{p}\Big) \Big(\frac Tn\Big)^{(\delta\wedge\theta\wedge\theta^{*}\wedge \gamma)(1-\varepsilon)}.
	\end{align*}}
	\noindent This  completes the proof of this step and of the theorem and therefore, the result is established.~\hfill $\Box$
	\medskip
%\noindent {\bf Acknowledgements:} 
% We would like to express our sincere gratitude to [acknowledge individuals, organizations, or institutions] for their invaluable contributions to this research. We are also grateful to [mention any additional acknowledgements, such as technical assistance, data providers, or colleagues] for their support and assistance throughout the course of this work.

    \medskip
    \noindent\textbf{Declarations:} The authors declare no conflicts of interest; the first author was supported by ``École Doctorale Sciences Mathématiques de Paris Centre".

	\bibliographystyle{plainnat}
	\bibliography{PathDependentSVIE}
	\normalsize  
	
	\appendix
	\section{About the Simulation of the Gaussian  stochastic integrals terms in the discrete time interpolated semi-integrated Euler scheme for path-dependent stochastic Volterra integral Equations (SVIEs)}\label{app:A}
	Let us denote by $ G^{n,\ell} $, $\ell \in \{1, ..., n\}$ the $\ell$-th column of the matrix $ G^n $ without the zero-elements, that is the Gaussian (column) vector of size $n - \ell + 2 $:
	\begin{equation}\label{eq:Cell}
		G^{n,\ell} = \big( G^n_{\ell-1+k,\ell}\big)_{k=1:n-\ell+2}=\left( \left[\int_{t^n_{\ell}}^{t^n_{\ell+1}} K_2(t^n_k, s) \, dW_s \right]_{k=\ell,\ldots,n}, \Delta W_{t_\ell} \right) =\left( \left[I^{n,\ell}_k \right]_{k=\ell,\ldots,n},\Delta W_{t_\ell} \right).
	\end{equation}
	Note that \(G^{n,\ell},\; \ell= 1,\ldots,n\)	are $n$ independent Gaussian vectors that we consider and will simulate.

	The $(n-\ell +2)\times (n-\ell +2)$ symmetric covariance matrix of $ G^{n,\ell} $, denoted by $C^{n+1, \ell}$, is given by:
	$$C^{n+1, \ell} =
	\begin{pmatrix}
		\Sigma^{n,\ell} & \begin{array}{ccc} ^t C^{0,\ell} \end{array} \\[10pt]
		\begin{array}{c} C^{0,\ell} \end{array} & \frac{T}{n}  
	\end{pmatrix}, \;C^{0,\ell}  = \left[ Cov(\Delta W_{t_\ell}, I^{n,\ell}_{i}) \right]_{\ell\le i \le n} =  \left[ \int_0^{T/n} K_2(t^n_i, t^n_{\ell-1}+u) \, du \right]_{\ell\le k\le n}.
	$$
	where $\Sigma^{n,\ell}$ is an $(n-\ell+1)\times (n-\ell+1)$ positive definite symmetric matrix corresponding to the covariance matrix of $\left(\left[I^{n,l}_k \right]_{k=\ell,\ldots,n} \right)$, given by \(\Sigma^{n,\ell} = \left[ Cov(I^{n,\ell}_i, I^{n,\ell}_{j}) \right]_{i,j \in \llbracket \ell, n \rrbracket} \) i.e.: 
	{\small
	\begin{equation}\label{eq:VCV1}
		\Sigma^{n,\ell} = \left[\int_{t_{\ell-1}}^{t_\ell} K_2(t_i, u) K_2(t_{j}, u) du \right]_{i,j \in \llbracket \ell, n \rrbracket}= \left[ \int_0^{T/n} K_2(t_i, t_{\ell-1}+u) K_2(t_j, t_{\ell-1}+u) \, du \right]_{\ell \leq i,j \leq n}.
	\end{equation}
	}
	or reading equivalently \(\Sigma^{n,\ell}=(\Sigma^{n,\ell}_{ij})_{i,j=1:n-\ell+1} \in \mathbb{R}^{(n-\ell+1)\times (n-\ell+1)}\;\mbox{ such that}\;\)
	\begin{equation}\label{eq:VCV2}
	%	\Sigma^{n,\ell}=(\Sigma^{n,\ell}_{ij})_{i,j=1:n-\ell+1} \in \mathbb{R}^{(n-\ell+1)\times (n-\ell+1)}\;\mbox{ such that}\;
	 \Sigma^{n,\ell}_{ij}= \int_{t_{\ell-1}}^{t_\ell} K_2(t_{\ell+i-1},s)K_2(t_{\ell+j-1},s)  \, ds =\int_0^{T/n} K_2(t_i, t_{\ell-1}+u) K_2(t_j, t_{\ell-1}+u) \, du
	\end{equation}
	Finally, $C^{0,\ell}$ is the line  vector reading $C^{0,\ell}=\left(\displaystyle  \int_{t_{\ell - 1}}^{t_{\ell}} K_2(t_{t_k}, s) \, ds \right)_{k=\ell:n}\!\in \R^{1\times (n-\ell+1)}$ with nonnegative components.\\
	At this stage, we can compute any fixed matrix  $C^{n+1, \ell}$ by a cubature formula (such as Trapezoid, Midpoint, Simpson, higher-order Newton-Cote, or Gauss-Legendre integration formulas) and then perform a Cholesky transform.  $C^{n+1, \ell}$ being a  positive definite matrix has a (unique) Cholesky decomposition $C^{n+1, \ell} = L^{n+1,\ell}\,  ^t(L^{n+1,\ell})$ where $L^{n+1,\ell}$ is lower triangular with positive diagonal terms.

   However the  regular Cholesky procedure is not stable when some eigenvalues are too close to zero due to the propagation of  rounding errors and the presence of square roots in the procedure.
	We then apply a numerically stable extended version of the Cholesky decomposition method on each $C^{n+1, \ell}$, namely the $T\,D\,^tT$ \textit{decomposition} in the sense that, there exists a unique lower triangle matrix $T^{ (\ell)}$ with $T^{(\ell)}_{ii}=1$, $ i=1:n$ and a unique diagonal matrix $D^{(\ell)}$ with positive entries such that $C^{n+1, \ell} =T^{(\ell)}\,D^{(\ell)}\,^tT^{(\ell)}$. 
	Once this decomposition is computed, we recover the classical Cholesky decomposition $C^{n+1, \ell} =L^{ \ell}\,^tL^{(\ell)}$ taking the  (lower triangular) matrix $L^{(\ell)}=T^{(\ell)}\sqrt{D^{ \ell}}$. This algorithm turns out to be more stable than the regular Cholesky procedure.
	Finally,  we can simulate $G^{n,\ell}$ via the equality in distribution
	\begin{equation*}
		G^{n, \ell} \stackrel{\mathcal L}{=} L^{(\ell)} Z^{(\ell)}, \text{ with } Z^{(\ell)} \sim \mathcal{N}(0,\mathcal{I}_{n-\ell+1}). 
	\end{equation*}
	From an algorithmic point of view, we only have
	to compute one Cholesky decomposition, by considering the covariance matrix $C:=C^{n+1, 1}$, of $G^{n,1}$, symmetric and of size $(n+1) \times (n+1)$, given by:
	\begin{equation}\label{eq:covariance-matrix}
		C =
	\begin{pmatrix}
		\Sigma^{n,1} & \begin{array}{ccc} ^t C^{0,1} \end{array} \\[10pt]
		\begin{array}{c} C^{0,1} \end{array} &  \frac{T}{n}
	\end{pmatrix}, \;\Sigma^{n,1} = \left[ Cov(I^{n,1}_i, I^{n,1}_{j}) \right]_{i,j \in \llbracket 1, n \rrbracket},\;C^{0,1}  = \left[ Cov(\Delta W_{t_\ell}, I^{n,1}_{i}) \right]_{1\le i \le n}.
    \end{equation}
	\begin{Proposition}\label{prop:EfficientSimul} Let the number of discretization \(n\) be fixed so as the step size $\frac Tn$.
		\begin{enumerate}
		\item [$(a)$] Let $C^{n+1, 1}:=C\!\in \R^{(n+1)\times (n+1)}$ be the symmetric matrix with entries defined in equation~\eqref{eq:covariance-matrix}.
		Then all matrices $C^{n+1, \ell}$, $\ell \in \{1,\ldots,n\}$,  from~\eqref{eq:Cell} are sub-matrices of $C$ in the sense that $C^{n+1, \ell}= [C_{ij}]_{i,j= \ell:n+1}$.
		
	\item [$(b)$] The \textit{extended Cholesky decomposition} or $T\, D\, ^tT$ \textit{decomposition}  of the positive definite matrix $C = (C_{ij})_{1 \le i,j \le n+1}$ is as follows:
		\[ [C_{ij}]_{1\le i,j\le n+1}= T^{(n+1)}\,D^{(n+1)}\,^tT^{(n+1)},\;  T^{(n+1)} \text{ lower triangular with diagonal entries}\; T^{(n+1)}_{ii} = 1.
		\]
	  and $D^{(n+1)}$ is a diagonal matrix with non-negative entries. 
		Then, taking advantage of the telescopic feature and the structure of this Cholesky transform one has:	$$[C_{ij}]_{\ell\le i,j\le n+1}=  [T_{ij}^{(n+1)}]_{\ell\le i,j\le n+1}[D^{(n+1)}_{ij}]_{\ell\le i,j\le n+1} [^tT_{ij}^{(n+1)}]_{\ell\le i,j\le n+1},\ell= 1,\ldots,n.$$
		\noindent Finally, for each $\ell = 1, \ldots, n$, we can effectively simulate $G^{n,\ell}$ via the equality in distribution \((G^{n,\ell})_{\ell=1, \dots, n} \stackrel{\mathcal{L}}{=} (L^{(n+2-\ell)} Z^{(\ell)})_{\ell=1, \dots, n},\) where 
		\[
		Z^{(\ell)} \sim \mathcal{N}(0, I_{n-\ell+2}) \quad \text{and }\quad L^{(n+2-\ell)} = [T_{ij}^{(n+1)}]_{\ell\leq i,j \leq n +1} [\sqrt{D^{(n+1)}_{ij}}]_{\ell \leq i,j \leq n +1}.
		\]
	\end{enumerate}
	\end{Proposition}
	\noindent {\bf Proof:} The claim in $(b)$  follows straightforwardly from what has been stated before the proposition.
	For $(a)$ 
		Let $\ell \in \{1,...,n\}$ and and $i,j=\ell, \dots,n+1$. It follows from~\eqref{eq:Cell} that         
		\begin{align*}
			C^{n+1, \ell}_{ij} &=\int_{t_{\ell-1}}^{t_\ell} K_2(t_{i},s) K_2(t_{j},s)ds \stackrel{(s=u+t_{\ell-1})}{=} \int_{0}^{\frac{T}{n}} K(t_{i},t_{\ell-1}+u) K(t_{j},t_{\ell-1}+u)du =  C_{ij},\quad \text{etc.}
		\end{align*}
	\section{Technical Lemmata and Proofs}\label{app:B}
	\noindent {$\rhd$ {\em Preliminaries}:}
	We recall~\cite[Corollary 4.4]{RiTaYa2020}, which is a version close to Kolmogorov's $C$-tightness criterion of the Garsia-Rodemich-Rumsey lemma~\citep{GRR1070}.
	\begin{Theorem}[GRR lemma]\label{thm:GRRtypeLemma} Let $(Y^n)_{n \ge 1}$ be a sequence of continuous processes where the processes $Y^n= (Y^n_t)_{t\in[0,T]}$ are defined on a probability space $(\Omega,{\cal A}, \P)$. Let $p\ge 1$. Assume  there exists  $a>1$, a sequence $(\delta_n)_{n\ge 1}$ of positive real numbers  converging to $0$ and  a real constant $C>0$ such that 
		\begin{equation}\label{eq:GRR}
			\forall\, n\ge 1,\; \forall\, s, t\!\in [0,T], \quad \E\, |Y^n_t-Y^n_s|^p \le C|t-s|^{a} \delta^p_n.
		\end{equation}
		Then 
		there exists a real constant $C_{p,T}>0$ such that
		\[
		\forall\, n\ge 1,\quad \E\, \sup_{t\in [0,T]}|Y^n_t-Y^n_0|^{p} \le C_{p,T}\,\delta^p_n.
		\]  
	\end{Theorem}
	\subsection{Proofs of Technical Lemma~\ref{lem:bounds}, Theorem \ref{thm:FreidlinLike} and Corollary \ref{Gronwall}}\label{app:B1}
	\subsubsection{Proof of Lemma~\ref{lem:bounds} }
%	\noindent {\bf Proof of Lemma~\ref{lem:bounds}:}
	With Equation \ref{eq:contKtilde}, i.e under Assumption $(\widehat {\cal K}^{cont}_{\widehat \theta})$, let $s,\, t\!\in [a, T]$. By using successively the triangle inequality in 
	\(\mathbb{H}\) for Bochner integrals, then the Burkholder-Davis-Gundy  and the elementary inequality \((a+b)^p \le 2^{p-1}\big(a^p + b^p\big)\) for \(a,b>0\), one gets:
	{\small
		\begin{align*}
			&\mathbb{E}\left[\left\|\int_{a}^{t\wedge b} K_2(t,u)\widetilde{H}_u dW_u - \int_{a}^{s\wedge b} K_2(s,u)\widetilde{H}_u dW_u\right\|_{\tilde{\mathbb{H}}}^p \right] \\
			&\phantom{ K_2(t,u)} \le 2^{p-1}C_p^{BDG} \left( \mathbb{E}\left[ \left(\int_{s\wedge b}^{t\wedge b}\|K_2(t,u)^2 \widetilde{H}_u^2 \|_{\tilde{\mathbb{H}}}du \right)^{\frac p 2}\right] + \mathbb{E}\left[ \left(\int_{a}^{s\wedge b}\|\big(K_2(t,u) - K_2(s,u)\big)^2 \widetilde{H}_u^2  \|_{\tilde{\mathbb{H}}}du \right)^{\frac p 2}\right] \right) \\
			&\phantom{ K_2(t,u)}\le 2^{p-1} C_p^{BDG} \left( \mathbb{E}\left[ \left(\int_{s}^{t} |K_2(t,u)|^2 \|\widetilde{H}_u\|_{\tilde{\mathbb{H}}}^2 du \right)^{\frac{p}{2}}\right] + \mathbb{E}\left[ \left(\int_{a}^{s} |(K_2(t,u) - K_2(s,u))|^2 \|\widetilde{H}_u\|_{\tilde{\mathbb{H}}}^2 du \right)^{\frac{p}{2}}\right] \right)
		\end{align*}
	}
	Jensen's inequality applied twice or rather H\"older inequality %  \((\frac{p}{p-2},\frac{p}{2})\) Jensen's inequality % or rather H\"older inequality with exponent \((\frac{p}{p-2},\frac{p}{2})\)
	\hyperref[fn:first]{\footnotemark[\getrefnumber{fn:first}]} with exponent \(\frac{p}{2}\geq1\) yields:
	\[\left(\int_{s}^{t} |K_2(t,u)|^2 \|\widetilde{H}_u\|_{\tilde{\mathbb{H}}}^2 du \right)^{\frac{p}{2}} \leq \left(\int_{s}^{t} |K_2(t,u)|^2 du \right)^{\frac{p}{2}-1} \int_{s}^{t} |K_2(t,u)|^2 \|\widetilde{H}_u\|_{\tilde{\mathbb{H}}}^p du  \]
	Taking expectation, we obtain
	\[\mathbb{E}\left[ \left(\int_{s}^{t} |K_2(t,u)|^2  \|\widetilde{H}_u\|_{\tilde{\mathbb{H}}}^2 du \right)^{\frac{p}{2}}\right] \leq \left(\int_{s}^{t} |K_2(t,u)|^2 du \right)^{\frac{p}{2}} \sup_{u \in [0,T]} \mathbb{E}\left[\|\widetilde{H}_u\|_{\tilde{\mathbb{H}}}^p\right] \]
	In a similar manner,
	\begin{align*} 
		&\, \mathbb{E}\left[ \left(\int_{a}^{s} |(K_2(t,u) - K_2(s,u))|^{2} \|\widetilde{H}_u\|_{\tilde{\mathbb{H}}}^2 du \right)^{\frac{p}{2}} \right] \leq \left(\int_{a}^{s} |(K_2(t,u) - K_2(s,u))|^2 du \right)^{\frac{p}{2}} \sup_{u \in [0,T]} \mathbb{E}\left[\|\widetilde{H}_u\|_{\tilde{\mathbb{H}}}^p\right] \\
		%	Consequently,
		&\, \text{Consequently,}\quad \mathbb{E}\left[\left\|\int_{a}^{t\wedge b} K_2(t,u)\widetilde{H}_u dW_u - \int_{a}^{s\wedge b} K_2(s,u)\widetilde{H}_u dW_u\right\|_{\tilde{\mathbb{H}}}^p \right] \\
		&\quad \le 2^{p-1} C_p^{BDG} \Bigg(\sup_{u \in [0,T]} \mathbb{E}\left[\|\widetilde{H}_u\|_{\tilde{\mathbb{H}}}^p\right] \Bigg)  \times \Bigg( \left( \int_{s}^{t} |K_2(t,u)|^{2} du\right)^{\frac p 2}  + \left( \int_a^s |(K_2(t,u) - K_2(s,u))|^2 du\right)^{\frac p 2} \Bigg)
	\end{align*}
	
	Hence owing to ~\eqref{eq:contKtilde} and  ~\eqref{eq:contK}, there exists a real constant $\kappa= \kappa_{p}$ only depending on $p$ such that
	\[
	\mathbb{E}\left[\left\|\int_{a}^{t\wedge b} K_2(t,u)\widetilde{H}_u dW_u - \int_{a}^{s\wedge b} K_2(s,u)\widetilde{H}_u dW_u\right\|_{\tilde{\mathbb{H}}}^p \right]  \le \kappa\,\left(\sup_{u \in [0,T]}\,\mathbb{E}\left[ |a_u|^{\frac{p}{2}} \right]\right) (t-s)^{p( \theta  \wedge \widehat \theta)}
	\]
	Likewise, with Equation \ref{eq:contKtilde}, i.e under Assumption $(\widehat {\cal K}^{cont}_{\widehat \theta})$, let $s,\, t\in \mathbb{T}_-$. Using the triangle inequality in 
	\(\mathbb{H}\) for Bochner integrals and the elementary inequality \((a+b)^p \le 2^{p-1}\big(a^p + b^p\big)\) for \(a,b>0\), 
	 one gets with $H_u := b(u, X_{\wedge u})$:
	\begin{align*}
		&\mathbb{E}\left[\left\|\int_{a}^{t\wedge b} K_1(t,u) H_u du - \int_{a}^{s\wedge b} K_1(s,u) H_u du \right\|_{\mathbb{H}}^p \right] \\
		&\le 2^{p-1} \left( \mathbb{E}\left[ \left(\int_{s}^{t} |K_1(t,u)| \|H_u\|_{\mathbb{H}} du \right)^{p}\right] + \mathbb{E}\left[ \left(\int_{a}^{s} |(K_1(t,u) - K_1(s,u))| \|H_u\|_{\mathbb{H}} du \right)^{p}\right] \right).
	\end{align*}
%	By using generalized Minkowski inequalities, we obtain By cauchy-schwarz inequality, then 
	By Jensen's inequality applied twice or rather H\"older inequality with exponent \((\frac{p}{p-1},p)\) , we obtain:
	
	\[\left(\int_{s}^{t} |K_1(t,u)| \|H_u\|_{\mathbb{H}} du \right)^{p} \leq  \left(\int_{s}^{t} |K_1(t,u)| du \right)^{p-1} \int_{s}^{t} |K_1(t,u)| \|H_u\|^p_{\mathbb{H}} du\]
	Taking expectation, we get:
	\[\mathbb{E}\left[ \left(\int_{s}^{t} |K_1(t,u)| \|H_u\|_{\mathbb{H}} du \right)^{p}\right] \leq  \left(\int_{s}^{t} |K_1(t,u)| du \right)^{p} \sup_{u \in [0,T]}\,\mathbb{E}\left[ \|H_u\|_{\mathbb{H}}^{p} \right] \]
	In a similar manner, ones shows that 
	\begin{align*}
		&\, \mathbb{E}\left[ \left(\int_{a}^{s} |(K_1(t,u) - K_1(s,u))| \|H_u\|_{\mathbb{H}} du \right)^{p} \right] \leq \left(\int_{a}^{s} |(K_1(t,u) - K_1(s,u))| du \right)^{p} \sup_{u \in [0,T]}\,\mathbb{E}\left[ \|H_u\|_{\mathbb{H}}^{p} \right]\\
		&\text{Consequently,} \quad \mathbb{E}\left[\left\|\int_{a}^{t\wedge b} K_1(t,u) H_u du - \int_{a}^{s\wedge b} K_1(s,u) H_u du \right\|_{\mathbb{H}}^p \right] \\
		&\qquad\qquad \le 2^{p-1} \sup_{u \in [0,T]} \mathbb{E}\left[ \|H_u\|_{\mathbb{H}}^{p} \right] \times \Bigg( \left( \int_{s}^{t} |K_1(t,u)| du\right)^p + \left( \int_a^s |(K_1(t,u) - K_1(s,u))| du\right)^p \Bigg)
	\end{align*}
	
	Hence, owing to ~\eqref{eq:contKtilde} and ~\eqref{eq:contK}, there exists a real constant $\kappa= \kappa_{p}$ only depending on $p$ such that
	\[
	\mathbb{E}\left[\left\|\int_{a}^{t\wedge b} K_1(t,u) H_u du - \int_{a}^{s\wedge b} K_1(s,u) H_u du \right\|_{\mathbb{H}}^p \right]  \le \kappa\sup_{u \in [0,T]}\,\mathbb{E}\left[ |b_u|^{p} \right] (t-s)^{p(\theta\wedge \widehat \theta)}
	\]
	If $(\widehat {\cal K}^{cont}_{\widehat \theta})$ does not hold true, but the kernels $K_i$ rather satisfy $({\cal K}^{int}_{\beta})$ for some $\beta>1$  one can replace  $\widehat \theta$ by $\frac{\beta-1}{2\beta}$ {\em mutatis mutandis} in the aboves claims (see \cite[Lemma C.1]{JouPag22}).
	\hfill$\Box$
	
	\subsubsection{Proof of Theorem \ref{Thrm:flotVolterra} on Path-dependent Volterra's flow. }
%	\noindent {\bf Proof of Theorem \ref{Thrm:flotVolterra}.}
%	\smallskip
	\noindent{\sc Step 1.}
	Let \( X \) be a pathwise continuous solution to the Volterra Equation~\eqref{eq:pthvolterra}.
	It follows from Theorem~\ref{Thm:pathExistenceUniquenes} that,  {for every $p> 2$}, the Path-dependent Volterra equation~\eqref{eq:pthvolterra} has a unique strong solution starting from any random function $x_0\!\in L^p(\P)$ which  can be proved (see Theorem~\ref{prop:pathkernelvolt}) to be pathwise continuous (and even $a$-H\"older pathwise continuous for some small enough $a>0$) and satisfying furthermore  $\E [\sup_{t\in [0,T]} \|X_t\|_{\mathbb{H}}^p] <+\infty$.
	
	Set \(x_0(t) := X_0 \varphi(t)\), \(y_0(t) := Y_0 \varphi(t)\) in \(\mathbb{X}\) and define \(f_p(t) = \sup_{0 \le s \le t} \Big\| \|X^{X_0}_s-X^{Y_0}_s\|_{\mathbb{H}} \Big\|_p.\)
	We already know that $ f_p(T) < \infty$, we now aim to derive an estimate that just depend on $\Big\| \sup_{t\in [0,T]} \|x_0(t) - y_0(t) \|_{\mathbb{H}} \Big\|_p=:\Big\| \|x_0 - y_0 \|_{T} \Big\|_p$. 
	For all \( s \in [0,t] \), we have using the generalized Minkowski  for the deterministic term
	\begin{align*}
		&\ \Big\| \left\| X_s^{X_0} - X_s^{Y_0} \right\|_{\mathbb{H}} \Big\|_p
		\leq \Big\| \left\| x_0(s) - y_0(s) \right\|_{\mathbb{H}} \Big\|_p +  \int_0^s K_1(s,r)  \Big\| \left\|b(r,X^r_\cdot) - b(r,Y^r_\cdot)\right\|_{\mathbb{H}} \Big\|_p\,dr  \\
		&\qquad\qquad\qquad\qquad\qquad + \Big\| \big\| \int_0^s K_2(s,r)(\sigma(r,X^r_\cdot) - \sigma(r,Y^r_\cdot))\,dW_r \big\|_{\mathbb{H}} \Big\|_p \leq \Big\| \left\| x_0(s) - y_0(s) \right\|_{\mathbb{H}} \Big\|_p\\ &\qquad+  \int_0^s K_1(s,r)  \Big\| \left\|b(r,X^r_\cdot) - b(r,Y^r_\cdot)\right\|_{\mathbb{H}} \Big\|_p\,dr + C_{p}^{BDG} \left[ \int_0^s K_2(s,r)^2 \Big\| \big\|\sigma(r,X^r_\cdot) - \sigma(r,Y^r_\cdot)\big\|_{\tilde{\mathbb{H}}} \Big\|_p^2 \, dr \right]^{\tfrac12}.
	\end{align*}
	where in the last inequality, we used the \(L^p\)-Burkholder--Davis--Gundy (BDG) inequality.
	
	\noindent But, now, using  the Lipschitz assumption\ref{assump:kernelVolterra} (iii) on \(b\) and \(\sigma\) and then applying Fubini Tonelli theorem,
	 {\small
	\[\Big\| \left\|b(r,X^r_\cdot) - b(r,Y^r_\cdot)\right\|_{\mathbb{H}} \Big\|_p \leq C_{b,\sigma,T} \left( \int_0^r \mathbb{E} \left[\|X_u^r - Y_u^r\|^p_{\mathbb{H}} \right] \mu(du) \right)^\frac1p \leq C_{b,\sigma,T}\,\mu([0,T])^\frac1p\,\sup_{0 \le u \le r} \Big\| \|X^{r}_u-X^{r}_u\|_{\mathbb{H}} \Big\|_p. \]
	}
	\noindent Likewise for the diffusion term, we get setting \(C_1=C_{b,\sigma,T}\,\mu([0,T])^\frac1p\) and \(C_2=C_{p}^{BDG}\,C_{b,\sigma,T}\,\mu([0,T])^\frac1p\)
	\begin{align*}
		\Big\| \left\| X_s^{X_0} - X_s^{Y_0} \right\|_{\mathbb{H}} \Big\|_p
		&\leq \Big\| \left\| x_0(s) - y_0(s) \right\|_{\mathbb{H}} \Big\|_p +  C_1 \int_0^s K_1(s,r)  f_p(r)\,dr + C_2 \left(\int_0^s K_2(s,r)^2f_p^2(r)\, dr\right)^\frac12\\
		&\leq \Big\| \left\| x_0(s) - y_0(s) \right\|_{\mathbb{H}} \Big\|_p + \sum_{i \in \{1,2\}} C_i \; \chi_i(s) \left( 
		\frac{\rho}{A^{\frac1\rho}}\,f_p(s) 
		+ (1-\rho)\,A^{\frac{1}{1-\rho}} \int_0^s f_p(r)\,dr
		\right).
	\end{align*}
	where we Applied relation~\eqref{eq:VoltIntegral} from Lemma~\ref{Lem:BoundVoltIntegral} for all \( A > 0 \) for \(i=1\) and \(i=2\) as \(f_p\) is non-decreasing with \(\rho \in (0,1)\), and \(\chi_1 = \varphi_{\frac{2\beta}{\beta+1}}\) and \(\chi_2 = \psi_{2\beta}\) as defined in equation ~\eqref{eq:notat2}.
	Now, passing to the supremun on \([0,t]\) while taking advantage of the fact that \(f_p\) is non-decreasing, we derive that
	 {\small
	\begin{align*}
		&\,\sup_{0 \le s \le t} \Big\| \|X^{X_0}_s-X^{Y_0}_s\|_{\mathbb{H}} \Big\|_p:= f_p(t) \leq  \sup_{0 \le s \le t} \Big\| \left\| x_0(s) - y_0(s) \right\|_{\mathbb{H}} \Big\|_p + C^\prime\,\left( 
		\frac{\rho}{A^{\frac1\rho}}\,f_p(t) 
		+ (1-\rho)\,A^{\frac{1}{1-\rho}} \int_0^t f_p(r)\,dr
		\right).\\
		&\text{where }\; C^\prime = C_{K_1,K_2,p,\beta,b,\sigma,T} = C_1\,\varphi^{*}_{\frac{2\beta}{\beta+1}} 
		+ C_2\,\psi^{*}_{2\beta}\leq +\infty \;\text{owing to }\; ({\cal K}^{int}_{\beta}).
	\end{align*}
	}
	We now choose \(A\) large enough, namely \(A = (2\rho\, C^\prime)^{\rho}\), so that 
	\(1 - \rho\, C^\prime\,A^{-\frac1\rho} = \frac{1}{2} < 1\).
	Consequently, for every \(t \in [0,T]\),
	
	\centerline{$f_p(t) \le 2\,\Big\| \sup_{0 \le s \le t} \left\| x_0(s) - y_0(s) \right\|_{\mathbb{H}} \Big\|_p
		+ 2C^\prime\,(1-\rho)\,A^{\frac{1}{1-\rho}} \int_0^t f_p (r)\,dr.$}
	\noindent One concludes by Gr\"onwall's lemma (See for example \cite[Lemma 7.2]{pages2018numerical}) that:
	\[
	f_p(t) \le 2\,\Big\| \left\| x_0 - y_0 \right\|_{t} \Big\|_p e^{K\,t},\;
	\text{with} \;
	K = K_{K_1,K_2,p,\beta,b,\sigma,T}
	= 2C^\prime(1 - \rho)\,A^{\frac1{1-\rho}}.
	\]
	As a conclusion, \(	\sup_{0 \le s \le T} \Big\| \|X^{X_0}_s-X^{Y_0}_s\|_{\mathbb{H}} \Big\|_p
	= f_p(T) \le C\Big\| \left\| x_0 - y_0 \right\|_{T} \Big\|_p,\;\text{where} \;
	C = 2\, e^{K\,T} \ge 2.\)
	
	\smallskip
	\noindent{\sc Step 2.} (Continuity of the flow.) In particular, by assuming \(X_0 = x \in\mathbb{H}\) and \(Y_0 =y \in\mathbb{H}\), the above implies that there exists a positive real constant $C_{p,T}$ such that:
	\[\forall \,x,y\in\mathbb{H},\;\sup_{t\in [0,T]}\;\Big\| \|X^{x}_t-X^{y}_t\|_{\mathbb{H}} \Big\|_p \le C_{p,T} \|x\varphi - y \varphi\|_{T} ,
	\; \text{where} \; C_{p,T}= 2\,e^{KT}.\]
	Moreover by equation~\eqref{eq:Lpincrements2} (or rather equation~\eqref{eq:Lpincrements2} in Proposition\ref{prop:pathkernelvolt2}  with Lemma\ref{lm:Uniformgrowth}), we have:
	$ \forall p>p_{eu}$ , $\forall\, t,\, t^\prime\!\in [0,T],\quad \mathbb{E}\left[\left\|X_{t^\prime} - X_t\right\|_\mathbb{H}^p\right] \le C_{p,T} {\left( 1 + \| x \varphi \|_{T}^{p} \right)}|t^\prime-t|^{p(\delta \wedge\theta\wedge \frac{\beta-1}{2\beta})}.$
	i.e.
	\[
	\forall\, s,\, t\!\in [0,T], \, \forall\, x\!\in \mathbb{H}, \quad  \| \|X^x_t-X^x_s\|_{\mathbb{H}}\big\|_p \le C'_{p,T}(1+\| x \varphi \|_{T})|t-s|^{ \tilde \theta}.
	\]
	where we set $\tilde \theta = \delta \wedge\theta\wedge \frac{\beta-1}{2\beta}$.
	Now, let \( \lambda \in (0,1) \) be fixed. One has
	{\small
		\begin{align*}
			&\big\| \| X^x_t - X^y_t - (X^x_s - X^y_s) \|_{\mathbb{H}} \big\|_p \le \big( \| \| X^x_t - X^y_t \|_{\mathbb{H}} \|_p + \| \| X^x_s - X^y_s \|_{\mathbb{H}} \|_p \big)^\lambda %\\ &\quad \times
			\big( \| \| X^x_t - X^x_s \|_{\mathbb{H}} \|_p + \| \| X^y_t - X^y_s \|_{\mathbb{H}} \|_p \big)^{1-\lambda} \\
			&\hspace{4cm}\le 2 C_{p,T}^\lambda (C'_{p,T})^{1-\lambda} \|x\varphi-y\varphi\|_{T}^\lambda (1 + \| x \varphi \|_{T} + \| y \varphi \|_{T})^{1-\lambda} |t-s|^{\tilde \theta(1-\lambda)}
		\end{align*}
	}
	\noindent or, equivalently, $\forall\, s,t \in [0,T], \, \forall\, x \in \mathbb{H},$
	\[
	\mathbb{E} [ \| X^x_t - X^y_t - (X^x_s - X^y_s) \|_{\mathbb{H}}^p] \le \tilde{C}_{p,T} (1 + \|x\varphi\|_{T} +  \|y\varphi\|_{T})^{p(1-\lambda)} \|x\varphi-y\varphi\|_{T}^{p\lambda} |t-s|^{p \tilde{\theta}(1-\lambda)}.
	\]
	Let us assume from now on that \( p > p^{**} = p_{eu} \vee \frac{1}{ \tilde \theta(1-\lambda)} \). It follows from the proof of Kolmogorov's criterion~\cite[Theorem 2.1]{RevuzYor} in time that, for any $a\!\in\big(0,(1-\lambda) \tilde \theta-\frac 1p\big)$, there exists a positive real constant $C_{a,p,T}$ such that $\forall\, x,y\!\in \mathbb{H}$:
	\[
	\E \big[ \sup_{s,t\in [0,T]} \left(\frac{ \| X^x_t - X^y_t - (X^x_s - X^y_s) \|_{\mathbb{H}}}{|t-s|^{a}}\right)^p \big] \le   C_{a,p,T}\,(1 + \|x\varphi\|_{T} + \|y\varphi\|_{T})^{p(1-\lambda)} \|x\varphi-y\varphi\|_{T}^{p\lambda}.
	\]
	where tracking the constants ensures that the final constant has linear growth in the constant which appears in the hypothesis.
	Using that 
	\[\| X^x_t - X^y_t \|_{\mathbb{H}}^p=\| X^x_t - X^y_t - (X^x_s - X^y_s) + (X^x_s - X^y_s) \|_{\mathbb{H}}^p\leq 2^{p-1}(\| X^x_t - X^y_t - (X^x_s - X^y_s) \|_{\mathbb{H}}^p + \| X^x_s - X^y_s \|_{\mathbb{H}}^p)\]
	% $\| X^x_t - X^y_t \|_{\mathbb{H}}^p=\| X^x_t - X^y_t - (X^x_s - X^y_s) + (X^x_s - X^y_s) \|_{\mathbb{H}}^p\leq 2^{p-1}(\| X^x_t - X^y_t - (X^x_s - X^y_s) \|_{\mathbb{H}}^p + \| X^x_s - X^y_s \|_{\mathbb{H}}^p)$
	 then taking  $s=0$ and owing to the rough inequality $\|x\varphi-y\varphi\|_{T}^{p(1-\lambda)} \le (1 + \|x\varphi\|_{T} + \|y\varphi\|_{T})^{p(1-\lambda)}$ $\forall\, x,y\!\in \mathbb{H}$ directly yields:
	\begin{align*}
		\E\sup_{t\in [0,T]} \| X^x_t - X^y_t \|_{\mathbb{H}}^p& \le  2^{p-1}\left(\| x\varphi - y \varphi\|_{T}^p+ T^{pa} C_{a,p,T}\, (1 + \|x\varphi\|_{T} + \|y\varphi\|_{T})^{p(1-\lambda)} \|x\varphi-y\varphi\|_{T}^{p\lambda}\right)\\
		& \le C'_{a,p,T, \phi} (1 + \|x\varphi\|_{T} + \|y\varphi\|_{T})^{p(1-\lambda)} \|x\varphi-y\varphi\|_{T}^{p\lambda}.
	\end{align*}
	where the constant $C'_{a,p,T, \phi} $ is given by: $C'_{a,p,T, \phi}= 2^{p-1}\left(1+ T^{pa} C_{a,p,T}\right)$.
	
	\smallskip Let \( r \in \mathbb{N} \) be fixed. For each \( x \in B(0,r) \), we define a pathwise continuous modification \( (X^x_t)_{t \in [0,T]} \), such that the map \( B(0,r) \ni x \mapsto (X^x_t)_{t \in [0,T]} \) takes values in the Polish space \( \mathcal{C}(\mathbb{T}, \mathbb{H}) \), equipped with the distance \( \rho_T \) induced by the \( \sup \)-norm \(\|\cdot\|_T\) on \( \mathbb{T} \). For any \( x, y \in B(0,r) \), we obtain the following estimate:
	\begin{align}\label{eq:flowsup}
		\forall\, x,\, y \in B(0,r), \quad \E\sup_{t\in [0,T]} \| X^x_t - X^y_t \|_{\mathbb{H}}^p 
		& \le C'_{a,p,T,\phi} (1 + 2r)^{p(1-\lambda)} \|x\varphi-y\varphi\|_{T}^{p\lambda}.
	\end{align}
	Assume now that $p>p^{*}= p^{**}\vee \frac{\text{dim}(\mathbb{H})}{\lambda}$, 
	% Given that \( p\lambda > \dim \mathbb{H} \), 
	a  subsequent application of Kolmogorov's criterion in space \cite[Theorem 2.1]{RevuzYor} guarantees the existence of a \( \mathbb{P} \)-modification \( (\widetilde{X}^{(r),x})_{x \in B(0,r)} \) of the \( \mathcal{C}(\mathbb{T}, \mathbb{H}) \)-valued process \( (X^x)_{x \in B(0,r)} \), where, \( \mathbb{P} \)-almost surely, the map \( x \mapsto \widetilde{X}^{(r),x} \) is continuous with respect to \( \rho_T \), and more specifically, it is \( a' \)-H\"older-continuous for \( a' \in \left( 0, \lambda - \frac{\text{dim}(\mathbb{H})}{p} \right) \). 
	For $n \in \mathbb{N}$, it is straightforward that the restriction of \( (\widetilde{X}^{(r+n),x})_{x \in B(0,r+n)} \) to \( B(0,r) \) is also a \( \mathbb{P} \)-modification with continuous "paths" in \( x \). Consequently, they are \( \mathbb{P} \)-indistinguishable. Therefore, there exists a \( \mathbb{P} \)-modification \( (\widetilde{X}^x)_{x \in \mathbb{H}} \) of \( (X^x)_{x \in \mathbb{H}} \) such that \( \mathbb{P}(d\omega) \)-almost surely, the map \( x \mapsto (\widetilde{X}^x_t)_{t \in [0,T]} \) is continuous from \( \mathbb{H} \) to \( \mathcal{C}(\mathbb{T}, \mathbb{H}) \), and \( (\widetilde{X}^x_t)_{t \in [0,T]} \) is a solution to~\eqref{eq:pthvolterra} for \( X_0 = x \).
	
	\medskip
	\noindent {\bf Remark:}
	If $N_0 \subset \Omega$ is the $\P$-negligible set where no continuous extension is possible,  one may assume that  continuity of $x\mapsto \widetilde X^x(\omega)$ holds for every $\omega \!\in \Omega\setminus N_0$.
	\subsubsection{Proof of Blagove$\check{\rm \bf s}\check{\rm \bf c}$enkii-Freidlin like theorem~\ref{thm:FreidlinLike} }
	%\medskip
	 %   \noindent {\bf Proof of Blagove$\check{\rm \bf s}\check{\rm \bf c}$enkii-Freidlin like theorem.} 
	In the spirit of \cite[Theorem~D.2]{JouPag22}, we adapt to path dependent Volterra SDEs the classical proof of Blagove$\check{\rm \bf s}\check{\rm \bf c}$enkii-Freidlin's theorem 
	(see, e.g., \cite[Theorem~10.4]{RogersWilliamsII}), originally written for standard Brownian diffusion processes and extended in \cite[Theorem~D.2]{JouPag22} to Volterra SDEs. 
	
	\smallskip
	\noindent  {\sc Step~1} ({\em The functional $F$)}. We only take advantage of the strong existence and uniqueness of pathwise continuous strong solutions given in Theorem \ref{prop:pathkernelvolt2}
	for every $\nu$-distributed starting random vector $X_0$ having a finite $L^p$-moment to exhibit a functional 
	$\tilde{F}_\nu :  \mathbb{H}\times {\cal C}_0(\mathbb{T}, \bar{\mathbb{H}})\to \mathbb{X}$ 
	(adapted with respect to the canonical filtrations of both spaces) such that, for
	any strong solution to~\eqref{eq:pthvolterra} driven by a $\bar{\mathbb{H}}$-valued $(\mathcal{F}_t)$-Brownian motion 
	$B = (B_t)_{t \in [0,T]}$ on a filtered probability space 
	$(\Omega, \mathcal{A}, (\mathcal{F}_t)_{t\in [0,T]}, \mathbb{P})$ and any $\mathcal{F}_0$-measurable starting random vector 
	$X_0$ with distribution $\nu$,
	\[
	\mathbb{P}\text{-a.s.} \quad X = \tilde{F}_\nu(X_0, B).
	\]
	Consider the canonical space $ \widetilde \Omega = \mathbb{H}\times{\cal C}_0(\mathbb{T}, \bar{\mathbb{H}})$ equipped with the product $\sigma$-field 
	{${\cal B}or(\mathbb{H})\otimes {\cal B}or({\cal C}_0(\mathbb{T}, \bar{\mathbb{H}}))$} and the product probability measure $\widetilde\P= \nu\otimes \Q_{W}$ where $\nu$ has a finite $p$-th moment for some $p>p_{eu}$ and \(\mathbb{Q}_{W}\) the \(\bar{\mathbb{H}}\)-valued wiener measure. Let $(X_0,B)$ denote the canonical process on $ \widetilde \Omega$ defined by $(X_0, B)(x,w)= (x,w)$. As $p>p_{eu}$, the path-dependent Volterra 
	SDE~\eqref{eq:pthvolterra} has a unique pathwise strong solution  $(\xi_t)_{t\in [0,T]}$ defined on  $ \widetilde \Omega$ (and adapted to the $\widetilde\P$-completed natural filtration  $\widetilde\cF_t = \sigma({\cal N}_{\widetilde\P}, X_0, B_s,\,0\le s\le t), t\!\in \mathbb{T}$, etc). We set 
	$$
	\widetilde F_\nu(x,w) = \xi(x,w).
	$$
	The functional $\widetilde F_{\nu}: \mathbb{H}\times   {\cal C}_0(\mathbb{T}, \bar{\mathbb{H}})\to  {\cal C}(\mathbb{T}, \mathbb{H})$ is adapted w.r.t. to the canonical filtrations of both spaces.
	Moreover, for a \(\bar{\mathbb{H}}\)-valued $(\mathcal{F}_t)$-Brownian motion $W$ on a filtered probability space $(\Omega, \mathcal{A}, (\mathcal{F}_t)_{t\in[0,T]}, \mathbb{P})$ and an $\mathcal{F}_0$-measurable $\mathbb{H}$-valued random vector $X_0$, the process $\widetilde{F}_\nu(X_0, W)$ is a strong solution to the Path-dependent Volterra equation~\eqref{eq:pthvolterra} driven by $W$. By strong uniqueness, it is $\mathbb{P}$-indistinguishable from the "natural" strong solution $X$ with starting random function $X_0$. Consequently, the functional $\widetilde{F}_\nu$, defined on the canonical space $\widetilde{\Omega}$, is independent of the chosen stochastic basis (nor the Brownian motion under consideration).
	In particular, for every $x \in \mathbb{H}$, define
	\[
	F(x, w) := \widetilde{F}_{\delta_x}(x, w), \quad (x, w) \in \mathbb{H} \times \mathcal{C}_0(\mathbb{T}, \bar{\mathbb{H}}).
	\]
	Then, owing to the fact that the starting value $x$ is deterministic and by uniqueness of strong solutions to the Path-dependent Volterra equation~\eqref{eq:pthvolterra}, we have $\tilde{X}^x = F(x, W_\cdot)$ $\mathbb{P}$-a.s. \footnote{$\tilde{X}^x$ is the \( \mathbb{P} \)-modification defined in the end of the proof of Theorem~\ref{Thrm:flotVolterra}~$(b)$.} Moreover,
	\[
	\int_{\mathbb{H}} \rho_T \big(F(x, w), F(y, w)\big)^p \, \mathbb{Q}_W(dw)
	= \mathbb{E}\big[ \rho_T(\tilde{X}^x, \tilde{X}^y)^p \big],
	\]
	so that, we show, as for $\tilde{X}^x$, that $x \mapsto F(x, w)$ admits an $a'$-Hölder $\mathbb{Q}_W$-modification 
	since $p > p^*$ (in Theorem \ref{Thrm:flotVolterra} (b)). Hence $x \mapsto F(x, W)$ and $(\tilde{X}^x)_{x \in \mathbb{H}}$ are 
	$\mathbb{P}$-indistinguishable.
	Note that since $F$ is continuous in $x$ and measurable in $w$, this functional is bi-measurable according to Lemma 4.51~\cite{AliBord}. The proof of the Kolmogorov's lemma (\cite[Lemma 44.4, Section IV.44, p.100]{RogersWilliamsII}) (Kolmogorov continuity criterion or see also ~\cite[Theorem 2.1, p. 26, 3rd edition]{RevuzYor}) makes it clear that the function F can be defined in a way independent of the particular set-up of the probability space.
	
	\smallskip
	\noindent  {\sc Step~2} ({\em Representation for $X_0\!\in L^p(\P)$, $p$ large i.e. $p>p^*$ for $p^*$ defined in Theorem \ref{Thrm:flotVolterra}  }).
	
	Consider solving equation~\eqref{eq:pthvolterra} with initial condition \(X_0\) for some set-up of the probability space.  
	According to the Proposition~\ref{prop:pathkernelvolt2}, the path-dependent SVIE starting from a random value $X_0\in L^p(\P)$, $p>p^*$, independent of $W$ has a
	pathwise continuous solution. We want to show that $X=F(X_0,W)$.
	By Equation ~\eqref{eq:boundflow1},
	if $X'$ denotes the pathwise continuous solution starting from $X'_0\in L^p(\P)$ independent of $W$,
	\begin{equation}\label{eq:flotgeneral}
		\sup_{t\in [0,T]} \;\Big\| \|X_t-X'_t\|_{\mathbb{H}} \Big\|_p \le  C_{p,T}\Big\| \|x_0-x'_0 \|_{T} \Big\|_p ,
		\; \text{where} \; C_{p,T}= 2\,e^{KT}.
	\end{equation}
	Let \(\mathbb{D} = \{j2^{-k}: j,k \in \Z\}\) the set of dyadic rational in \(\R\). We have $  \mathbb{D}^{\dim(\mathbb{H})}\subset\mathbb{H} $.
	Let's consider the measurable map
	$\varphi_k:\mathbb{H}\to \mathbb{D}^{\dim(\mathbb{H})} \subset\mathbb{H} $
	with $\|\varphi_k(a)-a\|_{\mathbb{H}} \leq 2^{-k}$
	for every $a\!\in \mathbb{H}$. 
	We have that \(\sup_{ x \in \mathbb{H}}\|\varphi_k(x)-x\|_{\mathbb{H}} \to 0\) as \(k\to +\infty\). Now, set  $X^{(k)}_0=\varphi_k(X_0)$
	and $N_x = \{\omega\!\in \Omega: \widetilde X^x(\omega)\neq F (x,W(\omega)) \}\cup N_0$\footnote{$N_0 \subset \Omega$ is the $\P$-negligible set where no continuous extension is possible}.  By construction the set $N=  \cup_{x\in \mathbb{D}^{\dim(\mathbb{H})}}N_x$  is $\P$-negligible.
	As 
	$X^{(k)}_0$ is a $\mathbb{D}^{\dim(\mathbb{H})}$-valued, $\cF_0$-measurable random vector, independent of $W$, clearly $X^{(k)} =F(X^{(k)}_0,W)$ is a pathwise continuous solution~\eqref{eq:pthvolterra} so that, by~\eqref{eq:flotgeneral}, setting $x_0:=X_0\varphi$ and $x^{(k)}_0:=X^{(k)}_0\varphi$
	$$ 
	\sup_{t\in [0,T]} \Big\| \|X_t-X^{(k)}_t \|_{\mathbb{H}} \Big\|_p \le C_{p,T} \Big\| \|x_0-x^{(k)}_0\|_{T} \Big\|_p \to 0\mbox{ as }k\to +\infty.
	$$
	With the triangle inequality, one deduces that
	\[
	\sup_{t\in [0,T]}\Big\| \|X_t-F(X_0,W)(t) \|_{\mathbb{H}} \Big\|_p \leq \sup_{t\in [0,T]} \Big\| \|X_t-X^{(k)}_t \|_{\mathbb{H}} \Big\|_p  + \Big\| \sup_{t\in [0,T]} \|X^{(k)}_t-F(X_0,W)(t) \|_{\mathbb{H}} \Big\|_p 
	\]
	so that    
	\begin{align*}
		\limsup_{k\to\infty}\sup_{t\in [0,T]} \Big\| \|X_t-F(X_0,W_\cdot)(t) \|_{\mathbb{H}} \Big\|_p &\le \limsup_{k\to\infty}\Big\| \rho_T\Big(F(X_0^{(k)},W), F(X_0,W)\Big) \Big\|_p .
	\end{align*}
	Using that $X_0$ and $W$ are independent and the equation ~\eqref{eq:flowsup}, we get: 
	\begin{align*}
		\E \,\Big(\rho_T\big(F(X_0^{(k)},W) \,&,\,F(X_0,W)\big)\Big) ^p 
		= \int_{\mathbb{H}} \P_{X_0}(d\xi)\, \E \,\Big(\rho_T\big(F(\varphi_k(\xi),W),F(\xi,W) \big)\Big)^p\\
		& = \int_{\mathbb{H}}\P_{X_0}(d\xi)\, \E \,\big[\rho_T\big(X^{\varphi_k(\xi)},X^{\xi} \big)^p\big] 
		= \sum_{r\in\N^*}\int_{A(0, r-1, r)}\P_{X_0}(d\xi)\, \E \,\big[\rho_T\big(X^{\varphi_k(\xi)},X^{\xi} \big)^p\big]\\
		& \leq  C'_{a,p,T}\sum_{r\in\N^*}(1 + 2r)^{p(1-\lambda)}
		\int_{A(0, r-1, r)}\P_{X_0}(d\xi)\, \|\varphi_k(\xi)\varphi-\xi\varphi \|_{T}^{p\lambda}\\
		&\le  C'_{a,p,T}\sum_{r\in\N^*}(1 + 2r)^{p(1-\lambda)}
		\E \Big[\|\varphi_k(X_0)\varphi-X_0\varphi \|_{T}^{p\lambda} 
		\mathbf{1}_{\{r-1\le \|X_0\|_{\mathbb{H}}\le r\}}\Big]
	\end{align*}
	We then have, and owing to Markov's inequality  in the last line
	\begin{align*}
		\E \,\Big(\rho_T\big(F(X_0^{(k)},W) \,,\,F(X_0,W)\big)\Big) ^p &\le  C'_{a,p,T}\Big( 3^{p(1-\lambda)}\E \Big[\|\varphi_k(X_0)\varphi-X_0\varphi \|_{T}^{p\lambda} \mbox{\bf 1}_{\{ \|X_0\|_{\mathbb{H}}\le 1\}}\Big]\\
		&\qquad+ \sum_{r\ge 2}(1+2r)^{p(1-\lambda)}\E \Big[\|\varphi_k(X_0)\varphi- X_0\varphi \|_{T}^{p\lambda}\mbox{\bf 1}_{\{r-1\le \|X_0\|_{\mathbb{H}}\le r\}}\Big]\Big)\\  
		&\le  C_{a,p,T, X_0}\Big(1+\sum_{r\ge 2}\frac{(1+2r)^{p(1-\lambda)}}{(r-1)^p}\Big)\E \Big[\|\varphi_k(X_0) \varphi-X_0\varphi \|_{T}^{p\lambda} \|X_0\|_{\mathbb{H}}^p \Big]
	\end{align*}
	We have \(\frac{(1+2r)^{p(1-\lambda)}}{(r-1)^p} \underset{+\infty}{\sim} \frac{1}{r^{p\lambda}} \), so that the series is converging (and then bounded by a certain positive constant \(K^{\prime}\)) since it is equivalent to the Riemann serie of general term $\frac{1}{r^{p\lambda}}$ which is converging as $p\lambda>  \text{dim}(\mathbb{H}) \ge 1$. 
	Consequently, there exists \(K^{\prime\prime}=K^{\prime\prime}_{a,p,T, X_0,\lambda,\varphi}:=C_{a,p,T, X_0} K^{\prime} \|\varphi\|_{T}^{p\lambda} \) such that
	\[
	\E \,\Big(\rho_T\big(F(X_0^{(k)},W) \,,\,F(X_0,W)\big)\Big) ^p \le K^{\prime\prime}  \E \Big[\|\varphi_k(X_0)-X_0 \|_{T}^{p\lambda} \|X_0\|_{\mathbb{H}}^p \Big]\le K^{\prime\prime} 2^{-p\lambda k} \E\|X_0\|_{\mathbb{H}}^p\underset{k\to +\infty}{\to} 0 
	\]
	Then, $\limsup_{k\to\infty}\sup_{t\in [0,T]}\Big\|  \|X_t-F(X_0,W)(t) \|_{\mathbb{H}} \Big\|_p=0$.
	We conclude that, for every $t\!\in [0,T]$, $X_t = F(X_0,W)(t)$ $\P$-$a.s.$. Owing to the pathwise continuity of both processes, it follows $X= F(X_0,W)$ $\P$-a.s.
	
	\smallskip
	\noindent  {\sc Step~4} ({\em Representation for $X_0\!\in L^0(\P)$}). Let $X_0\!\in L^0(\P)$.   We consider the sequence of  truncated starting values $X^{(k)}_0= X_0\mbox{\bf 1}_{A_k}$ with $A_k=\{\|X_0\|_{\mathbb{H}}< k\}$, $k\ge 1$ and the resulting pathwise continuous solution to~\eqref{eq:pthvolterra} still denoted $X^{(k)}$. As $X^{(k)}_0\!\in L^p(\P)$ for $p>p^*$, then $\P$-$a.s.$, \(X^{(k)}= F(X_0^{(k)},W),\quad k\ge 1.\)
	Hence, under the convention $A_0=\emptyset$,
	\begin{align*}
		\hskip3.25cm X=\sum_{k\ge 1}X^{(k)}\mbox{\bf 1}_{A_k\setminus A_{k-1}}=\sum_{k\ge 1}F(X_0^{(k)},W)\mbox{\bf 1}_{A_k\setminus A_{k-1}}=F(X_0,W).
	\end{align*}
	Therefore, the result is established, and we are done.~\hfill $\Box$
	\vspace{-.28cm}
	\subsection{Proofs of Lemma~\ref{lem:gap} and~\ref{Lem:BoundVoltIntegral},  Corollary \ref{Gronwall} and Proposition~\ref{lem:K-intF} }\label{app:B2}
	\subsubsection{Proofs of Lemma~\ref{Lem:BoundVoltIntegral} and Corollary \ref{Gronwall}}
	
	\noindent{\bf Proof of Lemma~\ref{Lem:BoundVoltIntegral}:} Set $\rho =  \frac 12 \left(1+ \frac {1}{\beta}\right) = \frac{\beta+1}{2\beta} \in (0,1)$ since $\beta >1$. As \( f \) is non-decreasing, one derives from the first term
	\begin{align*}
		\nonumber \int_0^t K_1(t,s) f(s) \, ds & \leq f^{\rho}(t) \int_0^t K_1(t,s) f(s)^{1-\rho} \, ds \\
		\nonumber   & \leq f^{\rho}(t) \varphi_{\frac{2\beta}{\beta+1}}(t) \left( \int_0^t f(s)^{\frac{2\beta}{\beta-1}(1-\rho)} \, ds \right)^{\frac{\beta-1}{2\beta}} = \varphi_{\frac{2\beta}{\beta+1}}(t) \frac{f^{\rho}(t)}{a} \cdot a \left( \int_0^t f(s) \, ds \right)^{1-\rho}.
	\end{align*}
	where we used H\"older's inequality in the second line with conjugate exponents \(\left(\frac{2\beta}{\beta + 1}, \frac{2\beta}{\beta - 1}\right) = \left(\frac{1}{\rho}, \frac{1}{1-\rho}\right)\).  Applying now Young's inequality\footnote{For all real numbers \(x \geq 0\) and \(y \geq 0\), and for all real numbers \(p > 0\) and \(q > 0\) such that \( \frac{1}{p} + \frac{1}{q} = 1 \) (often referred to as conjugate exponents), the following inequalities hold: 
		\( xy \leqslant \frac{x^p}{p} + \frac{y^q}{q} \) and \( xy \leqslant \frac{1}{(q\varepsilon)^{p/q}} \frac{x^p}{p} + \varepsilon y^q \quad \forall \varepsilon \geq 0 \).}
	with the same conjugate exponents yields
	\begin{equation*}
		\int_0^t K_1(t,s) f(s) \, ds \leq \varphi_{\frac{2\beta}{\beta+1}}(t) \left( \frac{\rho}{a^{1/\rho}} f(t) + (1-\rho) a^{1/(1-\rho)} \int_0^t f(s) \, ds \right).
	\end{equation*}
	We proceed likewise with the second term, \( f \) being non-decreasing. It follows, still with \(\rho = \frac{1}{2} \left(1 + \frac{1}{\beta}\right)\),
	\begin{align*}
		\left( \int_0^t K_2(t,s)^2 f^2(s) \, ds \right)^{\frac{1}{2}} \leq f^{\rho}(t) \left( \int_0^t K_2(t,s)^2 f(s)^{2(1-\rho)} \, ds \right)^{\frac{1}{2}} \leq \psi_{2\beta}(t) f^{\rho}(t) \left( \int_0^t f(s)^{\frac{2(1-\rho)\beta}{\beta-1}} \, ds \right)^{\frac{\beta-1}{2\beta}}.
	\end{align*}
	where we applied H\"older's inequality with conjugate exponents \((\beta, \frac{\beta}{\beta-1})\). Noting that \(\frac{2(1-\rho)\beta}{\beta-1} = 1\), applying Young's inequality with conjugate exponents \(\left(\frac{1}{\rho}, \frac{1}{1-\rho}\right)\) and introducing \(a > 0\) like for the drift term, we obtain
	\begin{equation*}
		\left( \int_0^t K_2(t,s)^2 f(t)^2(s) \, ds \right)^{\frac{1}{2}} \leq \psi_{2\beta}(t) \left( \rho \frac{f(t)}{a^{1/\rho}} + (1-\rho) a^{1/(1-\rho)} \left( \int_0^t f(s) \, ds \right) \right).
	\end{equation*}
	\noindent {\bf Proof of Corollary \ref{Gronwall}.}
	Applying relation~\eqref{eq:VoltIntegral} from Lemma~\ref{Lem:BoundVoltIntegral} for \(i=1\) and \(i=2\), and combining the resulting inequalities before Plugging them into relation~\eqref{eq:VoltGronwall} yields
	\begin{align}
		& \ f(t)  \leq \psi(t) + A \int_0^t K_1(t,s) f(s) \, ds + B \left( \int_0^t K_2(t,s)^2 f^2(s) \, ds \right)^{\frac{1}{2}} \leq \psi(t) + \nonumber \\
		& \qquad A \varphi_{\frac{2\beta}{\beta+1}}(t) \left( \frac{\rho}{a^{1/\rho}} f(t) + (1-\rho) a^{1/(1-\rho)} \int_0^t f(s) \, ds \right)  + B \psi_{2\beta}(t) \left( \rho \frac{f(t)}{a^{1/\rho}} + (1-\rho) a^{1/(1-\rho)}  \int_0^t f(s) \, ds  \right)  \nonumber \\
		& \qquad\leq \psi(t) +  \left(A \varphi_{\frac{2\beta}{\beta+1}}(t) + B \psi_{2\beta}(t) \right) \left( \frac{\rho}{a^{1/\rho}} f(t) + (1-\rho) a^{1/(1-\rho)} \int_0^t f(s) \, ds \right) \nonumber\\
		& \leq \psi(t) +  \left( A \sup_{t \in [0,T]} \varphi_{\frac{2\beta}{\beta+1}}(t) + B \sup_{t \in [0,T]} \psi_{2\beta}(t) \right) \left( \frac{\rho}{a^{\frac1\rho}} f(t) + (1-\rho) a^{\frac{1}{1-\rho}} \int_0^t f(s) \, ds \right)  \leq \psi(t) +\nonumber \\
		&  \left( A \varphi_{\frac{2\beta}{\beta+1}}^* + B \psi_{2\beta}^* \right) \left( \frac{\rho}{a^{\frac1\rho}} f(t) + (1-\rho) a^{\frac{1}{1-\rho}} \int_0^t f(s) \, ds \right) \leq \psi(t) +  K \left( \frac{\rho}{a^{\frac1\rho}} f(t) + (1-\rho) a^{\frac{1}{1-\rho}} \int_0^t f(s) \, ds \right) \nonumber
	\end{align}
	We choose now $a$ large enough, namely $ a= (2 \rho\, K  )^{\rho}$, so that $1- \rho \,K a^{-1/\rho}= \frac 12<1$. Consequently, for every $t\!\in [0,T]$, \(	f(t)\le 2\psi(t) +  2 (1-\rho) a^{1/(1-\rho)} \int_0^t f(s) \, ds.\) 
	One concludes by Gr\"onwall's lemma (See for example \cite[Lemma 7.2]{pages2018numerical}) that:  \(f(t)\le 2\psi(t) e^{K^{\prime}  t}\)
	with $K^{\prime}= K^{\prime}_{A,B,K_1,K_2,T,\beta,\rho} =  2 K   (1-\rho) a^{1/(1-\rho)} = 2^{\frac{1}{1-\rho}}(1-\rho)\rho^{\frac{\rho}{1-\rho}}\left( A \varphi_{\frac{2\beta}{\beta+1}}^* + B \psi_{2\beta}^* \right)^{\frac{1}{1-\rho}}$.
	
	 \subsubsection{Proofs of Proposition~\ref{lem:K-intF} and Lemma~\ref{lem:gap} }
	 \noindent {\bf Proof of Proposition~\ref{lem:K-intF}.} As for the interpolated $K$-integrated scheme~\eqref{eq:Eulergen2bis}, it amounts to prove that, for every $k\in\{0,\cdots,n-1\}$, the process
	 \[
	 [t^n_k,T]     \ni t\mapsto \int_{t^n_k}^{t\wedge t^n_{k+1}} K_2(t,u)dW_u
	 \]
	 can be written as a functional of $W$. 
	 We denote by
	 \(\mathbb{Q}_{W}\) the \(\bar{\mathbb{H}}\)-valued wiener measure on \({\cal C}([0,T], \bar{\mathbb{H}})\).
	 First we can assume temporarily  that $W$ is defined on the canonical space $\Omega_0= {\cal C}_0([0,T], \bar{\mathbb{H}} )$~--~ i.e. $W_t(w)= w(t)$, for every $w\!\in \Omega_0$~--~equipped with its Borel $\sigma$-field {${\cal B}or({\cal C}_{\bar{\mathbb{H}}, 0})$} induced by the $\sup$-norm topology and the Wiener measure $\Q_{_W}$. For every $t\!\in [t^n_k,T]$,
	 $\displaystyle  \int_{t^n_k}^{t\wedge t^n_{k+1}} K_2(t,u)dW_u$ is $\sigma\big({\cal N}_{\P},W_{u}-W_{t^n_k},\, t^n_k\le u\le t\wedge t^n_{k+1}\big)$-measurable hence $\sigma\big({\cal N}_{\P},W_u,\, 0\le u\le T\big)$-measurable  so that 
	 \begin{equation}\label{eq:Wienerrepresentation}  \int_{t^n_k}^{t\wedge t^n_{k+1}} K_2(t,u)dW_u = F^{k}(t,W).
	 \end{equation}
	 (A more precise result involving adaptation to the canonical filtration can be obtained likewise but it is not useful for our purpose here.) 
	 The pathwise continuity of $\left(\int_{t^n_k}^{t\wedge t^n_{k+1}} K_2(t,u)dW_u\right)_{t\in[t^n_k,T]}$ follows from Lemma \ref{lem:bounds} applied with $\widetilde H\equiv 1$ and implies that,  for  $\Q_{_W}(dw)$-almost every $w\!\in  \Omega_0$,  $[t^n_k,T]\ni t\mapsto F^{k}(t,w)\!\in {\cal C}_0([0,T],\mathbb{H})$.
	 Then for any standard Brownian motion $W$ on a stochastic basis $(  \Omega, {\cal A}, (\cF_t)_{t\in [0,T]}, \P)$ the representation~\eqref{eq:Wienerrepresentation} holds true (with the same functional) since the Gaussian law of $\left(W,\int_{t^n_k}^{t\wedge t^n_{k+1}} K_2(t,u)dW_u\right)$ does not depend on the choice of $W$.    
	 
	 This being done, one checks by induction on the discretization times  using~\eqref{eq:Eulergen2bis} that  the $K$-integrated Euler scheme starting from any  (finite) $\mathbb{H}$-valued starting random vector $X_0$ admits a representation 
	 \[
	 \bar X= \bar F_n (X_0,W),
	 \]
	 where {$\bar F_n : \big(\mathbb{H}\times \Omega_0, {\cal B}or(\mathbb{H})\otimes {\cal B}or({\cal C}_{\bar{\mathbb{H}}, 0}))\big) \to \big({{\cal C}([0,T],\mathbb{H}), {\cal B}or({\cal C}_{\mathbb{H}}))}\big)$ is  measurable.}~\hfill $\Box$ 
	 
	 \bigskip
	 \noindent {\bf Proof of Lemma~\ref{lem:gap}:}
	 According to Theorem~\ref{thm:FreidlinLike}
	 below (whose proof is an adaptation  to path dependent Volterra equations of the proof of Blagove$\check{\rm   s}\check{\rm  c}$enkii-Freidlin's theorem~\cite[Theorem 13.1, Section V.12-13, p.136]{RogersWilliamsII}),
	 there exists  a bi-measurable functional $F: \mathbb{H}\times {\cal C}_0(\mathbb{T}, \bar{\mathbb{H}})\to \mathbb{X}$ such that  the processes $X^{\xi}$ reads
	 \[
	 X^{\xi} = F\big(\xi,(W_t
	 )_{t\in [0,T]}\big) \quad \text{i.e.} \quad X^{\xi}_t = F\big(\xi,W\big)(t), \quad t \ge 0,%= F\big(x_0,W_{\cdot\wedge t}\big), \quad t \ge 0,
	 \]
	 and such that, for any starting random value $X_0\!\in L^p(\P)$, $p>0$,  $F\big(X_0,(W_t)_{t\in [0,T]}\big)$ solves the path-dependent stochastic Volterra equation~\eqref{eq:pthvolterra} starting from $x_0:=X_0\varphi$.  
	 
	 The same holds true for a $K$-integrated Euler schemes denoted   $\bar X$ :  there exists  a measurable functional $\bar F_n: \mathbb{H}\times {\cal C}_0(\mathbb{T}, \bar{\mathbb{H}})\to \mathbb{X}$ such that 
	 \[
	 \bar X = \bar F_n\big(X_0,(W_t)_{t\in [0,T]}\big),
	 \]
	 {(see Proposition~\ref{lem:K-intF} further on for the representation of the $K$-integrated scheme, the continuity of this scheme, being established in Proposition~\ref{propXtilde}.)} 
	 
	 \smallskip
	 \noindent This entails that the distribution $\P_{(X,\bar X)}$ on $\mathbb{X}^2$   of
	 $(X,\bar X)= \big(F(X_0,W), \bar F_n(X_0,W)\big)$ satisfies
	 \[
	 \P_{(X,\bar X)}(dx, d\bar x) = \int_{\mathbb{H}} \P_{X_0}(d\xi) \P_{(X^{\xi},\bar X^{\xi})}(dx, d \bar x).
	 \]
	 Consequently, using   the monotonicity of probabilistic $L^r(\P)$-norms and pseudo-norms (in the third line) and the elementary inequality $(u+v)^{p}\le 2^{(p-1)^+} (u^{p}+v^{p})$, for $u$, $v \ge 0$, we derive 
	 \begin{align*}
	 	\|\Phi(X,\bar X)\|_p^p = \E \left[|\Phi(X^{X_0},\bar X^{X_0})|^p\right] &= \E_{\P_{X_0}} \left[\,\E_{\P_{(X^{\xi},\bar X^{\xi})}}\left[|\Phi(X^{\xi},\bar X^{\xi})|^p\right]\, | X_0 = \xi\right] \\
	 	&= \int_{\mathbb{H}}\P_{X_0}(d\xi)\E\left[|\Phi(X^{\xi},\bar X^{\xi})|^p\right]\le  \int_{\mathbb{H}}\P_{X_0}(d\xi)\Big(\E\left[|\Phi(X^{\xi},\bar X^{\xi})|^{\bar p}\right]\Big)^{\frac{p}{\bar p}} \\
	 	&\le  \int_{\mathbb{H}}\P_{X_0}(d\xi)\left(C^{\bar p}\Big( 1 + \| \xi\varphi \|_{T} \Big)^{\bar p}\right)^{p/\bar p}\le  C^p \int_{\mathbb{H}} \P_{X_0}(d\xi)\left( 1 + \| \xi\varphi \|_{T} \right)^p \\
	 	&\le  C^p  2^{(p-1)^+}\int_{\mathbb{H}} \P_{X_0}(d\xi)\left(1+\| \xi \varphi \|_{T}^p\right)   = 2^{(p-1)^+} C^p (1+ \E\,\left[ \| x_0 \|_{T}^p\right])
	 \end{align*} 
	 We then have finally:
	 \begin{equation}\label{boundgap}	
	 	\|\Phi(X,\bar X)\|_p\le 2^{(1-\frac1p)^+} C 2^{(\frac1p-1)^+}C\left( 1 + \big\| \,\| x_0  \|_{T} \,\big\|_p\right) = 2^{|1-\frac1p|}C\left( 1 + \big\| \,\| x_0  \|_{T} \,\big\|_p\right).
	 \end{equation}
	 This end the proof and we are done. \hfill$\Box$ 
\end{document}